\documentclass[12pt]{article}
\usepackage{amsmath,mathrsfs}
\usepackage{amsfonts}
\usepackage{amssymb}
\usepackage{amscd}
\usepackage{amsthm}
\usepackage{latexsym}
\usepackage{amsbsy}
\usepackage{amsfonts, amsmath, amssymb, amsgen, amsthm, amscd,latexsym}
\usepackage[all]{xy}
\usepackage{amsxtra}

\def\Diff{\mathop{\rm Diff}\nolimits}
\def\End{\mathop{\rm End}\nolimits}

\def\Id{\mathop{\rm Id}\nolimits}

\def\Hom{\mathop{\rm Hom}\nolimits}
\def\Tot{\mathop{\rm Tot}\nolimits}

\def\Cb{{\mathbb C}}

\def\Rb{{\mathbb R}}

\def\Zb{{\mathbb Z}}

\def\Fc{{\cal F}}

\def\Hc{{\cal H}}

\def\Lc{{\cal L}}
\def\Mc{{\cal M}}

\def\Uc{{\cal U}}
\def\Vc{{\cal V}}

\def\Lc{{\cal L}}

\def\a{\alpha}
\def\b{\beta}
\def\d{\delta}
\def\D{\Delta}

\def\s{\sigma}

\def\t{\theta}
\def\z{\zeta}
\def\ve{\varepsilon}

\def\x{\xi}

\def\0b{\bf 0}

 \def\Zb{\mathbb{Z}}

\def\ot{\otimes}

\def\ra{\rightarrow}

\def\rt{\triangleright}

\def\lt{\triangleleft}
\def\cl{\blacktriangleright\hspace{-4pt} < }
\def\al{>\hspace{-4pt}\vartriangleleft}

\def\acl{\blacktriangleright\hspace{-4pt}\vartriangleleft }

\def\bi{\bowtie}
\def\hd{\overset{\ra}{\partial}}
\def \vd{\uparrow\hspace{-4pt}\partial}

\def \vs{\uparrow\hspace{-4pt}\sigma}
\def\hta{\overset{\ra}{\tau}}
\def \vta{\uparrow\hspace{-4pt}\tau}
\def\hb{\overset{\ra}{b}}
\def \vb{\uparrow\hspace{-4pt}b}

\def\hB{\overset{\ra}{B}}
\def \vB{\uparrow\hspace{-4pt}B}
\def\p{\partial}

\def\0D{\Delta^{(0)}}
\def\1D{\Delta^{(1)}}
\def\Db{\blacktriangledown}

\newcommand{\Fa}{\mathfrak{a}}
\newcommand{\Fg}{\mathfrak{g}}
\newcommand{\Fh}{\mathfrak{h}}

\newcommand{\Fl}{\mathfrak{l}}

\newcommand{\Fn}{\mathfrak{n}}

\newcommand{\FZ}{\mathfrak{Z}}

\newcommand{\wbar}[1]{\overline{#1}}

 \newcommand{\ie}{{\it i.e., }\ }

\newtheorem{theorem}{Theorem}[section]
\newtheorem{remark}[theorem]{Remark}
\newtheorem{proposition}[theorem]{Proposition}
\newtheorem{lemma}[theorem]{Lemma}
\newtheorem{corollary}[theorem]{Corollary}

\newtheorem{example}[theorem]{Example}
\newtheorem{definition}[theorem]{Definition}

\def\build#1_#2^#3{\mathrel{
\mathop{\kern 0pt#1}\limits_{#2}^{#3}}}
\newcommand{\ps}[1]{~\hspace{-4pt}_{_{(#1)}}}
\newcommand{\pr}[1]{~\hspace{-4pt}_{_{\{#1\}}}}
\newcommand{\ns}[1]{~\hspace{-4pt}_{_{{\langle#1\rangle}}}}
\newcommand{\sns}[1]{~\hspace{-4pt}_{_{{\langle\overline{#1}\rangle}}}}
\newcommand{\snsb}[1]{~\hspace{-4pt}_{_{{[\overline{#1}]}}}}

\newcommand{\nsb}[1]{~\hspace{-4pt}_{_{[#1]}}}

\def\odots{\ot\dots\ot}

\newcommand{\nm}[1]{{\mid}#1{\mid}}
\def\one{{\bf 1}}
\def\bfR{{\bf R}}

\def\wg{\wedge}

\def\td{\tilde}
\def\cop{{^{\rm cop}}}

\def\b{\beta}
\def\D{\Delta}

\def\d{\delta}

\def\d{\delta}

\def\nin{\noindent}

\numberwithin{equation}{section}
\begin{document}

\title{\bf SAYD modules over Lie-Hopf algebras}
\author{
\begin{tabular}{cc}
Bahram Rangipour \thanks{Department of Mathematics  and   Statistics,
     University of New Brunswick, Fredericton, NB, Canada      Email: bahram@unb.ca, \quad and \quad ssutlu@ unb.ca }
     \quad and  \quad  Serkan S\"utl\"u $~^\ast$
      \end{tabular}}

\maketitle
\abstract{\noindent In this paper a general van Est type  isomorphism is established. The isomorphism is between the Lie algebra cohomology of a bicrossed sum Lie algebra and the Hopf cyclic cohomology of its Hopf algebra. We  first prove  a one to one correspondence between stable-anti-Yetter-Drinfeld (SAYD) modules over the total Lie algebra  and SAYD modules over the associated Hopf algebra. In  contrast to the non-general  case done in our  previous work, here the  van Est isomorphism is found  at the first level of a natural spectral sequence,  rather than at the level of complexes.  It is proved that the Connes-Moscovici Hopf algebras do not admit any finite dimensional  SAYD modules  except the unique one-dimensional one found by Connes-Moscovici in 1998. This is done by extending  our techniques to work with the infinite dimensional  Lie algebra of formal vector fields.  At the end, the one to one  correspondence is applied to construct a highly nontrivial four dimensional  SAYD module over the Schwarzian  Hopf algebra. We then illustrate the whole theory on this example. Finally  explicit  representative  cocycles of the cohomology  classes for this example are calculated.}
\section{Introduction}
Hopf cyclic cohomology was invented by Connes-Moscovici in 1998 \cite{ConnMosc98}.    It is now beyond dispute
that this cohomology is a fundamental tool in  noncommutative geometry.   Admitting coefficients is one of the  most significant properties of this theory \cite{HajaKhalRangSomm04-I,HajaKhalRangSomm04-II,JaraStef}. These coefficients are called  stable-anti-Yetter-Drinfeld (SAYD) modules \cite{HajaKhalRangSomm04-I}.

\medskip

\nin   A ``geometric" Hopf algebra is a Hopf algebra associated to  (Lie) algebraic group or Lie algebra via  certain functors.   Such  Hopf algebras  are defined as representative (smooth) polynomial functions on the object in question  or as the  universal enveloping algebras of the Lie algebra  or even as a bicrossed product of  such Hopf algebras. The latter procedure is called semi-dualization.   The resulting Hopf algebra via semi-dualization is usually neither commutative nor cocommutative \cite{Maji}.

\medskip

\nin The study of SAYD modules over ``geometric" Hopf algebras begins in \cite{RangSutl}, where  we proved that any representation of the Lie algebra induces  a SAYD module over the associated Hopf algebra.  Therefore those  SAYD modules are called induced modules \cite{RangSutl}. We also proved that the Hopf cyclic cohomology of the associated Hopf algebra is isomorphic to the Lie algebra cohomology of the Lie algebra with coefficients in the original representation.

\medskip

\nin In \cite{RangSutl-II},  the notion  of SAYD modules over Lie algebras was defined and studied.  It was  observed that the corresponding cyclic complex has been known with different names for different SAYD modules. As the main example we proved that
 the (truncated)  polynomial algebra  of a Lie algebra is a SAYD module. The corresponding cyclic complex is identified with the (truncated) Weil algebra \cite{RangSutl-II}. In the same paper we identify the category of SAYD modules over the enveloping algebra of a Lie algebra with those on the Lie algebra.

\medskip
\nin Let us recall the main result of \cite{RangSutl-II} as follows.  For an arbitrary Lie algebra $\Fg$,  the comultiplication  of $U(\Fg)$ does not use  the Lie algebra structure of $\Fg$.  This fact  has been discouraged attention in  comodules over $U(\Fg)$.  It is shown that such comodules are in one to one correspondence with the nilpotent modules over the symmetric algebra $S(\Fg^\ast)$ .     Using this fundamental fact we can identify AYD modules over $U(\Fg)$ with modules over the semi-direct product Lie algebra $\widetilde\Fg=\Fg^\ast\al \Fg$. Here  $\Fg^\ast=\Hom(\Fg, \Cb)$ is  considered to be a commutative Lie algebra and to  be acted upon by $\Fg$ via the  coadjoint representation. We show that the notion of comodule over Lie algebras make sense. Furthermore, SAYD modules over Lie algebras  and the cyclic cohomology of a Lie algebra with coefficients in such modules is defined. It is shown that SAYD modules over $U(\Fg)$ and over $\Fg$ have a one-to-one correspondence and their cyclic homologies are identified.

\medskip
\nin Let  $\Fg=\Fg_1\bi\Fg_2$ be a bicrossed sum Lie algebra.  Let us denote $R(\Fg_2)$  and $U(\Fg_1)$ by $\Fc$ and $\Uc$  respectively. Here $R(\Fg_2)$ is the Hopf algebra of all representative functions on $\Fg_2$, and $U(\Fg_1)$ is the universal enveloping algebra of $\Fg_1$.  A module-comodule over $\Hc:=\Fc\acl \Uc$ is naturally a module-comodule over $\Uc$ and comodule-module  over  $\Fc$.
In \cite{RangSutl}, we completely determined those module-comodule  whose $\Uc$-coaction and $\Fc$-action is trivial. It is proved that such a module-comodule  is induced by a module   over $\Fg$ if and only if it is a  YD module over $\Hc$.

\medskip

\nin  Continuing our study in \cite{RangSutl, RangSutl-II}, we completely  determine  SAYD modules over the bicrossed product Hopf algebra $\Hc=\Fc\acl \Uc$. Roughly speaking,  we show that SAYD modules over  $\Hc$ and SAYD module over $\Fg$ are the same.   We then take advantage of a spectral sequence in  \cite{JaraStef} to prove a van Est isomorphism between   the Hopf cyclic cohomology of $\Hc$ with coefficients in $~^\s{M}_\d= M\ot~ ^\s\Cb_\d$  and the Lie algebra cohomology of $\Fg$ relative to a Levi subalgebra with coefficients in a $\Fg$-module $M$.

\medskip

\nin One of the results of this paper is about the SAYD modules over Connes-Moscovici Hopf algebras $\Hc_n$.  We know that    $\Hc_n$ is the bicrossed product Hopf algebra of $\Fc(N)$ and $U(gl_n^{\rm aff})$  \cite{MoscRang09}. However, the group $N$ is not of finite type. So we cannot apply our theory freely on $\Hc_n$. We overcome this problem by carefully analyzing  the SAYD modules over $\Hc_n$ to reduce the case to a finite type problem. As a result,   we prove that $\Hc_n$ has no AYD module except the most natural one,  $\Cb_\d$,  which was found by Connes-Mosocovici in \cite{ConnMosc98}.

\medskip

\nin To illustrate our theory in a nontrivial example we introduce a SAYD module over the Schwarzian Hopf algebra $\Hc_{1\rm S}$ introduced in \cite{ConnMosc98}.  By definition, $\Hc_{1\rm S}$ is a quotient  Hopf algebra of  $\Hc_1$ by the Hopf ideal generated by
$$\d_2-\frac{1}{2}\d_1^2 .$$
Here $\d_i$ are generators of $\Fc(N)$.  So the Hopf algebra $\Hc_{\rm 1S}$ is generated by
$$X,\quad Y,\quad  \d_1$$

\medskip

\nin As we know, $\Hc_{1\rm S}\cop$ is isomorphic to $R(\Cb)\acl U(gl_1^{\rm aff})$. So our theory guarantees that any suitable SAYD module $M$ over $sl_2=gl_1^{\rm aff}\bi \Cb$ will produce a SAYD module $M_\d$ over $\Hc_{1\rm S}\cop$. We take the truncated polynomial algebra $M=S(sl_2)\nsb{2}$ as our candidate.  The resulting 4-dimensional SAYD module $M_\d$ is then generated by

$$ \one, \quad \bfR^X, \quad \bfR^Y,\quad  \bfR^Z,$$

\nin with the $\Hc_{1\rm S}\cop$ action and coaction   defined by
$$
\left.
  \begin{array}{c|ccc}

    \lhd & X & Y & \d_1 \\[.2cm]
    \hline
    &&&\\[-.2cm]
    \one & 0 & 0 & \bfR^Z \\[.1cm]

    \bfR^X & -\bfR^Y & 2\bfR^X & 0 \\[.1cm]
    \bfR^Y & -\bfR^Z & \bfR^Y & 0 \\[.1cm]
    \bfR^Z & 0 & 0 & 0 \\
  \end{array}
\right.\qquad
  \begin{array}{rl}
  &\Db: M_\d \longrightarrow \mathcal{H}_{\rm 1S}\cop \ot M_\d \\[.2cm]
\hline
  &\\[-.2cm]
& \one \mapsto 1 \ot \one + X \ot \bfR^X + Y \ot \bfR^Y \\
& \bfR^X \mapsto 1 \ot \bfR^X \\
& \bfR^Y \mapsto 1 \ot \bfR^Y + \delta_1 \ot \bfR^X \\
& \bfR^Z \mapsto 1 \ot \bfR^Z + \delta_1 \ot \bfR^Y + \frac{1}{2}\delta_1^2 \ot \bfR^X.
  \end{array}
$$

\nin The surprises here are the nontriviality  of the action of $\d_1$  and the appearance of $X$ and $Y$ in the coaction. In other words this is not an induced module \cite{RangSutl}.

\medskip

\nin We illustrate our results in this paper on this example. We then apply the machinery   developed in \cite{MoscRang09} by  Moscovici and one of the  authors  to  prove that the following two cocycles generates  the  Hopf cyclic cohomology of $\Hc_{1\rm S}\cop$ with coefficients in $M_\d$.

\begin{align*}
&c^{\rm odd} = \one \ot \delta_1 + \bfR^Y \ot X + \bfR^X \ot \delta_1X + \bfR^Y \ot \delta_1Y + 2 \bfR^Z \ot Y , \\[.4cm]
& c^{{\rm even}} = \one \ot X \ot Y - \one \ot Y \ot X + \one \ot Y \ot \d_1Y - \bfR^X \ot XY \ot X \\
& - \bfR^X \ot Y^2 \ot \d_1X - \bfR^X \ot Y \ot X^2 + \bfR^Y \ot XY \ot Y + \bfR^Y \ot Y^2 \ot \d_1Y \\
& + \bfR^Y \ot X \ot Y^2 + \bfR^Y \ot Y \ot \d_1Y^2 - \bfR^Y \ot Y \ot X - \bfR^X \ot XY^2 \ot \d_1 \\
& - \frac{1}{3}\bfR^X \ot Y^3 \ot {\d_1}^2 + \frac{1}{3} \bfR^Y \ot Y^3 \ot \d_1 - \frac{1}{4} \bfR^X \ot Y^2 \ot {\d_1}^2 - \frac{1}{2} \bfR^Y \ot Y^2 \ot \d_1.
\end{align*}

\nin As can be seen by the inspection, the expression of the above cocycles cannot be easily found with bare hands.  It is the mentioned machinery in \cite{MoscRang09}  which allows to arrive at this elaborate formulae.

\tableofcontents

\section{Matched pair of Lie algebras and SAYD modules over double crossed sum Lie algebras}
In this section, matched pair of Lie algebras and their bicrossed sum Lie algebras are reviewed. We also recall double crossed product of Hopf algebras from \cite{Maji}. Next, we provide a brief account  of SAYD modules over Lie algebras from \cite{RangSutl-II}. Finally we investigate the relation between  SAYD modules over the double crossed sum Lie algebra of a matched pair of Lie algebras and SAYD modules over the individual Lie algebras.

\subsection{Matched pair of Lie algebras and mutual pair of Hopf algebras}

Let us recall the notion of matched pair of Lie algebras from \cite{Maji}. A pair of Lie algebras  $(\Fg_1, \Fg_2)$  is called a matched pair if there
are linear maps

\begin{equation}\label{g-1-g-2 action}
\alpha: \Fg_2\ot \Fg_1\ra \Fg_2, \quad \alpha_X(\z)=\z\lt X, \quad \beta:\Fg_2\ot \Fg_1\ra \Fg_1, \quad \beta_\z(X)=\z\rt X,
\end{equation}

\nin satisfying the following conditions,

\begin{align}\label{mp-L-1}
&[\z,\x]\rt X=\z\rt(\x\rt X)-\x\rt(\z\rt X),\\\label{mp-L-2}
 & \z\lt[X, Y]=(\z\lt X)\lt Y-(\z\lt Y)\lt X, \\\label{mp-L-3}
 &\z\rt[X, Y]=[\z\rt X, Y]+[X,\z\rt Y] +
(\z\lt X)\rt Y-(\z\lt Y)\rt X, \\\label{mp-L-4}
&[\z,\x]\lt X=[\z\lt X,\x]+[\z,\x\lt X]+ \z\lt(\x\rt X)-\x\lt(\z\rt X).
\end{align}

\nin Given a matched pair of  Lie algebras $(\Fg_1,\Fg_2)$, one defines a double crossed sum Lie algebra
 $\Fg_1\bowtie \Fg_2$. Its underlying vector space is $\Fg_1\oplus\Fg_2$ and its Lie bracket
is defined  by:

\begin{equation}
[X\oplus\z,  Z\oplus\x]=([X, Z]+\z\rt Z-\x\rt X)\oplus ([\z,\x]+\z\lt Z-\x\lt X).
\end{equation}

\nin Both $\Fg_1$ and $\Fg_2$ are Lie subalgebras of $\Fg_1\bowtie\Fg_2$ via obvious inclusions. Conversely,  if for a
Lie algebra $\Fg$ there are  two Lie subalgebras $\Fg_1$ and $\Fg_2$ so that $\Fg=\Fg_1\oplus\Fg_2$ as vector  spaces,  then
 $(\Fg_1, \Fg_2)$  forms a matched pair of Lie algebras and $\Fg\cong \Fg_1\bowtie \Fg_2$ as Lie algebras \cite{Maji}.
  In this case, the actions of $\Fg_1$ on $\Fg_2$ and $\Fg_2$ on $\Fg_1$  for $\z\in \Fg_2$ and $X\in\Fg_1$ are uniquely determined  by

\begin{equation}\label{lie-actions}
[\z,X]=\z\rt X+\z\lt X, \qquad \z\in \Fg_2, \quad X\in\Fg_1.
\end{equation}

\nin Next, we recall the notion of double crossed product Hopf algebra.  Let $(\Uc,\Vc)$ be a pair of Hopf algebras such that $\Vc$ is a
right $\Uc-$module coalgebra and $\Uc$ is left  $\Vc-$module coalgebra. We call them mutual pair if   their actions  satisfy the following  conditions.
\begin{align}\label{mutual-1}
&v\rt(u^1u^2)= (v\ps{1}\rt u^1\ps{1})((v\ps{2}\lt u^1\ps{2})\rt  u^2),\quad  1\lt u=\ve(u),\\ \label{mutual-2}
&(v^1v^2)\lt u= (v^1\lt(v^2\ps{1}\rt u\ps{1}))(v^2\ps{2}\lt u\ps{2}),\quad
v\rt 1=\ve(v),\\\label{mutual-3}
 &\sum v\ps{1}\lt u\ps{1}\ot v\ps{2}\rt u\ps{2}=\sum v\ps{2}\lt u\ps{2}\ot v\ps{1}\rt u\ps{1}.
\end{align}

\nin Having a  mutual   pair of Hopf algebras, one constructs the double crossed product Hopf algebra $\Uc \bowtie \Vc$. As a coalgebra, $\Uc \bowtie \Vc$ is  $\Uc\ot\Vc $. However, its algebra structure is defined by the rule
\begin{equation}
(u^1 \bowtie v^1)(u^2 \bowtie v^2):= u^1(v^1\ps{1}\rt u^2\ps{1}) \bowtie (v^1\ps{2}\lt u^2\ps{2})v^2,
\end{equation}
together with $1 \bowtie 1$ as its unit. The antipode of $\Uc \bowtie \Vc$ is defined by
\begin{equation}
S(u \bowtie v)=(1 \bowtie S(v))(S(u)\bowtie 1)= S(v\ps{1})\rt S(u\ps{1})\bowtie S(v\ps{2})\lt S(u\ps{2}).
\end{equation}

\nin It is shown in \cite{Maji} that if $\Fa=\Fg_1\bowtie\Fg_2$ is a double crossed sum of Lie algebras, then the
enveloping algebras  $(U(\Fg_1),U(\Fg_2))$  becomes a mutual pair of Hopf algebras. Moreover, $U(\Fa)$ and $ U(\Fg_1) \bowtie U(\Fg_2)$ are isomorphic as Hopf algebras.

\medskip

\nin In terms of the inclusions

\begin{equation}
i_1:U(\Fg_1) \to U(\Fg_1 \bowtie \Fg_2) \quad \mbox{ and } \quad i_2:U(\Fg_2) \to U(\Fg_1 \bowtie \Fg_2),
\end{equation}

\nin the Hopf algebra isomorphism mentioned above is

\begin{equation}
\mu\circ (i_1 \ot i_2):U(\Fg_1) \bowtie U(\Fg_2) \to U(\Fa).
\end{equation}

\nin Here $\mu$ is the multiplication on $U(\Fg)$. We easily observe that there is a linear map

\begin{equation}
\Psi:U(\Fg_2) \bowtie U(\Fg_1) \to U(\Fg_1) \bowtie U(\Fg_2),
\end{equation}

\nin satisfying

\begin{equation}
\mu \circ (i_2 \ot i_1) = \mu \circ (i_1 \ot i_2) \circ \Psi\,.
\end{equation}

\nin The mutual actions of $U(\Fg_1)$ and $U(\Fg_2)$ are defined as follows

\begin{equation}
\rhd := (\Id_{U(\Fg_2)} \ot \ve) \circ \Psi \quad \mbox{ and } \quad \lhd := (\ve \ot \Id_{U(\Fg_1)}) \circ \Psi\,.
\end{equation}

\subsection{SAYD modules over double crossed sum Lie algebras}

We first review the Lie algebra coactions and SAYD modules over Lie algebras. To this end, let us first introduce the notion of comodule over a Lie algebra.

\begin{definition}
\cite{RangSutl-II}. A vector space $M$ is a left comodule over a Lie algebra $\Fg$ if there is a map $\Db_{\Fg}: M \ra \Fg \ot M, \quad m \mapsto m\nsb{-1}\ot m\nsb{0}$ such that
\begin{equation}\label{g-comod}
m\nsb{-2}\wg m\nsb{-1}\ot m\nsb{0}=0,
\end{equation}
where $$m\nsb{-2}\ot m\nsb{-1}\ot m\nsb{0}= m\nsb{-1}\ot (m\nsb{0})\nsb{-1}\ot (m\nsb{0})\nsb{0}.$$
\end{definition}

\nin By \cite[Proposition 5.2]{RangSutl-II}, corepresentations of a Lie algebra $\Fg$ are nothing but the representations of the symmetric algebra $S(\Fg^\ast)$. The most natural corepresentation of a Lie algebra $\Fg$, with a basis $\Big\{X_1, \ldots ,X_N\Big\}$ and dual basis $\Big\{\t^1,\ldots ,\t^N\Big\}$, is $M = S(\Fg^\ast)$ via $m \mapsto X_i \ot m\t^i$. This is called the Koszul coaction. The corresponding representation on $M = S(\Fg^\ast)$ coincides with the initial multiplication of the symmetric algebra.

\medskip

\nin Next,  let $\Db_{\Fg}:M \to \Fg \ot M$ be a left $\Fg$-comodule structure on the linear space $M$. If the $\Fg$-coaction is locally conilpotent, \ie for any $m \in M$ there exists $n \in \mathbb{N}$ such that $\Db^n_{\Fg}(m) = 0$, then it is possible to construct a $U(\Fg)$-coaction structure $\Db_{U}:M \to U(\Fg) \ot M$ on $M$, \cite[Proposition 5.7]{RangSutl-II}. Conversely, any comodule over $U(\Fg)$ results a locally conilpotent comodule over $\Fg$  via its composition with the canonical projection $\pi:U(\Fg)\ra \Fg$ as follows:

$$
\xymatrix {
\ar[dr]_{\Db_{\Fg}} M \ar[r]^{\Db_U\;\;\;\;\;\;} & U(\Fg) \ot M \ar[d]^{\pi \ot \Id} \\
& \Fg \ot M
}
$$
\nin We denote the category of locally conilpotent left $\Fg$-comodules by $^{\Fg}\rm{conil}\Mc$, and we have $^{\Fg}\rm{conil}\Mc = \, ^{U(\Fg)}\Mc$, \cite[Proposition 5.8]{RangSutl-II}.

\begin{definition}
\cite{RangSutl-II}. Let $M$ be a right module and left comodule over a Lie algebra $\Fg$. We call $M$ a right-left AYD module over $\Fg$ if
\begin{equation}\label{AYD-condition}
\Db_{\Fg}(m \cdot X) = m\nsb{-1} \ot m\nsb{0} \cdot X + [m\nsb{-1}, X] \ot m\nsb{0}.
\end{equation}
Moreover, $M$ is called stable if
\begin{equation}\label{stability-condition}
m\nsb{0} \cdot m\nsb{-1} = 0.
\end{equation}
\end{definition}

\begin{example}\label{ex-1}\rm{
Let $\Fg$ be a Lie algebra with a basis $\Big\{X_1, \ldots, X_N\Big\}$ and a dual basis $\Big\{\t^1,\ldots,\t^N\Big\}$, and $M = S(\Fg^\ast)$ be the symmetric algebra of $\Fg^\ast$. We consider the following action of $\Fg$ on $S(\Fg^\ast)$:
\begin{equation}\label{action-1}
S(\Fg^\ast) \ot \Fg \ra S(\Fg^\ast), \quad m \ot X \mapsto m \lhd X := -\Lc_X(m) + \d(X)m
\end{equation}
Here, $\Lc:\Fg \ra \End S(\Fg^\ast)$ is the coadjoint representation of $\Fg$ on $S(\Fg^\ast)$ and $\d \in \Fg^\ast$ is the trace of the adjoint representation of the Lie algebra $\Fg$ on itself. Via the action \eqref{action-1} and the Koszul coaction
\begin{equation}\label{coaction-1}
M \ra \Fg \ot M, \quad m \mapsto X_i \ot m\t^i,
\end{equation}
$M = S(\Fg^\ast)$ is a SAYD module over the Lie algebra $\Fg$.
}\end{example}

\begin{example}\label{ex-2}\rm{
Let $\Fg$ be a Lie algebra and $M = S(\Fg^\ast)\nsb{2q}$ be a truncation of the symmetric algebra of $\Fg^\ast$. Then by the action \eqref{action-1} and the coaction \eqref{coaction-1}, $M$ becomes an SAYD module over the Lie algebra $\Fg$. Note that in this case the coaction is locally conilpotent.
}\end{example}

\nin We recall from \cite{HajaKhalRangSomm04-I} the definition of a  right-left  stable-anti-Yetter-Drinfeld module over a Hopf algebra $\Hc$. Let $M$ be a right module and left comodule over a Hopf algebra $\Hc$. We say that it is a stable-anti-Yetter-Drinfeld (SAYD) module over $\Hc$ if
\begin{align}\label{AYD-Hopf}
&\Db(m\cdot h)= S(h\ps{3})m\ns{-1}h\ps{1}\ot m\ns{0}\cdot h\ps{2},\\\label{stability-Hopf}
&m\ns{0} \cdot m\ns{-1}=m,
\end{align}
for any $v\in V$ and $h\in \Hc$.

\nin According to \cite[Proposition 5.10]{RangSutl-II}, AYD modules over a Lie algebra $\Fg$ with locally conilpotent coaction are in one to one correspondence with  AYD modules over the universal enveloping algebra $U(\Fg)$. In this case, while it is possible to carry the $\Fg$-stability to $U(\Fg)$-stability \cite[Lemma 5.11]{RangSutl-II}, the converse is not necessarily true \cite[Example 5.12]{RangSutl-II}.

\medskip

\nin A family of examples of SAYD modules over a Lie algebra $\Fg$ is given by the modules over the Weyl algebra $D(\Fg)$, \cite[Corollary 5.14]{RangSutl-II}. As for finite dimensional examples, it is proven in \cite{RangSutl-II} that there is no non-trivial $sl_2$-coaction that makes a simple two dimensional $sl_2$-module an SAYD module over $sl_2$.

\medskip
\nin Let $(\Fg_1,\Fg_2)$ be a matched pair of Lie algebras, with  $\Fa := \Fg_1 \bowtie \Fg_2$ as their  double crossed sum Lie algebra.   A vector space $M$ is a module over $\Fa$ if and only if it is a module over  $\Fg_1$ and $\Fg_2$, such that

\begin{align}
(m \cdot Y) \cdot X - (m \cdot X) \cdot Y = m \cdot (Y \rhd X) + m \cdot (Y \lhd X)
\end{align}
is satisfied. In the converse argument one considers the $\Fa$ action on $M$ by
\begin{equation}\label{module on doublecrossed sum}
m \cdot (X \oplus Y) = m \cdot X + m \cdot Y.
\end{equation}

\nin For the comodule structures we have the following analogous result. If $M$ is a comodule over $\Fg_1$ and $\Fg_2$  via
\begin{equation}
m \mapsto m\nsb{-1} \ot m\nsb{0} \in \Fg_1 \ot M \quad \text{ and } \quad m \mapsto m\ns{-1} \ot m\ns{0} \in \Fg_2 \ot M,
\end{equation}
\nin then we define the following linear map

\begin{equation}\label{comodule on doublecrossed sum}
m \mapsto m\nsb{-1} \ot m\nsb{0} + m\ns{-1} \ot m\ns{0} \in \Fa \ot M.
\end{equation}

\nin Conversely, if $M$ is a $\Fa$-comodule via  $\Db_\Fa:M\ra \Fa\ot M$,  then we define the  linear maps with the help of projections

\begin{equation}\label{auxy2}
\xymatrix{
\ar[r]^{\Db_{\Fa}} M \ar[dr]_{\Db_{\Fg_1}} & \Fa \ot M \ar[d]^{p_1 \ot \Id} \\
& \Fg_1 \ot M}
\qquad \text{and }\qquad \xymatrix{
\ar[r]^{\Db_{\Fa}} M \ar[dr]_{\Db_{\Fg_2}} & \Fa \ot M \ar[d]^{p_2 \ot \Id} \\
& \Fg_2 \ot M}
\end{equation}

\begin{proposition}\label{proposition-comodule-doublecrossed-sum}
 A vector space $M$ is an $\Fa$-comodule if and only if it is a $\Fg_1$-comodule and $\Fg_2$-comodule such that
\begin{equation}\label{aux-18}
m\nsb{-1} \ot m\nsb{0}\ns{-1} \ot m\nsb{0}\ns{0} = m\ns{0}\nsb{-1} \ot m\ns{-1} \ot m\ns{0}\nsb{0}.
\end{equation}

\end{proposition}

\begin{proof}
Assume first that $M$ is an $\Fa$-comodule. By the $\Fa$-coaction compatibility, we have

\begin{align}\label{aux-20}
\begin{split}
& m\nsb{-2} \wg m\nsb{-1} \ot m\nsb{0} + m\nsb{-1} \wg m\nsb{0}\ns{-1} \ot m\nsb{0}\ns{0} \\
& + m\ns{-1} \wg m\ns{0}\nsb{-1} \ot m\ns{0}\nsb{0} + m\ns{-2} \wg m\ns{-1} \ot m\ns{0} = 0.
\end{split}
\end{align}

\nin Applying the antisymmetrization map $\alpha:\Fa \wg \Fa \to U(\Fa) \ot U(\Fa)$ we get

\begin{align}\label{aux-19}
\begin{split}
& (m\nsb{-1} \ot 1) \ot (1 \ot m\nsb{0}\ns{-1}) \ot m\nsb{0}\ns{0} - (1 \ot m\nsb{0}\ns{-1}) \ot (m\nsb{-1} \ot 1) \ot m\nsb{0}\ns{0} \\
& + (1 \ot m\ns{-1}) \ot (m\ns{0}\nsb{-1} \ot 1) \ot m\ns{0}\nsb{0} - (m\ns{0}\nsb{-1} \ot 1) \ot (1 \ot m\ns{-1}) \ot m\ns{0}\nsb{0} = 0.
\end{split}
\end{align}

\nin Finally, applying $\Id \ot \ve_{U(\Fg_1)} \ot \ve_{U(\Fg_2)} \ot \Id$ on the both hand sides of the above equation to get the  equation \eqref{aux-18}.

\medskip
\nin Let $m \mapsto m\pr{-1} \ot m\pr{0} \in \Fa \ot M$ denote the $\Fa$-coaction on $M$. Also let $p_1:\Fa \to \Fg_1$ and $p_2:\Fa \to \Fg_2$ be the projections onto the subalgebras $\Fg_1$ and $\Fg_2$ respectively. Then the $\Fa$-coaction is

\begin{equation}
m\pr{-1} \ot m\pr{0} = p_1(m\pr{-1}) \ot m\pr{0} + p_2(m\pr{-1}) \ot m\pr{0}.
\end{equation}

\nin Next, we  shall  prove that
\begin{equation}
m \mapsto p_1(m\pr{-1}) \ot m\pr{0} \in \Fg_1 \ot M \quad \text{ and } \quad m \mapsto p_2(m\pr{-1}) \ot m\pr{0} \in \Fg_2 \ot M
\end{equation}

\nin are coactions. To this end, we observe that the
\begin{equation}
\alpha(p_1(m\pr{-2}) \wg p_1(m\pr{-1})) \ot m\pr{0} = (p_1 \ot p_1)(\alpha(m\pr{-2} \wg m\pr{-1})) \ot m\pr{0} = 0,
\end{equation}
\nin for $M$ is an $\Fa$-comodule.

\medskip
\nin Since the antisymmetrization map $\alpha:\Fg_1 \wg \Fg_1 \to U(\Fg_1) \ot U(\Fg_1)$ is injective, we have

\begin{equation}
p_1(m\pr{-2}) \wg p_1(m\pr{-1}) \ot m\pr{0} = 0,
\end{equation}

\nin proving that $m \mapsto p_1(m\pr{-1}) \ot m\pr{0}$ is a $\Fg_1$-coaction. Similarly $m \mapsto p_2(m\pr{-1}) \ot m\pr{0}$ is a $\Fg_2$-coaction on $M$.

\medskip
\nin Conversely, assume that $M$ is a $\Fg_1$-comodule and $\Fg_2$-comodule such that the compatibility \eqref{aux-18} is satisfied. Then obviously \eqref{aux-20} is true, which is the $\Fa$-comodule compatibility for the coaction \eqref{comodule on doublecrossed sum}.
\end{proof}

\nin We proceed by  investigating the relations between AYD modules over the Lie algebras $\Fg_1$ and $\Fg_2$, and AYD modules over the double crossed sum Lie algebra $\Fa = \Fg_1 \bowtie \Fg_2$.

\begin{proposition}\label{auxx-2}
Let $(\Fg_1,\Fg_2)$ be a matched pair of Lie algebras, $\Fa = \Fg_1 \bowtie \Fg_2$,  and $M \in \, ^{\Fa}\rm{conil}\Mc_{\Fa}$. Then, $M$ is an AYD module over $\Fa$ if and  only if $M$ is an AYD module over $\Fg_1$ and $\Fg_2$, and the following conditions are satisfied
\begin{align}\label{prop-ax2-1}
&(m \cdot X)\ns{-1} \ot (m \cdot X)\ns{0} = m\ns{-1} \lhd X \ot m\ns{0} + m\ns{-1} \ot m\ns{0} \cdot X, \\[.1cm]\label{prop-ax2-2}
& m\ns{-1} \rhd X \ot m\ns{0} = 0, \\[.1cm]\label{prop-ax2-3}
& (m \cdot Y)\nsb{-1} \ot (m \cdot Y)\nsb{0} = - Y \rhd m\nsb{-1} \ot m\nsb{0} + m\nsb{-1} \ot m\nsb{0} \cdot Y, \\[.1cm]\label{prop-ax2-4}
& Y \lhd m\nsb{-1} \ot m\nsb{0} = 0,
\end{align}
for any $X \in \Fg_1$, $Y \in \Fg_2$ and any $m \in M$.
\end{proposition}

\begin{proof}
For $M \in \, ^{\Fa}\rm{conil}\Mc_{\Fa}$, assume that $M$ is an AYD module over the double crossed sum Lie algebra $\Fa$ via the coaction

\begin{equation}
m \mapsto m\pr{-1} \ot m\pr{0} = m\nsb{-1} \ot m\nsb{0} + m\ns{-1} \ot m\ns{0}.
\end{equation}

\nin As the $\Fa$-coaction is locally conilpotent, by \cite[Proposition 5.10]{RangSutl-II} we have $M \in \, ^{U(\Fa)}\mathcal{AYD}_{U(\Fa)}$. Then since the projections
\begin{equation}
\pi_1:U(\Fa) = U(\Fg_1) \bowtie U(\Fg_2) \to U(\Fg_1), \quad \pi_2:U(\Fa) = U(\Fg_1) \bowtie U(\Fg_2) \to U(\Fg_2)
\end{equation}
\nin are coalgebra maps, we conclude that $M$ is a comodule over $U(\Fg_1)$ and $U(\Fg_2)$. Finally, since $U(\Fg_1)$ and $U(\Fg_2)$ are Hopf subalgebras of $U(\Fa)$, AYD conditions on $U(\Fg_1)$ and $U(\Fg_2)$ are immediate, and thus $M$ is an AYD module over $\Fg_1$ and $\Fg_2$.

\medskip
\nin We now prove the compatibility conditions \eqref{prop-ax2-1}, \ldots, \eqref{prop-ax2-4}. To this end, we will make use of the AYD condition for an arbitrary $X \oplus Y \in \Fa$ and $m \in M$. On one hand side we have
\begin{align}\label{aux-21}
\begin{split}
& [m\pr{-1}, X \oplus Y] \ot m\pr{0} + m\pr{-1} \ot m\pr{0} \cdot (X \oplus Y) = \\
& [m\nsb{-1} \oplus 0, X \oplus Y] \ot m\nsb{0} + [0 \oplus m\ns{-1}, X \oplus Y] \ot m\ns{0} \\
& + m\nsb{-1} \ot m\nsb{0} \cdot (X \oplus Y) + m\ns{-1} \ot m\ns{0} \cdot (X \oplus Y) \\
& = ([m\nsb{-1},X] - Y \rhd m\nsb{-1} \oplus -Y \lhd m\nsb{-1}) \ot m\nsb{0} \\
& + (m\ns{-1} \rhd X \oplus [m\ns{-1},Y] + m\ns{-1} \lhd X) \ot m\ns{0} \\
& + (m\nsb{-1} \oplus 0) \ot m\nsb{0} \cdot (X \oplus Y) + (0 \oplus m\ns{-1}) \ot m\ns{0} \cdot (X \oplus Y) \\
& = ((m \cdot X)\nsb{-1} \oplus 0) \ot (m \cdot X)\nsb{0} + (0 \oplus (m \cdot Y)\ns{-1}) \ot (m \cdot Y)\ns{0} \\
& + (- Y \rhd m\nsb{-1} \oplus -Y \lhd m\nsb{-1}) \ot m\nsb{0} + (m\ns{-1} \rhd X \oplus m\ns{-1} \lhd X) \ot m\ns{0} \\
& + (m\nsb{-1} \oplus 0) \ot m\nsb{0} \cdot Y + (0 \oplus m\ns{-1}) \ot m\ns{0} \cdot X.
\end{split}
\end{align}
\nin On the other hand,
\begin{align}\label{aux-22}
\begin{split}
& (m \cdot (X \oplus Y))\pr{-1} \ot (m \cdot (X \oplus Y))\pr{0} = ((m \cdot X)\nsb{-1} \oplus 0) \ot (m \cdot X)\nsb{0} + \\
& ((m \cdot Y)\nsb{-1} \oplus 0) \ot (m \cdot Y)\nsb{0} + (0 \oplus (m \cdot X)\ns{-1}) \ot (m \cdot X)\ns{0} \\
& + (0 \oplus (m \cdot Y)\ns{-1}) \ot (m \cdot Y)\ns{0}.
\end{split}
\end{align}
\nin Since  $M$ is an AYD module over $\Fg_1$ and $\Fg_2$, AYD compatibility \eqref{aux-21} = \eqref{aux-22} translates into
\begin{align}\label{aux-23}
\begin{split}
& ((m \cdot Y)\nsb{-1} \oplus 0) \ot (m \cdot Y)\nsb{0} + (0 \oplus (m \cdot X)\ns{-1}) \ot (m \cdot X)\ns{0} = \\
& (- Y \rhd m\nsb{-1} \oplus -Y \lhd m\nsb{-1}) \ot m\nsb{0} + (m\ns{-1} \rhd X \oplus m\ns{-1} \lhd X) \ot m\ns{0} \\
&  + (m\nsb{-1} \oplus 0) \ot m\nsb{0} \cdot Y + (0 \oplus m\ns{-1}) \ot m\ns{0} \cdot X.
\end{split}
\end{align}

\nin Finally, we set $Y := 0$ to get \eqref{prop-ax2-1} and \eqref{prop-ax2-2}. The equations \eqref{prop-ax2-3} and \eqref{prop-ax2-4} are similarly implied by setting $X:=0$.

\medskip
\nin The converse argument is clear.
\end{proof}

\nin In general, if $M$ is an AYD module over the double crossed sum Lie algebra $\Fa = \Fg_1 \oplus \Fg_2$, then $M$ is not necessarily an AYD module over the Lie algebras $\Fg_1$ and $\Fg_2$.

\begin{example}\label{example-sl2}\rm{
Consider the Lie algebra $sl_2 = \Big\langle X,Y,Z \Big\rangle$,
\begin{equation}
[Y,X] = X, \quad [Z,X] = Y, \quad [Z,Y] = Z.
\end{equation}
\nin Then, $sl_2 = \Fg_1 \bowtie \Fg_2$ for $\Fg_1 = \Big\langle X,Y \Big\rangle$ and $\Fg_2 = \Big\langle Z \Big\rangle$.

\medskip
\nin In view of Example \ref{ex-1}, the symmetric algebra $M = S({sl_2}^\ast)$ is a right-left AYD module over $sl_2$. The module structure  is defined by the coadjoint action, that coincides with \eqref{action-1} since $sl_2$ is unimodular, and comodule structure is given by  the Koszul coaction \eqref{coaction-1}.

\medskip
\nin We now show that it is not an AYD module over $\Fg_1$.  Let $\Big\{\t^X,\t^Y,\t^Z\Big\}$ be a dual basis for $sl_2$. The linear map
\begin{equation}
\Db_{\Fg_1}: M \to \Fg_1 \ot M, \quad m \mapsto X \ot m\t^X + Y \ot m\t^Y,
\end{equation}
which is the projection onto the Lie algebra $\Fg_1$, endows $M$ with a left $\Fg_1$-comodule structure. However,  the AYD compatibility on $\Fg_1$ is not satisfied. Indeed, on  one side we have
\begin{equation}
\Db_{\Fg_1}(m \lhd X) = X \ot (m \lhd X)\t^X + Y \ot (m \lhd X)\t^Y,
\end{equation}
and the other  one we get
\begin{align}
\begin{split}
& [X,X] \ot m\t^X + [Y,X] \ot m\t^Y + X \ot (m\t^X) \lhd X + Y \ot (m\t^Y) \lhd X = \\
& X \ot (m \lhd X)\t^X + Y \ot (m \lhd X)\t^Y - Y \ot m\t^Z.
\end{split}
\end{align}
}\end{example}

\begin{remark}{\rm
Assume that the mutual actions of $\Fg_1$ and $\Fg_2$ are trivial. In this case, if $M$ is an AYD module over $\Fg_1 \bowtie \Fg_2$, then it is an AYD module over $\Fg_1$ and $\Fg_2$.

\medskip
\nin To see it, let us apply $p_1\ot \Id_M$  on the both hand sides of the AYD condition \eqref{AYD-condition}  for  $X \oplus 0 \in \Fa$, where  $p_1:\Fa \to \Fg_1$ is the obvious projection. That is
\begin{align}\label{auxy1}
\begin{split}
& p_1([m\pr{-1}, X \oplus 0]) \ot m\pr{0} + p_1(m\pr{-1}) \ot m\pr{0} \cdot (X \oplus 0) \\
& = p_1((m \cdot (X \oplus 0))\pr{-1}) \ot (m \cdot (X \oplus 0))\pr{0}.
\end{split}
\end{align}
Since in this case the projection $p_1:\Fa \to \Fg_1$ is a map of Lie algebras, the equation \eqref{auxy1} reads
\begin{align}
\begin{split}
& [p_1(m\pr{-1}), X] \ot m\pr{0} + p_1(m\pr{-1}) \ot m\pr{0} \cdot X = p_1((m \cdot X)\pr{-1}) \ot (m \cdot X)\pr{0},
\end{split}
\end{align}
which is the AYD compatibility for the $\Fg_1$-coaction.
Similarly, one proves  that $M$ is an  AYD module over the Lie algebra $\Fg_2$.
}\end{remark}

\nin Let $\Fa = \Fg_1 \bowtie \Fg_2$ be a double crossed sum  Lie algebra and $M$ be an SAYD module over $\Fa$.  By the next example we show that $M$ is not necessarily   stable over  $\Fg_1$ and $\Fg_2$.

\begin{example}{\rm
Consider the Lie algebra $\Fa = gl_2 = \Big\langle Y^1_1, Y^1_2, Y^2_1, Y^2_2 \Big\rangle$ with a dual basis $\Big\{\t^1_1, \t^2_1, \t^1_2, \t^2_2\Big\}$.

\medskip
\nin We have a decomposition $gl_2 = \Fg_1 \bowtie \Fg_2$, where $\Fg_1 = \Big\langle Y^1_1, Y^1_2 \Big\rangle$ and $\Fg_2 = \Big\langle Y^2_1, Y^2_2 \Big\rangle$. Let  $M := S({gl_2}^\ast)$ be the symmetric algebra as an SAYD module over $gl_2$ with the action \eqref{action-1} and the Koszul coaction \eqref{coaction-1} as in  Example \ref{ex-1}.  Then the $\Fg_1$-coaction on $M$ becomes
\begin{equation}
m \mapsto m\ns{-1} \ot m\ns{0} = Y^1_1 \ot m\t^1_1 + Y^1_2 \ot m\t^2_1.
\end{equation}
Accordingly, since $\d(Y^1_1) = 0 = \d(Y^1_2)$ we have
\begin{equation}
{\t^2_1}\ns{0} \lhd {\t^2_1}\ns{-1} = - \Lc_{Y^1_1}\t^1_1 - \Lc_{Y^1_2}\t^2_1 = - \t^2_1\t^1_1 \neq 0.
\end{equation}
}\end{example}

\nin We know that if a comodule   over a  Lie algebra $\Fg$  is locally conilpotent then it can be  lifted to a comodule  over  $U(\Fg)$. In the rest of this section, we are interested  in translating  Proposition \ref{auxx-2} in terms of AYD modules over universal enveloping algebras.

\begin{proposition}\label{auxx-1}
Let $\Fa = \Fg_1 \bowtie \Fg_2$ be a double crossed sum Lie algebra and $M$ be a left comodule over $\Fa$. Then $\Fa$-coaction  is locally conilpotent if and only if the corresponding $\Fg_1$-coaction and $\Fg_2$-coaction are locally conilpotent.
\end{proposition}

\begin{proof}
By \eqref{auxy2} we know that $\Db_{\Fa} = \Db_{\Fg_1} + \Db_{\Fg_2}$. Therefore,
\begin{equation}
\Db_{\Fa}^2(m) = m\nsb{-2} \ot m\nsb{-1} \ot m\nsb{0} + m\ns{-2} \ot m\ns{-1} \ot m\ns{0} + m\nsb{-1} \ot m\nsb{0}\ns{-1} \ot m\nsb{0}\ns{0}.
\end{equation}
By induction we assume that
\begin{align}
\begin{split}
& \Db_{\Fa}^k(m) = m\nsb{-k} \ot \ldots \ot m\nsb{-1} \ot m\nsb{0} + m\ns{-k} \ot \ldots \ot m\ns{-1} \ot m\ns{0} \\
& + \sum_{p + q = k} m\nsb{-p} \ot \ldots \ot m\nsb{-1} \ot m\nsb{0}\ns{-q} \ot \ldots \ot m\nsb{0}\ns{-1} \ot m\nsb{0}\ns{0},
\end{split}
\end{align}
and we apply the coaction one more times to get
\begin{align}
\begin{split}
& \Db_{\Fa}^{k+1}(m) = m\nsb{-k-1} \ot \ldots \ot m\nsb{-1} \ot m\nsb{0} + m\ns{-k-1} \ot \ldots \ot m\ns{-1} \ot m\ns{0} \\
& + \sum_{p + q = k} m\nsb{-p} \ot \ldots \ot m\nsb{-1} \ot m\nsb{0}\ns{-q} \ot \ldots \ot m\nsb{0}\ns{-1} \ot m\nsb{0}\ns{0}\nsb{-1} \ot m\nsb{0}\ns{0}\nsb{0} \\
& + \sum_{p + q = k} m\nsb{-p} \ot \ldots \ot m\nsb{-1} \ot m\nsb{0}\ns{-q-1} \ot \ldots \ot m\nsb{0}\ns{-1} \ot m\nsb{0}\ns{0} \\
& = m\nsb{-k-1} \ot \ldots \ot m\nsb{-1} \ot m\nsb{0} + m\ns{-k-1} \ot \ldots \ot m\ns{-1} \ot m\ns{0} \\
& + \sum_{p + q = k} m\nsb{-p-1} \ot \ldots \ot m\nsb{-1} \ot m\nsb{0}\ns{-q} \ot \ldots \ot m\nsb{0}\ns{-1} \ot m\nsb{0}\ns{0} \ot m\nsb{0}\ns{0} \\
& + \sum_{p + q = k} m\nsb{-p} \ot \ldots \ot m\nsb{-1} \ot m\nsb{0}\ns{-q-1} \ot \ldots \ot m\nsb{0}\ns{-1} \ot m\nsb{0}\ns{0}.
\end{split}
\end{align}
On the second equality we used \eqref{aux-18}. This result immediately implies the claim.
\end{proof}

\nin Let $M$ be a locally conilpotent comodule over $\Fg_1$ and $\Fg_2$. We denote by
\begin{equation}
M \to U(\Fg_1) \ot M, \quad  m \mapsto m\snsb{-1} \ot m\snsb{0}
\end{equation}
the lift  of the $\Fg_1$-coaction and similarly by
\begin{equation}
M \to U(\Fg_2) \ot M, \quad  m \mapsto m\sns{-1} \ot m\sns{0}
\end{equation}
the lift of the $\Fg_2$-coaction.

\begin{corollary}\label{aux-58}
Let $\Fa = \Fg_1 \bowtie \Fg_2$ be a double crossed sum Lie algebra and $M \in \, ^{\Fa}\rm{conil}\Mc_{\Fa}$. Then the $\Fa$-coaction
lifts  to the $U(\Fa)$-coaction
\begin{equation}
m \mapsto m\snsb{-1} \ot m\snsb{0}\sns{-1} \ot m\snsb{0}\sns{0} \in U(\Fg_1) \bowtie U(\Fg_2) \ot M.
\end{equation}
\end{corollary}

\begin{proposition}\label{aux-24}
Let $\Fa = \Fg_1 \bowtie \Fg_2$ be a double crossed sum Lie algebra and $M \in \, ^{\Fa}\rm{conil}\Mc_{\Fa}$. Then $M$ is a  AYD module  over $\Fa$ if and only if $M$ is a AYD module over $\Fg_1$ and $\Fg_2$, and the following conditions are satisfied for any $m \in M$, any $u \in U(\Fg_1)$ and $v \in U(\Fg_2)$.
\begin{align}\label{auxy3}
&(m \cdot u)\sns{-1} \ot (m \cdot u)\sns{0} = m\sns{-1} \lhd u\ps{1} \ot m\sns{0} \cdot u\ps{2}, \\\label{auxy4}
&m\sns{-1} \rhd u \ot m\sns{0} = u \ot m, \\\label{auxy5}
&(m \cdot v)\snsb{-1} \ot (m \cdot v)\snsb{0} = S(v\ps{2}) \rhd m\snsb{-1} \ot m\snsb{0} \cdot v\ps{1}, \\\label{auxy6}
&v \lhd m\snsb{-1} \ot m\snsb{0} = v \ot m.
\end{align}
\end{proposition}

\begin{proof}
Let $M$ be AYD module over $\Fa$. Since the coaction is conilpotent, it lifts to an  AYD module over $U(\Fa)$ by \cite[Proposition 5.7]{RangSutl-II}. We write the AYD condition of Hopf algebras \eqref{AYD-Hopf} for $u \bowtie 1 \in U(\Fa)$,
\begin{align}\label{auxy7}
\begin{split}
& (m \cdot u)\snsb{-1} \ot (m \cdot u)\snsb{0}\sns{-1} \ot (m \cdot u)\snsb{0}\sns{0} = \\
& (S(u\ps{3}) \ot 1)(m\snsb{-1} \ot m\snsb{0}\sns{-1})(u\ps{1} \ot 1) \ot m\snsb{0}\sns{0} \cdot (u\ps{2} \ot 1) = \\
& S(u\ps{4})m\snsb{-1}(m\snsb{0}\sns{-2} \rhd u\ps{1}) \ot m\snsb{0}\sns{-1} \lhd u\ps{2} \ot m\snsb{0}\sns{0} \cdot u\ps{3}.
\end{split}
\end{align}
Applying $\ve \ot \Id \ot \Id$ on the both hand sides of \eqref{auxy7},  we get \eqref{auxy3}.
Similarly we get
\begin{equation}
(m \cdot u)\snsb{-1} \ot (m \cdot u)\snsb{0} = S(u\ps{3})m\snsb{-1}(m\snsb{0}\sns{-1} \rhd u\ps{1}) \ot m\snsb{0}\sns{0} \cdot u\ps{2},
\end{equation}
which yields the following equation after using AYD condition on the left hand side
 \begin{equation}
S(u\ps{3})m\snsb{-1}u\ps{1} \ot m\snsb{0}u\ps{2} = S(u\ps{3})m\snsb{-1}(m\snsb{0}\sns{-1} \rhd u\ps{1}) \ot m\snsb{0}\sns{0} \cdot u\ps{2}.
\end{equation}
This  immediately implies \eqref{auxy4}. Switching to the Lie algebra $\Fg_2$ and writing the AYD condition with a $1 \bowtie v \in U(\Fa)$, we obtain \eqref{auxy5} and  \eqref{auxy6}.

\medskip
\nin Conversely, for $M \in \, ^{\Fa}\rm{conil}\Mc_{\Fa}$ which is also an AYD module over  $\Fg_1$ and $\Fg_2$, assume that \eqref{auxy3},\ldots,\eqref{auxy6} are satisfied. Then  $M$ is an AYD module over $U(\Fg_1)$ and $U(\Fg_2)$. We  show that \eqref{auxy3} and \eqref{auxy4} together imply the AYD condition for the elements of the form $u \bowtie 1 \in U(\Fg_1) \bowtie U(\Fg_2)$. Indeed,
\begin{align}
\begin{split}
& (m \cdot u)\snsb{-1} \ot (m \cdot u)\snsb{0}\sns{-1} \ot (m \cdot u)\snsb{0}\sns{0} = \\
& S(u\ps{3})m\snsb{-1}u\ps{1} \ot (m\snsb{0} \cdot u\ps{2})\sns{-1} \ot (m\snsb{0} \cdot u\ps{2})\sns{0} = \\
& S(u\ps{4})m\snsb{-1}u\ps{1} \ot m\snsb{0}\sns{-1} \lhd u\ps{2} \ot m\snsb{0}\sns{0} \cdot u\ps{3} = \\
& S(u\ps{4})m\snsb{-1}(m\snsb{0}\sns{-2} \rhd u\ps{1}) \ot m\snsb{0}\sns{-1} \lhd u\ps{2} \ot m\snsb{0}\sns{0} \cdot u\ps{3},
\end{split}
\end{align}
where the first equality follows from the AYD condition on $U(\Fg_1)$, the second equality follows from the \eqref{auxy3}, and the last equality is obtained by using \eqref{auxy4}. Similarly, using \eqref{auxy5} and \eqref{auxy6} we prove the AYD condition for  the elements of the form $1 \bowtie v \in U(\Fg_1) \bowtie U(\Fg_2)$. The proof is then complete, since the AYD condition is multiplicative.
\end{proof}

\nin The following generalization of Proposition \ref{aux-24} is now straightforward.

\begin{corollary}\label{auxx-24}
Let $(\Uc,\Vc)$ be a mutual pair of Hopf algebras and $M$ a linear space. Then $M$ is an AYD module over $\Uc\bi \Vc$ if and only if $M$ is an AYD module over $\Uc$ and $\Vc$, and  the following conditions are satisfied for any $m \in M$, any $u \in \Uc$ and $v \in \Vc$.
\begin{align}\label{auxy13}
&(m \cdot u)\sns{-1} \ot (m \cdot u)\sns{0} = m\sns{-1} \lhd u\ps{1} \ot m\sns{0} \cdot u\ps{2}, \\\label{auxy14}
&m\sns{-1} \rhd u \ot m\sns{0} = u \ot m, \\\label{auxy15}
& (m \cdot v)\snsb{-1} \ot (m \cdot v)\snsb{0} = S(v\ps{2}) \rhd m\snsb{-1} \ot m\snsb{0} \cdot v\ps{1}, \\\label{auxy16}
&v \lhd m\snsb{-1} \ot m\snsb{0} = v \ot m,
\end{align}

\end{corollary}

\section{Lie-Hopf algebras and their SAYD modules}
\label{Sec-Lie-Hopf}
In this section we first recall  the associated matched pair of Hopf algebras to a matched pair of Lie algebras from \cite{RangSutl}. We then identify the AYD modules  over the universal enveloping algebra of a double crossed sum Lie algebra with the YD modules over the corresponding bicrossed product Hopf algebra. Finally we prove that the only finite  dimensional SAYD module over the Connes-Moscovici Hopf algebras is the one-dimensional one found in \cite{ConnMosc98}.

\subsection{Lie-Hopf algebras}

Let us first review the bicrossed product construction from \cite{Maji}. Let  $\Uc$ and $\Fc$ be two Hopf algebras.  A linear map $$ \Db:\Uc\ra\Uc\ot \Fc , \qquad \Db u \, = \, u^{\pr{0}} \ot u^{\pr{1}} \, , $$ defines a right
coaction and equips
 $\Uc$ with a  right $\Fc-$comodule coalgebra structure,  if the
following conditions are satisfied for any $u\in \Uc$:
\begin{align}
&u^{\pr{0}}\ps{1}\ot u^{\pr{0}}\ps{2}\ot u^{\pr{1}}= {u\ps{1}}^{\pr{0}}\ot {u\ps{2}}^{\pr{0}}\ot {u\ps{1}}^{\pr{1}}{u\ps{2}}^{\pr{1}}, \quad \ve(u^{\pr{0}})u^{\pr{1}}=\ve(u)1.
\end{align}
We  then  form a  cocrossed product coalgebra $\Fc\cl\Uc$. It has $\Fc\ot \Uc$ as underlying vector space and the coalgebra structure is given by
\begin{align}
&\Delta(f\cl u)= f\ps{1}\cl {u\ps{1}}^{\pr{0}}\ot  f\ps{2}{u\ps{1}}^{\pr{1}}\cl u\ps{2}, \quad \ve(f\cl u)=\ve(f)\ve(u).
\end{align}
In a dual fashion,   $\Fc$ is called a {left $\Uc-$module algebra}, if $\Uc$ acts from the left on $\Fc$ via a left action $$
\rt : \Fc\ot \Uc \ra \Fc
$$ which satisfies the following conditions for any $u\in \Uc$, and $f,g\in \Fc$ :
\begin{align}
&u\rt 1=\ve(u)1, \quad u\rt(fg)=(u\ps{1}\rt f)(u\ps{2}\rt g).
\end{align}
This time  we can endow the underlying vector space $\Fc\ot \Uc$ with
 an algebra structure, to be denoted by $\Fc\al \Uc$, with  $1\al 1$
 as its unit and the product
\begin{equation}
(f\al u)(g\al v)=f \;u\ps{1}\rt g\al u\ps{2}v.
\end{equation}

\nin A pair of Hopf algebras $(\Fc,\Uc)$ is called a matched pair of Hopf algebras if they are equipped, as above,  with an action and a coaction which satisfy the following compatibility conditions
\begin{align}\label{mp1}
& \Delta(u\rt f)={u\ps{1}}^{\pr{0}} \rt f\ps{1}\ot {u\ps{1}}^{\pr{1}}(u\ps{2}\rt f\ps{2}), \quad \ve(u\rt f)=\ve(u)\ve(f) \\ \label{mp4}
& \Db(uv)={u\ps{1}}^{\pr{0}} v^{\pr{0}}\ot {u\ps{1}}^{\pr{1}}(u\ps{2}\rt v^{\pr{1}}), \quad \Db(1)=1 \ot 1 \\
& {u\ps{2}}^{\pr{0}}\ot (u\ps{1}\rt f){u\ps{2}}^{\pr{1}}={u\ps{1}}^{\pr{0}}\ot {u\ps{1}}^{\pr{1}}(u\ps{2}\rt f).
\end{align}
for any $u\in\Uc$, and any $f\in \Fc$. We then form  a new Hopf algebra $\Fc\acl \Uc$, called the  bicrossed product of the matched pair  $(\Fc , \Uc)$. It has $\Fc\cl \Uc$ as the underlying coalgebra and $\Fc\al \Uc$ as the underlying algebra. The antipode is given by
\begin{equation}\label{anti}
S(f\acl u)=(1\acl S(u^{\pr{0}}))(S(fu^{\pr{1}})\acl 1) , \qquad f \in \Fc , \, u \in \Uc.
\end{equation}

\nin Next, we recall  Lie-Hopf algebras from \cite{RangSutl}. A Lie-Hopf algebra produces a bicrossed product Hopf algebra such that  one of the Hopf algebras involved is commutative and the other one is the universal enveloping algebra of a Lie algebra.

\medskip
\nin Let $\Fc$ be a commutative Hopf algebra on which a Lie algebra  $\Fg$ acts  by derivations. Then the vector space $\Fg \ot \Fc$ endowed with the bracket
\begin{equation}\label{bracket}
[X\ot f, Y\ot g]= [X,Y]\ot fg+ Y\ot \ve(f)X\rt g- X\ot \ve(g) Y\rt f
\end{equation}
becomes a Lie algebra. Next, we assume that $\Fc$ coacts on $\Fg$ via $\Db_\Fg:\Fg\ra \Fg\ot \Fc$. We say that the coaction $\Db_\Fg:\Fg\ra \Fg\ot \Fc$ satisfies the structure identity of $\Fg$  if $\Db_\Fg:\Fg\ra \Fg\ot \Fc$ is a Lie algebra map. Finally one uses the  action of $\Fg$ on $\Fc$ and the coaction of  $\Fc$  on $\Fg$ to define the following useful action of $\Fg$ on $\Fc\ot \Fc$:
\begin{equation}
X\bullet (f^1 \ot f^2)= \sum X^{\pr{0}}\rt f^1\ot X^{\pr{1}} f^2 + f^1\ot X\rt f^2.
\end{equation}
We are now ready to define the notion of Lie-Hopf algebra.

\begin{definition}\label{def-Lie-Hopf}
\cite{RangSutl}. Let a Lie algebra $\Fg$ act on a commutative Hopf algebra $\Fc$ by derivations. We say that $\Fc$ is a $\Fg$-Hopf algebra if
 \begin{enumerate}
   \item $\Fc$ coacts on $\Fg$ and its coaction satisfies the structure identity of $\Fg$.
   \item $\D$ and $\ve$ are $\Fg$-linear, that is   $\D(X\rt f)=X\bullet\D(f)$, \quad $\ve(X\rt f)=0$, \quad $f\in \Fc$ and $X\in \Fg$.
    \end{enumerate}
\end{definition}

\nin If $\Fc$ is a  $\Fg$-Hopf algebra, then  $U(\Fg)$ acts on $\Fc$ naturally and makes it a $U(\Fg)$-module algebra. On the other hand, we extend the coaction $\Db_\Fg$ of $\Fc$ on $\Fg$ to a coaction $\Db_{U}$ of $\Fc$  on $U(\Fg)$ inductively via the rule \eqref{mp4}.

\medskip
\nin As for the corresponding bicrossed product Hopf algebra, we have the following result.

\begin{theorem}\label{Theorem-Lie-Hopf-matched-pair}
\cite{RangSutl}. Let $\Fc$ be a commutative Hopf algebra and $\Fg$ be a Lie algebra. Then the pair $(\Fc,U(\Fg))$ is a matched pair of Hopf algebras if and only if $\Fc$ is a $\Fg$-Hopf algebra.
\end{theorem}

\nin A class of examples of Lie-Hopf algebras arises from  matched pairs of Lie algebras. To be able to express such an example, let us recall first the definition of $R(\Fg)$, the  Hopf algebra of representative functions on a Lie algebra $\Fg$.
\begin{equation}\notag
 R(\Fg)=\Big\{f\in U(\Fg)^\ast \mid \exists \; I\subseteq \ker f \text{ such that }  \; \dim(\ker f)/I < \infty  \Big\}.
\end{equation}
 The finite codimensionality condition  in the definition of $R(\Fg)$ guarantees that  for any $f\in R(\Fg)$  there exist  a  finite number of  functions $f_i',
 f_i''\in R(\Fg)$   such that  for any $u^1,u^2\in U(\Fg)$,
\begin{equation}
f(u^1u^2)=\sum_{i}f_i'(u^1)f_i''(u^2).
\end{equation}
The Hopf algebraic structure of $R(\Fg)$ is summarized by:
\begin{align}
&\mu: R(\Fg)\ot R(\Fg)\ra R(\Fg), && \mu(f\ot g)(u)=f(u\ps{1})g(u\ps{2}),\\ &\eta:\Cb\ra R(\Fg),&& \eta(1)=\ve,\\ &\D:R(\Fg)\ra R(\Fg)\ot R(\Fg),&&\\\notag & \D(f)=\sum_if_i'\ot
f_i'',&& \text{if}\;\; f(u^1u^2)= \sum f_i'(u^1)f_i''(u^2),\\
& S: R(\Fg)\ra R(\Fg), &&S(f)(u)=f(S(u)).
\end{align}

\nin The following proposition produces a family of examples.

\begin{proposition}\label{Proposition-matched-Lie-Hopf-Lie}
\cite{RangSutl}. For any matched pair of Lie algebras $(\Fg_1,\Fg_2)$, the Hopf algebra $R(\Fg_2)$ is a $\Fg_1$-Hopf algebra.
\end{proposition}

\subsection{SAYD modules over Lie-Hopf algebras}

Let us start with a very brief introduction to SAYD modules over Hopf algebras. Let $\Hc$ be a Hopf algebra. By definition, a character $\t: \Hc\ra \Cb$ is an algebra map.
A group-like  $\s\in \Hc$ is the dual object of the character, \ie $\D(\s)=\s\ot \s$. The pair $(\t,\s)$ is called a modular pair in involution \cite{ConnMosc00} if
 \begin{equation}
 \t(\s)=1, \quad \text{and}\quad  S_\t^2=Ad_\s,
 \end{equation}
where  $Ad_\s(h)= \s h\s^{-1}$ and  $S_\d$ is defined by
\begin{equation}
S_\t(h)=\t(h\ps{1})S(h\ps{2}).
\end{equation}
We recall from \cite{HajaKhalRangSomm04-I} the definition of a  right-left  stable-anti-Yetter-Drinfeld module over a Hopf algebra $\Hc$. Let $M$ be a right module and left comodule over a Hopf algebra $\Hc$. We say that it is stable-anti-Yetter-Drinfeld (SAYD) module over $\Hc$ if
\begin{equation}
\Db(m\cdot h)= S(h\ps{3})m\ns{-1}h\ps{1}\ot m\ns{0}\cdot h\ps{2},\qquad  m\ns{0} \cdot m\ns{-1}=v,
\end{equation}
for any $m\in M$ and $h\in \Hc$.
It is shown in \cite{HajaKhalRangSomm04-I} that any modular pair in involution defines a one dimensional SAYD module and all one dimensional SAYD modules come this way.

\medskip
\nin If $M$ is a module over a bicrossed product Hopf algebra $\mathcal{F} \acl \mathcal{U}$, then by the fact that $\mathcal{F}$ and $\mathcal{U}$ are subalgebras of $\mathcal{F} \acl \mathcal{U}$ we can immediately conclude that $M$ is a module on $\mathcal{F}$ and $\mathcal{U}$. More explicitly, we have the following elementary lemma, see \cite[Lemma 3.4]{RangSutl-II}.

\begin{lemma}\label{module on bicrossed product}
Let $(\mathcal{F}, \mathcal{U})$ be a matched pair of Hopf algebras and $M$ a linear space. Then $M$ is a right module over the bicrossed product Hopf algebra $\mathcal{F} \acl \mathcal{U}$ if and only if $M$ is a right module over $\mathcal{F}$ and  a right module over $\mathcal{U}$, such that
\begin{equation}\label{aux-25}
(m \cdot u) \cdot f = (m \cdot (u\ps{1} \rhd f)) \cdot u\ps{2}.
\end{equation}
\end{lemma}
Let $(\Fg_1,\Fg_2)$ be a matched pair of Lie algebras and $M$ be a  module over the double crossed sum $\Fg_1 \bowtie \Fg_2$ such that $\Fg_1 \bowtie \Fg_2$-coaction is locally conilpotent.
Being a right $\Fg_1$-module, $M$ has a right $U(\Fg_1)$-module structure. Similarly, since it is a locally conilpotent left $\Fg_2$-comodule, $M$ is a right $R(\Fg_2)$-module. Then we define
\begin{align}\label{aux-17}
\begin{split}
& M \ot R(\Fg_2) \acl U(\Fg_1) \to M \\
& m \ot (f \acl u) \mapsto (m \cdot f) \cdot u = f(m\sns{-1})m\sns{0} \cdot u.
\end{split}
\end{align}

\begin{corollary}
Let $(\Fg_1,\Fg_2)$ be a matched pair of Lie algebras and $M$ be an AYD module over the double crossed sum $\Fg_1 \bowtie \Fg_2$ such that $\Fg_1 \bowtie \Fg_2$-coaction is locally conilpotent. Then $M$ has a right $R(\Fg_2) \acl U(\Fg_1)$-module structure via \eqref{aux-17}.
\end{corollary}

\begin{proof}
For $f \in R(\Fg_2)$, $u \in U(\Fg_1)$ and $m \in M$, we have
\begin{align}
\begin{split}
& (m \cdot u) \cdot f = f((m \cdot u)\sns{-1})(m \cdot u)\sns{0} = f(m\sns{-1} \lhd u\ps{1})m\sns{0} \cdot u\ps{2} \\
& = (u\ps{1} \rhd f)(m\sns{-1})m\sns{0} \cdot u\ps{2} = (m \cdot (u\ps{1} \rhd f)) \cdot u\ps{2}.
\end{split}
\end{align}
Here in the second equality we used Proposition \ref{aux-24}. So by Lemma \ref{module on bicrossed product}  the proof is complete.
\end{proof}

\nin Let us assume that $M$ is a left comodule over the bicrossed product $\mathcal{F} \acl \mathcal{U}$. Since the projections $\pi_1 := \Id_{\mathcal{F}} \ot \ve_{\mathcal{U}}:\mathcal{F} \acl \mathcal{U} \to \mathcal{F}$ and $\pi_2 := \ve_{\mathcal{F}} \ot \Id_{\mathcal{U}}:\mathcal{F} \acl \mathcal{U} \to \mathcal{U}$ are coalegbra maps, $M$ becomes  a left $\mathcal{F}$-comodule as well as a left $\mathcal{U}$-comodule via $\pi_1$ and $\pi_2$. Denoting these comodule structures by
\begin{equation}
m \mapsto m^{\sns{-1}} \ot m^{\sns{0}} \in \mathcal{F} \ot M \quad \text{and} \quad m \mapsto m\snsb{-1} \ot m\snsb{0} \in \mathcal{U} \ot M,
\end{equation}
we mean the $\mathcal{F} \acl \mathcal{U}$-comodule structure is
\begin{equation}\label{auxy12}
m \mapsto m^{\sns{-1}} \ot m^{\sns{0}}\snsb{-1} \ot m^{\sns{0}}\snsb{0} \in \mathcal{F} \acl \mathcal{U} \ot M.
\end{equation}

\begin{lemma}\label{comodule on bicrossed product}
Let $(\mathcal{F}, \mathcal{U})$ be a matched pair of Hopf algebras and $M$ a linear space. Then $M$ is a left comodule over the bicrossed product Hopf algebra $\mathcal{F} \acl \mathcal{U}$ if and only if it is a left comodule over $\mathcal{F}$ and  a left comodule over $\mathcal{U}$, such that  for any $m \in M$
\begin{equation}\label{aux-27}
(m^{\sns{0}}\snsb{-1})^{\pr{0}} \ot m^{\sns{-1}} \cdot (m^{\sns{0}}\snsb{-1})^{\pr{1}} \ot m^{\sns{0}}\snsb{0} = m\snsb{-1} \ot (m\snsb{0})^{\sns{-1}} \ot (m\snsb{0})^{\sns{0}},
\end{equation}
where $u \mapsto u^{\pr{0}} \ot u^{\pr{1}} \in \mathcal{U} \ot \mathcal{F}$ is the right $\mathcal{F}$-coaction on $\mathcal{U}$.
\end{lemma}

\begin{proof}
Let assume that $M$ is a comodule over the bicrossed product Hopf algebra $\mathcal{F} \acl \mathcal{U}$. Then by the coassociativity of the coaction, we have
\begin{align}\label{aux-28}
\begin{split}
& m^{\sns{-2}} \acl (m^{\sns{0}}\snsb{-2})^{\pr{0}} \ot m^{\sns{-1}} \cdot (m^{\sns{0}}\snsb{-2})^{\pr{1}} \acl m^{\sns{0}}\snsb{-1} \ot m^{\sns{0}}\snsb{0} \\
& = m^{\sns{-1}} \acl m^{\sns{0}}\snsb{-1} \ot (m^{\sns{0}}\snsb{0})^{\sns{-1}} \acl (m^{\sns{0}}\snsb{0})^{\sns{0}}\snsb{-1} \ot (m^{\sns{0}}\snsb{0})^{\sns{0}}\snsb{0}.
\end{split}
\end{align}
By applying  $\ve_{\mathcal{F}} \ot \Id_{\mathcal{U}} \ot \Id_{\mathcal{F}} \ot \ve_{\mathcal{U}} \ot \Id_{M}$ on both hand sides of \eqref{aux-28}, we get
\begin{equation}
(m^{\sns{0}}\snsb{-1})^{\pr{0}} \ot m^{\sns{-1}} \cdot (m^{\sns{0}}\snsb{-1})^{\pr{1}} \ot m^{\sns{0}}\snsb{0} = m\snsb{-1} \ot (m\snsb{0})^{\sns{-1}} \ot (m\snsb{0})^{\sns{0}}.
\end{equation}
Conversely, assume that \eqref{aux-27} holds for any $m \in M$. This results
\begin{align}
\begin{split}
& m^{\sns{-2}} \ot (m^{\sns{0}}\snsb{-1})^{\pr{0}} \ot m^{\sns{-1}} \cdot (m^{\sns{0}}\snsb{-1})^{\pr{1}} \ot m^{\sns{0}}\snsb{0} \\
& = m^{\sns{-1}} \ot m^{\sns{0}}\snsb{-1} \ot (m^{\sns{0}}\snsb{0})^{\sns{-1}} \ot (m^{\sns{0}}\snsb{0})^{\sns{0}},
\end{split}
\end{align}
which implies  \eqref{aux-28} \ie the coassociativity of the $\mathcal{F} \acl \mathcal{U}$-coaction.
\end{proof}

\begin{corollary}
Let $(\Fg_1,\Fg_2)$ be a matched pair of Lie algebras and $M$ be an AYD module over the double crossed sum $\Fg_1 \bowtie \Fg_2$ with locally finite action and locally conilpotent coaction. Then $M$ has a left $R(\Fg_2) \acl U(\Fg_1)$-comodule structure.
\end{corollary}

\begin{proof}
Since $M$ is a locally conilpotent left $\Fg_1$-comodule, it has a left $U(\Fg_1)$-comodule structure. On the other hand, being a locally finite right $\Fg_2$-module, $M$ is a left $R(\Fg_2)$-comodule \cite{Hoch74}.
By Proposition \ref{aux-24} we have
\begin{equation}
(m \cdot v)\snsb{-1} \ot (m \cdot v)\snsb{0} = S(v\ps{2}) \rhd m\snsb{-1} \ot m\snsb{0} \cdot v\ps{1},
\end{equation}
or in other words
\begin{equation}
v\ps{2} \rhd (m \cdot v\ps{1})\snsb{-1} \ot (m \cdot v\ps{1})\snsb{0} = m\snsb{-1} \ot m\snsb{0} \cdot v.
\end{equation}
Using the $R(\Fg_2)$-coaction on $M$ and $R(\Fg_2)$-coaction on $U(\Fg_1)$, we can translate this equality into
\begin{equation}
(m^{\sns{-1}} \cdot (m^{\sns{0}}\snsb{-1})^{\pr{1}})(v) (m^{\sns{0}}\snsb{-1})^{\pr{0}} \ot m^{\sns{0}}\snsb{0} = m\snsb{-1} \ot (m\snsb{0})^{\sns{0}}((m\snsb{0})^{\sns{-1}})(v).
\end{equation}
Finally, by the non-degenerate pairing between $U(\Fg_2)$ and $R(\Fg_2)$ we get
\begin{equation}
(m^{\sns{0}}\snsb{-1})^{\pr{0}} \ot m^{\sns{-1}} \cdot (m^{\sns{0}}\snsb{-1})^{\pr{1}} \ot m^{\sns{0}}\snsb{0} = m\snsb{-1} \ot (m\snsb{0})^{\sns{-1}} \ot (m\snsb{0})^{\sns{0}},
\end{equation}
\ie the $R(\Fg_2) \acl U(\Fg_1)$-coaction compatibility.
\end{proof}

\nin Our next challenge is to identify  the Yetter-Drinfeld modules over Lie-Hopf algebras.

\bigskip
\nin Let us recall that a right module left comodule $M$ over a Hopf algebra $\mathcal{H}$ is called a YD module if
\begin{equation}\label{YD compatibility}
h\ps{2}(m \cdot h\ps{1})\ns{-1} \ot (m \cdot h\ps{1})\ns{0} = m\ns{-1}h\ps{1} \ot m\ns{0} \cdot h\ps{2}
\end{equation}
for any $h \in \mathcal{H}$ and any $m \in M$.

\begin{proposition}\label{aux-41}
Let $(\mathcal{F}, \mathcal{U})$ be a matched pair of Hopf algebras and $M$ be a right module and left comodule over $\mathcal{F} \acl \mathcal{U}$ such that via the corresponding module and comodule structures it becomes a YD-module over $\Uc$. Then $M$ is a YD-module over $\Fc\acl\Uc$ if and only if $M$ is a YD-module over  $\Fc$ via the corresponding module and comodule structures, and the following conditions are satisfied
\begin{align}\label{auxy8}
&(m \cdot f)\snsb{-1} \ot (m \cdot f)\snsb{0} = m\snsb{-1} \ot m\snsb{0} \cdot f, \\\label{auxy9}
&m^{\sns{-1}}f\ps{1} \ot m^{\sns{0}}f\ps{2} = m^{\sns{-1}}(m^{\sns{0}}\snsb{-1} \rhd f\ps{1}) \ot m^{\sns{0}}\snsb{0} \cdot f\ps{2}, \\\label{auxy10}
&m^{\sns{-1}} \ot m^{\sns{0}} \cdot u = (u\ps{1})^{\pr{1}}(u\ps{2} \rhd (m \cdot (u\ps{1})^{\pr{0}})^{\sns{-1}}) \ot (m \cdot (u\ps{1})^{\pr{0}})^{\sns{0}}, \\\label{auxy11}
&m\snsb{-1}u^{\pr{0}} \ot m\snsb{0} \cdot u^{\pr{1}} = m\snsb{-1}u \ot m\snsb{0}.
\end{align}
\end{proposition}

\begin{proof}
First we assume that $M$ is a YD module over $\mathcal{F} \acl \mathcal{U}$. Since $\Fc$ is a Hopf subalgebra of $\Fc \acl \Uc$, $M$ is a YD module over $\Fc$.

\medskip
\nin Next, we prove the compatibilities \eqref{auxy8},\ldots,\eqref{auxy11}. Writing \eqref{YD compatibility} for an arbitrary $f \acl 1 \in \mathcal{F} \acl \mathcal{U}$, we get
\begin{align}
\begin{split}
& (f\ps{2} \acl 1) \cdot ((m \cdot f\ps{1})^{\sns{-1}} \acl (m \cdot f\ps{1})^{\sns{0}}\snsb{-1}) \ot (m \cdot f\ps{1})^{\sns{0}}\snsb{0} \\
& = (m^{\sns{-1}} \acl m^{\sns{0}}\snsb{-1}) \cdot (f\ps{1} \acl 1) \ot m^{\sns{0}}\snsb{0} \cdot f\ps{2}.
\end{split}
\end{align}
Using the YD condition on $\mathcal{F}$ on the left hand side, we get
\begin{align}\label{aux-33}
\begin{split}
& f\ps{2}(m \cdot f\ps{1})^{\sns{-1}} \acl (m \cdot f\ps{1})^{\sns{0}}\snsb{-1} \ot (m \cdot f\ps{1})^{\sns{0}}\snsb{0} \\
& = m^{\sns{-1}}f\ps{1} \acl (m^{\sns{0}}f\ps{2})\snsb{-1} \ot (m^{\sns{0}}f\ps{2})\snsb{0} \\
& = m^{\sns{-1}}(m^{\sns{0}}\snsb{-2} \rhd f\ps{1}) \acl m^{\sns{0}}\snsb{-1} \ot m^{\sns{0}}\snsb{0} \cdot f\ps{2}.
\end{split}
\end{align}
Now we apply $\ve_{\mathcal{F}} \ot \Id_{\mathcal{U}} \ot \Id_M$ on the both hand sides of \eqref{aux-33} to get \eqref{auxy8}.
Similarly we apply $\Id_{\mathcal{F}} \ot \ve_{\mathcal{U}} \ot \Id_M$ to get \eqref{auxy9}.

\medskip
\nin By the same argument,  the YD compatibility of $\Fc\acl \Uc$  for an element of the form $1 \acl u \in \mathcal{F} \acl \mathcal{U}$,  followed by the YD compatibility of $\Uc$ yields \eqref{auxy10} and \eqref{auxy11}.

\medskip
\nin Conversely, assume that $M \in \, ^{\mathcal{F}}\mathcal{YD}_{\mathcal{F}}$ and \eqref{auxy8},\ldots,\eqref{auxy11} are satisfied. We will prove that the YD condition over $\mathcal{F} \acl \mathcal{U}$ holds for  the elements of the forms $f \acl 1 \in \mathcal{F} \acl \mathcal{U}$ and  $1 \acl u \in \mathcal{F} \acl \mathcal{U}$. By \eqref{auxy9}, we have
\begin{align}
\begin{split}
& m^{\sns{-1}}f\ps{1} \acl (m^{\sns{0}}f\ps{2})\snsb{-1} \ot (m^{\sns{0}}f\ps{2})\snsb{0} \\
& = m^{\sns{-1}}(m^{\sns{0}}\snsb{-1} \rhd f\ps{1}) \acl (m^{\sns{0}}\snsb{0} \cdot f\ps{2})\snsb{-1} \ot (m^{\sns{0}}\snsb{0} \cdot f\ps{2})\snsb{0},
\end{split}
\end{align}
which, by  using  \eqref{auxy8}, implies the YD compatibility for the elements of the form $f \acl 1 \in \mathcal{F} \acl \mathcal{U}$.

\medskip
\nin Next, by \eqref{auxy10} we have
\begin{align}
\begin{split}
& (u\ps{1})^{\pr{1}}(u\ps{2} \rhd (m \cdot (u\ps{1})^{\pr{0}})^{\sns{-1}}) \acl u\ps{3}(m \cdot (u\ps{1})^{\pr{0}})^{\sns{0}}\snsb{-1} \ot (m \cdot (u\ps{1})^{\pr{0}})^{\sns{0}}\snsb{0} \\
& = m^{\sns{-1}} \acl u\ps{2}(m^{\sns{0}} \cdot u\ps{1})\snsb{-1} \ot (m^{\sns{0}} \cdot u\ps{1})\snsb{0}, \end{split}
\end{align}
which amounts to the  YD compatibility  for the elements of the form $1 \acl u \in \mathcal{F} \acl \mathcal{U}$ by using YD compatibility over $\mathcal{U}$ and \eqref{auxy11}.

\medskip
\nin Since YD condition is multiplicative, it is then satisfied for any $f \acl u \in \mathcal{F} \acl \mathcal{U}$, and hence we have proved that $M$ is YD module over $\mathcal{F} \acl \mathcal{U}$.
\end{proof}

\begin{proposition}\label{aux-46}
Let $(\Fg_1,\Fg_2)$ be a matched pair of finite dimensional Lie algebras, $M$ an AYD module over the double crossed sum $\Fg_1 \bowtie \Fg_2$ with locally finite action and locally conilpotent coaction. Then, by the action \eqref{aux-17} and the coaction \eqref{auxy12}, $M$ becomes a right-left YD module over $R(\Fg_2) \acl U(\Fg_1)$.
\end{proposition}

\begin{proof}
We prove the proposition  by verifying the conditions of Proposition \ref{aux-41}.
Since $M$ is an AYD module over $\Fg_1 \bowtie \Fg_2$ with a locally conilpotent coaction, it is an AYD module over $U(\Fg_1) \bowtie U(\Fg_2)$. In particular, it is a left comodule over $U(\Fg_1) \bowtie U(\Fg_2)$ with the  following  coaction   as proved in Corollary \ref{aux-58}
\begin{equation}
m \mapsto m\snsb{-1} \bowtie m\snsb{0}\sns{-1} \ot m\snsb{0}\sns{0} \in U(\Fg_1) \bowtie U(\Fg_2).
\end{equation}
By the coassociativity of this coaction, we have
\begin{equation}\label{aux-38}
m\sns{0}\snsb{-1} \ot m\sns{-1} \ot m\sns{0}\snsb{0} = m\snsb{-1} \ot m\snsb{0}\sns{-1} \ot m\snsb{0}\sns{0}.
\end{equation}
Thus, the application of $\Id_{U(\Fg_1)} \ot f \ot \Id_M$  on both hand sides  results \eqref{auxy8}.

\medskip
\nin Using \eqref{auxy6} and \eqref{aux-38} we get
\begin{align}
\begin{split}
& v\ps{2}(m \cdot v\ps{1})\sns{-1} \ot (m \cdot v\ps{1})\sns{0} = (v\ps{2} \lhd (m \cdot v\ps{1})\sns{0}\snsb{-1})(m \cdot v\ps{1})\sns{-1} \ot (m \cdot v\ps{1})\sns{0}\snsb{0} \\
& = (v\ps{2} \lhd (m \cdot v\ps{1})\snsb{-1})(m \cdot v\ps{1})\snsb{0}\sns{-1} \ot (m \cdot v\ps{1})\snsb{0}\sns{0}.
\end{split}
\end{align}
Then applying $f \ot \Id_M$ to both sides and using the non-degenerate pairing between $R(\Fg_2)$ and $U(\Fg_2)$, we conclude \eqref{auxy9}.

\medskip
\nin To verify  \eqref{auxy10}, we use the $U(\Fg_1) \bowtie U(\Fg_2)$-module compatibility on $M$, \ie for any $u \in U(\Fg_1)$, $v \in U(\Fg_2)$ and $m \in M$,
\begin{equation}
(m \cdot v) \cdot u = (m \cdot (v\ps{1} \rhd u\ps{1})) \cdot (v\ps{2} \lhd u\ps{2}).
\end{equation}
Using the non-degenerate pairing between $R(\Fg_2)$ and $U(\Fg_1)$, we rewrite this equality as
\begin{align}
\begin{split}
& m^{\sns{-1}}(v)m^{\sns{0}} \cdot u = (m \cdot (v\ps{1} \rhd u\ps{1}))^{\sns{-1}}(v\ps{2} \lhd u\ps{2})(m \cdot (v\ps{1} \rhd u\ps{1}))^{\sns{0}} \\
& = u\ps{2} \rhd (m \cdot (v\ps{1} \rhd u\ps{1}))^{\sns{-1}}(v\ps{2})(m \cdot (v\ps{1} \rhd u\ps{1}))^{\sns{-1}} \\
& = (u\ps{1})^{\pr{1}}(v\ps{1})(u\ps{2} \rhd (m \cdot (u\ps{1})^{\pr{0}})^{\sns{-1}})(v\ps{2})(m \cdot (u\ps{1})^{\pr{0}})^{\sns{0}} \\
& = [(u\ps{1})^{\pr{1}}(u\ps{2} \rhd (m \cdot (u\ps{1})^{\pr{0}})^{\sns{-1}})](v)(m \cdot (u\ps{1})^{\pr{0}})^{\sns{0}},
\end{split}
\end{align}
which means \eqref{auxy10}.

\medskip
\nin Using the $U(\Fg_1) \bowtie U(\Fg_2)$-coaction compatibility \eqref{aux-38}, together with \eqref{auxy4},  we have
\begin{align}
\begin{split}
& m\snsb{-1}u^{\pr{0}} \ot m\snsb{0} \cdot u^{\pr{1}} = m\snsb{-1}u^{\pr{0}}u^{\pr{1}}(m\snsb{0}\sns{-1}) \ot m\snsb{0}\sns{0} \\
& = m\snsb{-1}(m\snsb{0}\sns{-1} \rhd u) \ot m\snsb{0}\sns{0} = m\sns{0}\snsb{-1}(m\sns{-1} \rhd u) \ot m\sns{0}\snsb{0} \\
& = m\snsb{-1}u \ot m\sns{0},
\end{split}
\end{align}
which is \eqref{auxy11}.
\end{proof}

\nin We are now ready to express the main result of this section.

\begin{theorem}
Let $(\mathcal{F}, \mathcal{U})$ be a matched pair of Hopf algebras such that $\Fc$ is commutative and $\Uc$ is cocommutative, and $\Big\langle\;, \Big\rangle:\mathcal{F} \times \mathcal{V} \to \mathbb{C}$ a non-degenerate Hopf pairing. Then $M$ is an AYD-module over $\Uc\bi\Vc$ if and only if $M$ is a YD-module over $\Fc\acl \Uc$ such that by the corresponding module and comodule structures it is a YD-module over $\Uc$.
\end{theorem}

\begin{proof}
Let $M \in \, ^{\mathcal{F} \acl \mathcal{U}}\mathcal{YD}_{\mathcal{F} \acl \mathcal{U}} \, \cap \, ^{\mathcal{U}}\mathcal{YD}_{\mathcal{U}}$.
We first prove that $M \in \Mc_{\mathcal{U} \bowtie \mathcal{V}}$. By Proposition \ref{aux-41}, we have
\eqref{auxy10}.  Evaluating both sidesof this equality on an arbitrary $v \in \mathcal{V}$ we get
\begin{align}
(m \cdot v) \cdot u = (m \cdot (v\ps{1} \rhd u\ps{1})) \cdot (v\ps{2} \lhd u\ps{2}).
\end{align}
This proves that $M$ is a right module on the double crossed product $\mathcal{U} \bowtie \mathcal{V}$.

\medskip
\nin Next, we show that $M \in \, ^{\mathcal{U} \bowtie \mathcal{V}}\Mc$. This time using \eqref{auxy8}
 and the duality between right $\mathcal{F}$-action and left $\mathcal{V}$-coaction we get
\begin{align}
f(m\sns{-1})m\sns{0}\snsb{-1} \ot m\sns{0}\snsb{0} = f(m\snsb{0}\sns{-1})m\snsb{-1} \ot m\snsb{0}\sns{0}.
\end{align}
Since the pairing is non-degenerate, we conclude that $M$ is a left comodule over $\Uc\bi\Vc$.

\medskip
\nin Finally, we prove that AYD condition over $\mathcal{U} \bowtie \mathcal{V}$ is satisfied by using Corollary \ref{auxx-24}, that is  we show that \eqref{auxy13},\ldots,\eqref{auxy16} are satisfied.

\medskip
\nin Firstly, by considering the Hopf  duality between the $\mathcal{F}$ and  $\mathcal{V}$, the right $\mathcal{F} \acl \mathcal{U}$-module compatibility reads
\begin{equation}
f((m \cdot u)\sns{-1})(m \cdot u)\sns{0} = f(m\sns{-1} \lhd u\ps{1})m\sns{0} \cdot u\ps{2}.
\end{equation}
Hence \eqref{auxy13} holds.

\medskip
\nin Secondly, by  \eqref{auxy11} and the Hopf duality between  $\Fc$ and  $\Vc$, we get
\begin{equation}
m\snsb{-1}u^{\pr{0}}u^{\pr{1}}(m\snsb{0}\sns{-1}) \ot m\snsb{0}\sns{0} = m\snsb{-1}(m\snsb{0}\sns{-1} \rhd u) \ot m\snsb{0}\sns{0} = m\snsb{-1}u \ot m\snsb{0},
\end{equation}
which immediately imply \eqref{auxy14}.

\medskip
\nin Thirdly, evaluating the left $\mathcal{F} \acl \mathcal{U}$-coaction compatibility
\begin{equation}
(m^{\sns{0}}\snsb{-1})^{\pr{0}} \ot m^{\sns{-1}} \cdot (m^{\sns{0}}\snsb{-1})^{\pr{1}} \ot m^{\sns{0}}\snsb{-1} = m\snsb{-1} \ot (m\snsb{0})^{\sns{-1}} \ot (m\snsb{0})^{\sns{0}}
\end{equation}
on an arbitrary $v \in \mathcal{V}$, we get
\begin{equation}
v\ps{2} \rhd (m \cdot v\ps{1})\snsb{-1} \ot (m \cdot v\ps{1})\snsb{0} = m\snsb{-1} \ot m\snsb{0} \cdot v,
\end{equation}
which  immediately  implies \eqref{auxy15}.

\medskip
\nin Finally,  evaluating the left hand side of the evaluation \eqref{auxy9} on an arbitrary $v \in \mathcal{V}$, we get
\begin{align}
\begin{split}
& LHS = f\ps{1}(v\ps{2})(m \cdot v\ps{1}) \cdot f\ps{2} = f\ps{1}(v\ps{2})f\ps{2}((m \cdot v\ps{1})\sns{-1}) (m \cdot v\ps{1})\sns{0} \\
& = f(v\ps{2}(m \cdot v\ps{1})\sns{-1}) (m \cdot v\ps{1})\sns{0} = f(m\sns{-1} \cdot v\ps{1}) m\sns{0} \cdot v\ps{2},
\end{split}
\end{align}
and the the right hand side turns into
\begin{align}
\begin{split}
& RHS = (m \cdot v\ps{1})\snsb{-1} \rhd f\ps{1}(v\ps{2})f\ps{2}((m \cdot v\ps{1})\snsb{0}\sns{-1})(m \cdot v\ps{1})\snsb{0}\sns{0} \\
& = f\ps{1}(v\ps{2} \lhd (m \cdot v\ps{1})\snsb{-1})f\ps{2}((m \cdot v\ps{1})\snsb{0}\sns{-1})(m \cdot v\ps{1})\snsb{0}\sns{0} \\
& = f\ps{1}(v\ps{4} \lhd (m\sns{0} \cdot v\ps{2})\snsb{-1})f\ps{2}(S(v\ps{3})m\sns{-1}v\ps{1})(m\sns{0} \cdot v\ps{2})\snsb{0} \\
& = f([v\ps{4} \lhd (m\sns{0} \cdot v\ps{2})\snsb{-1}]S(v\ps{3})m\sns{-1}v\ps{1})(m\sns{0} \cdot v\ps{2})\snsb{0},
\end{split}
\end{align}
where on the third equality we use \eqref{auxy8}. So we get
\begin{align}
m\sns{-1} \cdot v\ps{1} \ot m\sns{0} \cdot v\ps{2} = [v\ps{4} \lhd (m\sns{0} \cdot v\ps{2})\snsb{-1}]S(v\ps{3})m\sns{-1}v\ps{1} \ot (m\sns{0} \cdot v\ps{2})\snsb{0}.
\end{align}
Using the cocommutativity of $\Vc$, we conclude \eqref{auxy16}.

\bigskip

\nin Conversely, take $M \in \, ^{\mathcal{U} \bowtie \mathcal{V}}\mathcal{AYD}_{\mathcal{U} \bowtie \mathcal{V}}$. Then  $M$ is a  left comodule over $\Fc\acl \Uc$ by \eqref{auxy15} and a right module over $\Fc\acl \Uc$ by  \eqref{auxy13}.  So by  Proposition \ref{aux-41} it suffices to verify \eqref{auxy8},\ldots, \eqref{auxy11}.

\medskip
\nin Indeed, \eqref{auxy8} follows from the coaction compatibility over $\Uc \bowtie \Vc$. The condition \eqref{auxy9} is the consequence of  \eqref{auxy16}. The equation \eqref{auxy10} is obtained from   the module compatibility over $\Uc \bowtie \Vc$. Finally \eqref{auxy11} follows from \eqref{auxy14}.
\end{proof}

\begin{proposition}\label{aux-59}
Let $(\Fg_1,\Fg_2)$ be a matched pair of Lie algebras and $M$ be an AYD module over the double crossed sum $\Fg_1 \bowtie \Fg_2$ with locally finite action and locally conilpotent coaction. Assume also that $M$ is stable over $R(\Fg_2)$ and $U(\Fg_1)$. Then  $M$ is stable over $R(\Fg_2) \acl U(\Fg_1)$.
\end{proposition}

\begin{proof}
For an $m \in M$, using the $U(\Fg_1 \bowtie \Fg_2)$-comodule compatibility \eqref{aux-38}, we get
\begin{align}
\begin{split}
& (m^{\sns{0}})\snsb{0} \cdot (m^{\sns{-1}} \acl (m^{\sns{0}})\snsb{-1}) = ((m^{\sns{0}})\snsb{0} \cdot m^{\sns{-1}}) \cdot (m^{\sns{0}})\snsb{-1} \\
& = (m^{\sns{0}} \cdot m^{\sns{-1}})\snsb{0} \cdot (m^{\sns{0}} \cdot m^{\sns{-1}})\snsb{-1} = m\snsb{0} \cdot m\snsb{-1} = m.
\end{split}
\end{align}
\end{proof}

\subsection{AYD modules over the Connes-Moscovici Hopf algebras}

In this subsection we investigate  the finite dimensional SAYD modules over the Connes-Moscovici Hopf algebras $\Hc_n$. Let us first recall from \cite{MoscRang09} the bicrossed product decomposition of the Connes-Moscovici Hopf algebras.

\medskip
\nin Let $\Diff(\Rb^n)$ denote the group of diffeomorphisms on $\Rb^n$. Via the splitting $\Diff(\Rb^n) = G \cdot N$, where $G$ is the group of affine transformation on $\Rb^n$ and
\begin{equation}
N = \Big\{\psi \in \Diff(\Rb^n) \; \Big| \; \psi(0) = 0, \; \psi'(0) = \Id\Big\},
\end{equation}
we have $\Hc_n = \Fc(N) \acl U(\Fg)$. Elements of the Hopf algebra $\Fc:=\Fc(N)$ are called regular functions. They are the coefficients of the Taylor expansions at $0 \in \Rb^n$ of the elements of the group $N$. Here, $\Fg$ is the Lie algebra of the group $G$ and $\Uc:=U(\Fg)$ is the universal enveloping algebra of $\Fg$.

\medskip
\nin On the other hand,  by \cite{Fuks}  the Lie algebra $\Fa$ of formal vector fields on $\Rb^n$ admits the filtration
\begin{align}
\Fa = \Fl_{-1} \supseteq \Fl_{0} \supseteq \Fl_{1} \supseteq \ldots
\end{align}
with the bracket
\begin{align}
[\Fl_p,\Fl_q] \subseteq \Fl_{p+q}.
\end{align}
Here, the subalgebra $\Fl_k \subseteq \Fa$, $k \geq -1$, consists of the vector fields $\sum f_i\p/\p x^i$ such that $f_1, \ldots, f_n$ belongs to the $(k+1)$st power of the maximal ideal of the ring of formal power series. Then it is immediate to conclude
\begin{equation}
gl_n = \Fl_0/\Fl_1, \quad \Fl_{-1}/\Fl_0 \cong \mathbb{R}^n, \quad\text{and} \quad \Fg_1 = \Fl_{-1}/\Fl_0 \oplus \Fl_0/\Fl_1 \cong {gl_n}^{\mbox{aff}}.
\end{equation}

\nin As a result, setting $\Fn := \Fl_1$, the Lie algebra  $\Fa$ admits the decomposition $\Fa = \Fg \oplus \Fn$, and hence we have a matched pair of Lie algebras $(\Fg,\Fn)$. The Hopf algebra  $\Fc(N)$ is isomorphic with $R(\Fn)$ via the following non-degenerate pairing
\begin{equation}
\Big\langle  \a^i_{ j_1,\ldots,j_p}, Z^{k_1, \ldots, k_q}_l\Big\rangle= \d^p_q\d^i_l\d_{j_1,\ldots,j_p}^{k_1, \ldots, k_q}.
\end{equation}
Here
\begin{equation}
\a^i_{ j_1,\ldots,j_p}(\psi)= {\left.\frac{\p^p}{\p x^{j_p}\ldots \p x^{j_1}}\right|}_{x=0}\psi^i(x),
\end{equation}
and
\begin{equation}
Z^{k_1, \ldots, k_q}_l= x^{k_1}\ldots x^{k_q}\frac{\p}{\p x^l}.
\end{equation}
We refer the reader to \cite{ConnMosc98} for more details on this duality.

\medskip
\nin Let $\d$ be the trace of the adjoint representation of $\Fg$ on itself. Then it is known that $\Cb_\d$ is a SAYD module over the Hopf algebra $\Hc_n$ \cite{ConnMosc98}.

\begin{lemma}\label{lemma-(co)action-trivial}
For any YD module over $\Hc_n$, the action of $\Uc$ and the coaction of $\Fc$ are trivial.
\end{lemma}
\begin{proof}
Let $M$ be a finite dimensional YD module over $\mathcal{H}_n = \Fc \acl \Uc$.
 One uses the same argument as in Proposition \ref{aux-41}  to show that $M$ is a module over $\Fa$. However we know that $\Fa$ has no nontrivial finite dimensional representation  by  \cite{Fuks}. We conclude that the $\Uc$ action and the $\Fc$-coaction on $M$ are trivial.
\end{proof}

\nin Let us introduce the isotropy subalgebra $\Fg_0\subset \Fg$ by
\begin{align}
\Fg_0 = \Big\{X \in \Fg_1\,\Big|\, Y \rhd X = 0, \forall Y \in \Fg_2\Big\} \subseteq \Fg_1.
\end{align}
 By the construction of $\Fa$ it is obvious that $\Fg_0$ is generated by $Z^i_j$. So $\Fg_0\cong gl_n$.
  By the definition of the  coaction $\Db_\Uc:\Uc\ra \Uc\ot \Fc$ we see that $U(\Fg_0)=\Uc^{co\Fc}$.

\begin{lemma}\label{lemma-coaction-lands}
For any finite dimensional YD module $M$ over  $\Hc_n$ the coaction $$\Db:M\ra \Hc_n\ot M$$ lands merely in $U(\Fg_0)\ot M.$
\end{lemma}
\begin{proof}
By Lemma \ref{lemma-(co)action-trivial} we know that $\Uc$-action and $\Fc$-coaction
 on $M$ are trivial. Hence, the left coaction $M \to \Fc\acl \Uc \ot M$ becomes $m \mapsto 1 \acl m\snsb{-1} \ot m\snsb{0}$. The coassociativity of the coaction
\begin{equation}
1 \acl m\snsb{-2} \ot 1 \acl m\snsb{-1} \ot m\snsb{0} = 1 \acl (m\snsb{-2})^{\pr{0}} \ot (m\snsb{-2})^{\pr{1}} \acl m\snsb{-1} \ot m\snsb{0}
\end{equation}
implies that
\begin{equation}
m \mapsto m\snsb{-1} \ot m\snsb{0} \in \Uc^{co\Fc} \ot M =U(\Fg_0) \ot M.
\end{equation}
\end{proof}

\begin{lemma}\label{lemma-coaction-trivial}
Let $M$ be a finite dimensional YD module over the Hopf algebra $\Hc_n$ then the coaction of $\Hc_n$ on $M$ is trivial.
\end{lemma}

\begin{proof}
By  Lemma \ref{lemma-coaction-lands} we know that the coaction of $\Hc_n$ on $M$ lands in $U(\Fg_0)\ot M$. Since $U(\Fg_0)$ is a Hopf subalgebra of $\Hc_n$, it is obvious that $M$ is an AYD module over $U(\Fg_0)$. Since $\Fg_0$ is finite dimensional,  $M$ becomes an AYD module over $\Fg_0$.

\medskip
\nin Let us express  the $\Fg_0$-coaction for an arbitrary basis element $m^i \in M$ as
\begin{align}
m^i \mapsto m^i\nsb{-1} \ot m^i\nsb{0} = \alpha^{ip}_{kq}Z^q_p \ot m^k \in \Fg_0 \ot M.
\end{align}
Then AYD condition over $\Fg_0$ becomes
\begin{align}
\alpha^{ip}_{kq}[Z^q_p,Z] \ot m^k = 0.
\end{align}
Choosing an arbitrary $Z = Z^{p_0}_{q_0} \in gl_n = \Fg_0$, we get
\begin{align}
\alpha^{ip_0}_{kq}Z^q_{q_0} \ot m^k - \alpha^{ip}_{kq_0}Z^{p_0}_p \ot m^k = 0,
\end{align}
or equivalently
\begin{align}
\alpha^{ip_0}_{kq_0}(Z^{q_0}_{q_0} - Z^{p_0}_{p_0}) + \sum_{q \neq q_0}\alpha^{ip_0}_{kq}Z^q_{q_0} - \sum_{p \neq p_0}\alpha^{ip}_{kq_0}Z^{p_0}_p = 0.
\end{align}
Thus for $n \geq 2$ we have proved that the $\Fg_0$-coaction is trivial. Hence its lift  as a  $U(\Fg_0)$-coaction that  we have started with  is trivial. This proves that the $\Uc$ coaction and hence the $\Hc_n$ coaction on $M$ is trivial.

\medskip
\nin On the other hand,  for $n = 1$, the YD condition for $X \in \Hc_1$ reads, in view of the triviality of the action of ${gl_1}^{\rm aff}$,
\begin{equation}
Xm\pr{-1} \ot m\pr{0} + Z(m \cdot \d_1)\pr{-1} \ot (m \cdot \d_1)\pr{0} = m\pr{-1}X \ot m\pr{0}.
\end{equation}
By Lemma \ref{lemma-coaction-lands} we know that the coaction lands in $U({gl_1}^{\rm aff})$. Together with the relation $[Z,X] = X$ this forces the $\Hc_1$-coaction (and also the action) to be trivial.

\end{proof}

\begin{lemma}\label{lemma-action-trivial}
Let $M$ be a finite dimensional YD module over the Hopf algebra $\Hc_n$. Then the action of $\Hc_n$ on $M$ is trivial.
\end{lemma}
\begin{proof}
By Lemma \ref{lemma-(co)action-trivial} we know that the action of $\Hc_n$ on $M$ is concentrated on the action of $\Fc$ on $M$. So it suffices to prove that this action is trivial.

\medskip
\nin For an arbitrary $m \in M$ and $1 \acl X_k \in \mathcal{H}_n$, we write the YD compatibility. First we calculate
\begin{align}
\begin{split}
& \Delta^2(1 \acl X_k) = \\
& (1 \acl 1) \ot (1 \acl 1) \ot (1 \acl X_k) + (1 \acl 1) \ot (1 \acl X_k) \ot (1 \acl 1) \\
& + (1 \acl X_k) \ot (1 \acl 1) \ot (1 \acl 1) + (\delta^p_{qk} \acl 1) \ot (1 \acl Y_p^q) \ot (1 \acl 1) \\
& + (1 \acl 1) \ot (\delta^p_{qk} \acl 1) \ot (1 \acl Y_p^q) + (\delta^p_{qk} \acl 1) \ot (1 \acl 1) \ot (1 \acl Y_p^q).
\end{split}
\end{align}
Since, by  Lemma \ref{lemma-coaction-trivial}, the coaction of $\Hc_n$ on $M$ is trivial, the AYD condition  can be written as
\begin{align}
\begin{split}
& (1 \acl 1) \ot m \cdot X_k = S(1 \acl X_k) \ot m + 1 \acl 1 \ot m \cdot X_k + 1 \acl X_k \ot m + \\
& \delta^p_{qk} \acl 1 \ot m \cdot Y_p^q - 1 \acl Y_p^q \ot m \cdot \delta^p_{qk} - \delta^p_{qk} \acl Y_p^q \ot m = \\
& \delta^p_{qk} \acl 1 \ot m \cdot Y_p^q + Y_i^j \rhd \delta^i_{jk} \acl 1 \ot m - 1 \acl Y_p^q \ot m \cdot \delta^p_{qk}.
\end{split}
\end{align}
Therefore,
\begin{align}
m \cdot \delta^p_{qk} = 0.
\end{align}
Finally, by the module compatibility on a bicrossed product $\Fc \acl \Uc$, we get
\begin{align}
(m \cdot X_l) \cdot \delta^p_{qk} = m \cdot (X_l \rhd \delta^p_{qk}) + (m \cdot \delta^p_{qk}) \cdot X_l,
\end{align}
which in turn, by using one more time the triviality of the $U(\Fg_1)$-action on $M$, implies
\begin{align}
m \cdot \delta^p_{qkl} = 0.
\end{align}
Similarly we have
\begin{align}
m \cdot \delta^p_{qkl_1 \ldots l_s} = 0, \quad \forall s
\end{align}
This proves that the   $\Fc$-action and a posteriori
the    $\Hc_n$ action on $M$ is trivial.
\end{proof}

\nin Now we prove the main result of this section.
\begin{theorem}
The only finite dimensional AYD module over the Connes-Moscovici Hopf algebra $\mathcal{H}_n$  is $\Cb_\d$.
\end{theorem}
\begin{proof}
By Lemma \ref{lemma-action-trivial} and Lemma \ref{lemma-coaction-trivial} we know that the only finite dimensional YD module on $\Hc_n$ is the trivial one. On the other hand, by the result of M. Staic in \cite{Stai} we know that the category of AYD modules and the category of YD modules over a Hopf algebra $H$ are equivalent provided $H$ has a modular pair in involution $(\t,\s)$. In fact the equivalence functor between these  two categories are simply given by
\begin{equation}
^H\mathcal{YD}_H\ni M \longmapsto ~ {^\s}M_\t:=M\ot\;^\s\Cb_\t\in\;\; ^H\mathcal{AYD}_H.
\end{equation}
Since by the result of Connes-Moscovici  in \cite{ConnMosc98} the Hopf algebra $\Hc_n$ admits a modular pair in involution $(\d,1)$,  we conclude that the only finite dimensional AYD module on $\Hc_n$ is $\Cb_\d$.
\end{proof}

\section{Hopf-cyclic cohomology with coefficients}
Thanks to the results in the second section, we know all SAYD modules over a Lie-Hopf algebra $(R(\Fg_2),\Fg_1)$ in terms of AYD modules over the ambient Lie algebra $\Fg_1\bi\Fg_2$.   The next natural question is the Hopf cyclic cohomology of the bicrossed product Hopf algebra with  coefficients in such a module $^\s M_\d$, where $(\d,\s)$ is the natural modular pair in involution associated to $(\Fg_1,\Fg_2)$ and $M$ is a SAYD module over $\Fg_1\bi\Fg_2$ . To answer this question we need to prove a van Est type theorem between Hopf cyclic complex of the Hopf algebra $R(\Fg_2)\acl U(\Fg_1)$  with coefficients $^\s M_\d$ and the relative perturbed Koszul complex of $\Fg_1\bi\Fg_2$ with coefficient $M$ introduced in \cite{RangSutl-II}. Actually we observe that our strategy in \cite{RangSutl} can be improved to include all cases, not only the induced coefficients  introduced in \cite{RangSutl}. The main obstacle  here is the $R(\Fg_2)$ action and  $U(\Fg_1)$ coaction which prevent us from having two trivial  (co)bounadry maps. The first one is the Hochschild coboundary map of $U(\Fg_1)$ and the second one is the Connes boundary map of $R(\Fg_2)$. We observe that the filtration on $^\s M_\d$ originally defined by Jara-Stefan in \cite{JaraStef} is extremely helpful. In the first page of the spectral sequence associated to such filtration  these two (co)boundary vanish and the situation descends to the case of \cite{RangSutl}.

\subsection{Relative Lie algebra cohomology and cyclic cohomology of Hopf algebras}

For a Lie subalgebra $\Fh\subseteq \Fg$ and a right $\Fg$-module $M$ we define the relative cochains by
\begin{equation}
C^q(\Fg,\Fh,M)=\Big\{\a \in C^q(\Fg,M)=\wg^q \Fg^\ast \ot M \Big|\;  \iota(X) \a = \Lc_X(\a) = 0, \quad X\in \Fh\Big\},
\end{equation}
where
\begin{align}
&\iota(X)(\a)(X_1,\ldots,X_{q})=\a(X,X_1,\ldots,X_q), \\
&\Lc_X(\a)(X_1,\ldots,X_{q})=\\\notag
&\sum(-1)^i\a([X,X_i], X_1,\ldots, \widehat{X}_i,\ldots,X_q)+\theta(X_1,\ldots,X_q)X.
\end{align}
We can identify $C^q(\Fg,\Fh,M)$ with $\Hom_{\Fh}(\wg^q(\Fg/\Fh),M)$ which is $(\wedge^q(\Fg/\Fh)^\ast \ot M)^\Fh$, where the action of $\Fh$ on $\Fg/\Fh$ is induced  by the  adjoint action of $\Fh$ on $\Fg$.

\medskip
\nin It is checked in \cite{ChevEile} that the Chevalley-Eilenberg coboundary $d_{\rm CE}: C^q(\Fg,M) \ra C^{q+1}(\Fg,M)$
 \begin{align}
 \begin{split}
& d_{\rm CE}(\a)(X_0, \ldots,X_q)=\sum_{i<j} (-1)^{i+j}\a([X_i,X_j], X_0\ldots \widehat{X}_i, \ldots, \widehat{X}_j, \ldots, X_q)+\\
& ~~~~~~~~~~~~~~~~~~~~~~~\sum_{i}(-1)^{i+1}\a(X_0,\ldots,\widehat{X}_i,\ldots X_q)X_i.
\end{split}
 \end{align}
is well defined on $C^{\bullet}(\Fg,\Fh,M)$. We denote the homology of the complex $(C^\bullet(\Fg,\Fh,M),d_{\rm CE})$ by $H^\bullet(\Fg,\Fh,M)$ and refer  to it as the relative Lie algebra cohomology of $\Fh\subseteq \Fg$ with coefficients in $M$.

\medskip
\nin Next, we recall the perturbed Koszul complex $W(\Fg, M)$ from \cite{RangSutl-II}. Let $M$ be a right $\Fg$-module and $S(\Fg^\ast)$-module satisfying
\begin{equation}
(m \cdot X_j) \cdot \theta^t = m \cdot (X_j \rhd \theta^t) + (m \cdot \theta^t) \cdot X_j.
\end{equation}
Then $M$ is a module over the semi direct sum Lie algebra $\widetilde{\Fg} = \Fg^\ast \al \Fg$.

\medskip
\nin Let $M$ be a module over the Lie algebra $\widetilde\Fg$. Then
\begin{align}
&\text{$M$ is called unimodular stable if }\qquad \sum_k (m \cdot X_k) \cdot \theta^k = 0 , \\\notag
&\text{$M$ is called  stable if }\qquad \sum_k (m \cdot \theta^k) \cdot X_k = 0.
\end{align}
By \cite[Proposition 4.3]{RangSutl-II}, if $M$ is unimodular stable, then $M_{\b}:= M\ot\Cb_{\b}$ is stable over $\Fg$. Here $\b$ is the trace of the adjoint representation of the Lie algebra $\Fg$ on itself.

\medskip
\nin For a unimodular stable right $\widetilde\Fg$-module $M$, the graded space $W^n(\Fg,M) := \wedge^n \Fg^* \ot M$, $n \geq 0$, becomes a mixed complex with Chevalley-Eilenberg coboundary $$d_{\rm CE}: W^n(\Fg,M) \to W^{n+1}(\Fg,M)$$  and
 the Kozsul differential
\begin{align}
\begin{split}
& d_{\rm K}:W^n(\Fg,M) \to W^{n-1}(\Fg,M) \\
& \alpha \ot m \mapsto \sum_i \iota_{X_i}(\alpha) \ot m \lhd \theta^i.
\end{split}
\end{align}
By \cite[Proposition 5.13]{RangSutl-II}, a unimodular stable right $\widetilde\Fg$-module is a right-left unimodular stable AYD module, unimodular SAYD in short, over the Lie algebra $\Fg$.

\medskip
\nin Finally we introduce  the relative perturbed Koszul complex,
\begin{equation}
W(\Fg, \Fh, M) = \Big\{f \in W(\Fg, M)\, \Big|\, \iota(Y)f = 0, \iota(Y)(d_{\rm CE}f) = 0 , \forall Y \in \Fh\Big\}.
\end{equation}
We have the following result.

\begin{lemma}\label{lemma-relative-Koszul}
$d_{\rm K}(W(\Fg, \Fh, M)) \subseteq W(\Fg, \Fh, M)$.
\end{lemma}

\begin{proof}
For any $\a \ot m \in W^{n+1}(\Fg, \Fh, M)$ and any $Y \in \Fh$,
\begin{align}
\begin{split}
& \iota(Y)(d_{\rm K}(\a \ot m)) = \iota(Y)((-1)^n\iota(m\nsb{-1})\a \ot m\nsb{0}) = (-1)^n\iota(Y)\iota(m\nsb{-1})\a \ot m\nsb{0} \\
& = (-1)^{n-1}\iota(m\nsb{-1})\iota(Y)\a \ot m\nsb{0} = d_{\rm K}(\iota(Y)\a \ot m) = 0
\end{split}
\end{align}
Similarly, using $d_{\rm CE} \circ d_{\rm K} + d_{\rm K} \circ d_{\rm CE} = 0$,
\begin{align}
\iota(Y)(d_{\rm CE}(d_{\rm K}(\a \ot m))) = - \iota(Y)(d_{\rm K} \circ d_{\rm CE}(\a \ot m)) = - d_{\rm K}(\iota(Y)d_{\rm CE}(\a \ot m)) = 0.
\end{align}
\end{proof}

\begin{definition}
Let $\Fg$ be a Lie algebra, $\Fh \subseteq \Fg$ be a Lie subalgebra and $M$ be a unimodular SAYD module over $\Fg$. We call the homology of the mixed subcomplex $(W(\Fg, \Fh, M), d_{\rm CE} + d_{\rm K})$ the relative periodic cyclic cohomology of the Lie algebra $\Fg$ relative to the Lie subalgebra $\Fh$ with coefficients in unimodular stable right $\widetilde\Fg$-module $M$. We use the notation $\widetilde{HP}^{\bullet}(\Fg,\Fh,M)$.
\end{definition}

\nin In case of the trivial Lie algebra coaction, this cohomology becomes the relative Lie algebra cohomology.

\medskip
\nin We conclude this subsection by a brief account of cyclic cohomology of Hopf algebras. Let $M$ be a right-left SAYD module over a Hopf algebra $\Hc$. Let
 \begin{equation}
 C^q(\Hc,M):= M\ot \Hc^{\ot q}, \quad q\ge 0.
 \end{equation}

\nin We recall the following operators on $C^{\bullet}(\Hc,M)$
\begin{align*}
&\text{face operators} \quad\p_i: C^q(\Hc,M)\ra C^{q+1}(\Hc,M), && 0\le i\le q+1\\
&\text{degeneracy operators } \quad\s_j: C^q(\Hc,M)\ra C^{q-1}(\Hc,M),&& \quad 0\le j\le q-1\\
&\text{cyclic operators} \quad\tau: C^q(\Hc,M)\ra C^{q}(\Hc,M),&&
\end{align*}
by
\begin{align}
\begin{split}
&\p_0(m\ot h^1\odots h^q)=m\ot 1\ot h^1\odots h^q,\\
&\p_i(m\ot h^1\odots h^q)=m\ot h^1\odots h^i\ps{1}\ot h^i\ps{2}\odots h^q,\\
&\p_{q+1}(m\ot h^1\odots h^q)=m\pr{0}\ot h^1\odots h^q\ot m\pr{-1},\\
&\s_j (m\ot h^1\odots h^q)= (m\ot h^1\odots \ve(h^{j+1})\odots h^q),\\
&\tau(m\ot h^1\odots h^q)=m\pr{0}h^1\ps{1}\ot S(h^1\ps{2})\cdot(h^2\odots h^q\ot m\pr{-1}),
\end{split}
\end{align}
where $\Hc$ acts on $\Hc^{\ot q}$ diagonally.

\medskip
\nin The graded module $C(\Hc,M)$  endowed with the above operators is then a cocyclic module \cite{HajaKhalRangSomm04-II}, which means that $\p_i,$ $\s_j$ and $\tau$ satisfy the following identities
\begin{eqnarray}
\begin{split}
& \p_{j}  \p_{i} = \p_{i} \p_{j-1},   \hspace{35 pt}  \text{ if} \quad\quad i <j,\\
& \sigma_{j}  \sigma_{i} = \sigma_{i} \sigma_{j+1},    \hspace{30 pt}  \text{    if} \quad\quad i \leq j,\\
&\sigma_{j} \p_{i} =   \label{rel1}
 \begin{cases}
\p_{i} \sigma_{j-1},   \quad
 &\text{if} \hspace{18 pt}\quad\text{$i<j$}\\
\text{Id}   \quad\quad\quad
 &\text{if}\hspace{17 pt} \quad   \text{$i=j$ or $i=j+1$}\\
\p_{i-1} \sigma_{j}  \quad
 &\text{    if} \hspace{16 pt}\quad \text{$i>j+1$},\\
\end{cases}\\
&\tau\p_{i}=\p_{i-1} \tau, \hspace{43 pt} 1\le i\le  q\\
&\tau \p_{0} = \p_{q+1}, \hspace{43 pt} \tau \sigma_{i} = \sigma _{i-1} \tau,  \hspace{33 pt} 1\le i\le q\\ \label{rel2}
&\tau \sigma_{0} = \sigma_{n} \tau^2, \hspace{43 pt} \tau^{q+1} = \Id.
\end{split}
\end{eqnarray}

\nin One uses the face operators to define the Hochschild coboundary
\begin{align}
\begin{split}
&b: C^{q}(\Hc,M)\ra C^{q+1}(\Hc,M), \qquad \text{by}\qquad b:=\sum_{i=0}^{q+1}(-1)^i\p_i.
\end{split}
\end{align}
It is known that $b^2=0$. As a result, one obtains the Hochschild complex of the coalgebra $\Hc$ with coefficients in the bicomodule $M$. Here, the right $\Hc$-comodule structure on $M$ defined trivially. The cohomology of $(C^\bullet(\Hc,M),b)$ is denoted by $H^\bullet_{\rm coalg}(H,M)$.

\medskip
\nin One uses the rest of the operators to define the Connes boundary operator,
\begin{align}
\begin{split}
&B: C^{q}(\Hc,M)\ra C^{q-1}(\Hc,M), \qquad \text{by}\qquad B:=\left(\sum_{i=0}^{q}(-1)^{qi}\tau^{i}\right) \s_{q-1}\tau.
\end{split}
\end{align}

\nin It is shown in \cite{Conn83} that for any cocyclic module we have $b^2=B^2=(b+B)^2=0$. As a result, one defines the cyclic cohomology of $\Hc$ with coefficients in the  SAYD module $M$, which is  denoted by $HC^{\bullet}(\Hc,M)$, as the total  cohomology of the bicomplex
\begin{align}
C^{p,q}(\Hc,M)= \left\{ \begin{matrix} M\ot \Hc^{\ot q-p},&\quad\text{if}\quad 0\le p\le q, \\
&\\
0, & \text{otherwise.}
\end{matrix}\right.
\end{align}

\nin One also defines the periodic cyclic cohomology of $\Hc$ with coefficients in $M$, which is denoted by $HP^\ast(\Hc,M)$, as the total cohomology of direct sum total  of the bicomplex

\begin{align}
C^{p,q}(\Hc,M)= \left\{ \begin{matrix} M\ot \Hc^{\ot q-p},&\quad\text{if}\quad  p\le q, \\
&\\
0, & \text{otherwise.}
\end{matrix}\right.
\end{align}

\nin It can be seen that the periodic cyclic complex and hence the cohomology is $\Zb_2$ graded.

\subsection{Hopf-cyclic cohomology of Lie-Hopf algebras}
Our first task in this subsection is to calculate the periodic cyclic cohomology of $R(\Fg)$ with coefficients in a general  SAYD module. This will generalize our result in  \cite[ Theorem 4.7]{RangSutl}, where the coefficients were the induced ones, \ie those SAYD modules induced by a module over $\Fg$.

\medskip
\nin In the second subsubsection we prove the main result of this paper. Roughly  speaking,  we identify the  Hopf cyclic cohomology of a bicrossed product Hopf algebra, associated with a Lie algebra decomposition, with the Lie algebra cohomology of the ambient Lie algebra. The novelty here is the fact that we are able to prove such a  van Est type theorem  with the most general coefficients.

\subsubsection{Hopf algebra of representative  functions}
Let $M$ be a locally finite $\Fg$-module and locally conilpotent $\Fg$-comodule.  We first define a left $R(\Fg)$-coaction on $M$. It is known, by \cite{Hoch74}(see also \cite{RangSutl}), that the linear map
\begin{equation}
 M \to U(\Fg)^* \ot M , \quad  m \mapsto m^{\sns{-1}} \ot m^{\sns{0}}
\end{equation}
defined by  the rule  $m^{\sns{-1}}(u)m^{\sns{0}} = m \cdot u$, defines a left $R(\Fg)$-comodule structure
\begin{align}
\begin{split}
& \Db: M \to R(\Fg) \ot M \\
& m \mapsto m^{\sns{-1}} \ot m^{\sns{0}}.
\end{split}
\end{align}
Then using the left $U(\Fg)$-comodule on $M$, we define the  right $R(\Fg)$-module structure
\begin{align}
\begin{split}
& M \ot R(\Fg) \to M \\
& m \ot f \mapsto m \cdot f := f(m\snsb{-1})m\snsb{0}.
\end{split}
\end{align}
\begin{proposition}\label{proposition-AYD-R(g)}
Let $M$ be locally finite as a $\Fg$-module and locally conilpotent as a  $\Fg$-comodule. If $M$ is an AYD over $\Fg$, then it is an AYD over $R(\Fg)$.
\end{proposition}

\begin{proof}
Since $M$ is AYD module over the Lie algebra $\Fg$ with a locally conilpotent $\Fg$-coaction, it is an AYD over $U(\Fg)$.

\medskip
\nin Let $m \in M$, $f \in R(\Fg)$ and $u \in U(\Fg)$. On one hand side of the AYD condition we  have
\begin{align}
\Db_{R(\Fg)}(m \cdot f)(u) = (m \cdot f) \cdot u = f(m\snsb{-1})m\snsb{0} \cdot u,
\end{align}
and on the other hand,
\begin{align}
\begin{split}
& (S(f\ps{3})m^{\sns{-1}}f\ps{1})(u)m^{\sns{0}} \cdot f\ps{2} = \\
& S(f\ps{3})(u\ps{1})m^{\sns{-1}}(u\ps{2})f\ps{1}(u\ps{3})f\ps{2}((m^{\sns{0}})\snsb{-1})(m^{\sns{0}})\snsb{0} = \\
& f\ps{2}(S(u\ps{1}))m^{\sns{-1}}(u\ps{2})f\ps{1}(u\ps{3}(m^{\sns{0}})\snsb{-1})(m^{\sns{0}})\snsb{0} = \\
& m^{\sns{-1}}(u\ps{2})f(u\ps{3}(m^{\sns{0}})\snsb{-1}S(u\ps{1}))(m^{\sns{0}})\snsb{0} = \\
& f(u\ps{3}(m^{\sns{-1}}(u\ps{2})m^{\sns{0}})\snsb{-1}S(u\ps{1}))(m^{\sns{-1}}(u\ps{2})m^{\sns{0}})\snsb{0} = \\
& f(u\ps{3}(m \cdot u\ps{2})\snsb{-1}S(u\ps{1}))(m \cdot u\ps{2})\snsb{0} = \\
& f(u\ps{3}S(u\ps{2}\ps{3})m\snsb{-1}u\ps{2}\ps{1}S(u\ps{1}))m\snsb{0} \cdot u\ps{2}\ps{2} = \\
& f(u\ps{5}S(u\ps{4})m\snsb{-1}u\ps{2}S(u\ps{1}))m\snsb{0} \cdot u\ps{3} = \\
& f(m\snsb{-1})m\snsb{0} \cdot u,
\end{split}
\end{align}
where we used the AYD condition on $U(\Fg)$ on the sixth equality. This proves that $M$ is an AYD module over $R(\Fg)$.
\end{proof}

\begin{theorem}\label{g-R(g)spectral sequence}
Let $\Fg$ be a Lie algebra and $\Fg = \Fh \ltimes \Fl$ be a Levi decomposition. Let $M$ be a unimodular SAYD module over $\Fg$ as a locally finite $\Fg$-module and locally conilpotent $\Fg$-comodule. Assume also that $M$ is stable over $R(\Fg)$.  Then the  periodic cyclic cohomology of $\Fg$ relative to the subalgebra $\Fh \subseteq \Fg$ with coefficients in $M$ is the same as the periodic cyclic cohomology of $R(\Fg)$ with coefficients in $M$. In short,
\begin{align}
\widetilde{HP}(\Fg, \Fh, M) \cong HP(R(\Fg), M)
\end{align}
\end{theorem}

\begin{proof}
Since $M$ is a unimodular stable AYD module over $\Fg$, by Lemma \ref{lemma-relative-Koszul} the relative perturbed Koszul complex $(W(\Fg, \Fh, M), d_{\rm CE} + d_{\rm K})$ is well defined. On the other hand, since $M$ is locally finite as a $\Fg$-module and locally conilpotent as a $\Fg$-comodule, it is an AYD module over $R(\Fg)$ by Proposition \ref{proposition-AYD-R(g)}. Together with the assumption that $M$ is stable over $R(\Fg)$, the Hopf-cyclic complex $(C(R(\Fg),M), b + B)$ is well defined.

\medskip
\nin Since $M$ is a unimodular SAYD over  $\Fg$, $M_{\b}:= M\ot\Cb_{\b}$ is SAYD over $\Fg$ by \cite[Proposition 4.3]{RangSutl-II}. Here $\b$ is the trace of the adjoint representation of the Lie algebra $\Fg$ on itself. Therefore, by \cite[Lemma 6.2]{JaraStef} we have the filtration $M = \cup_{p \in \Zb}F_pM$ defined as $F_0M = M^{coU(\Fg)}$ and inductively
\begin{equation}
F_{p+1}M/F_pM = (M/F_pM)^{coU(\Fg)}.
\end{equation}
This filtration naturally induces an analogous filtration on the complexes as
\begin{equation}
F_jW(\Fg, \Fh, M) = W(\Fg, \Fh, F_jM), \quad \text{and} \quad F_jC(R(\Fg),M) = C(R(\Fg),F_jM).
\end{equation}

\nin We now show that the (co)boundary maps $d_{\rm CE},d_{\rm K},b,B$ respect this filtration. To do so for $d_{\rm K}$ and $d_{\rm CE}$, it  suffices to show that the $\Fg$-action and $\Fg$-coaction on $M$ respect the filtration; which is done by  the same argument as in  the proof of Theorem 6.2 in \cite{RangSutl-II}. Similarly, to show that the Hochschild coboundary $b$ and the Connes boundary map $B$ respect the filtration we need to show that $R(\Fg)$-action and $R(\Fg)$-coaction respects the filtration.

\medskip
\nin Indeed, for an element $m \in F_pM$, writing the $U(\Fg)$-coaction as
\begin{equation}
m\snsb{-1} \ot m\snsb{0} = 1 \ot m + m\ns{-1} \ot m\ns{0}, \qquad m\ns{-1} \ot m\ns{0} \in U(\Fg) \ot F_{p-1}M,
\end{equation}
we get for any $f \in R(\Fg)$
\begin{equation}
m \cdot f = \ve(f)m + f(m\ns{-1})m\ns{0} \in F_pM.
\end{equation}
This proves that $R(\Fg)$-action respects the filtration. To prove that $R(\Fg)$-coaction respects the filtration, we first write the coaction on $m \in F_pM$ as
\begin{equation}
m \mapsto \sum_i f^i \ot m_i \in R(\Fg) \ot M.
\end{equation}
By \cite{Hoch-book} there are elements $u_j \in U(\Fg)$ such that $f^i(u_j) = \d^i_j$. Hence, for any $m_{i_0}$ we have
\begin{equation}
m_{i_0} = \sum_i f^i(u_{i_0}) m_i = m\cdot u_{i_0} \in F_pM.
\end{equation}
We have proved that $R(\Fg)$-coaction respects the filtration.

\medskip
\nin Next, we write the $E_1$ terms of the associated spectral sequences. We have
\begin{align}
\begin{split}
& E_1^{j,i}(\Fg,\Fh,M) = H^{i+j}(F_jW(\Fg, \Fh, M)/F_{j-1}W(\Fg, \Fh, M)) \\
& = H^{i+j}(W(\Fg, \Fh, F_jM/F_{j-1}M)) = \bigoplus_{i+j = n \, mod \, 2}H^n(\Fg, \Fh, F_jM/F_{j-1}M),
\end{split}
\end{align}
where on the last equality we used the fact that $F_{j}M/F_{j-1}M$ has trivial $\Fg$-coaction.

\medskip
\nin Similarly we have
\begin{align}
\begin{split}
& E_1^{j,i}(R(\Fg),M) = H^{i+j}(F_jC(R(\Fg),M)/F_{j-1}C(R(\Fg),M)) \\
& = H^{i+j}(C(R(\Fg),F_jM/F_{j-1}M)) = \bigoplus_{i+j = n\, mod\,2}H^n_{\rm coalg}(R(\Fg),F_jM/F_{j-1}M)),
\end{split}
\end{align}
where on the last equality we could use \cite[Theorem 4.3]{RangSutl} due to the fact that $F_jM/F_{j-1}M$ has trivial $\Fg$-coaction, hence trivial $R(\Fg)$-action.

\medskip
\nin Finally, under the hypothesis of the theorem, a quasi-isomorphism between $E_1$ terms is given by \cite[Theorem 4.7]{RangSutl}.
\end{proof}

\begin{remark}{\rm
If $M$ has a trivial $\Fg$-comodule structure, then $d_{\rm K} = 0$ and hence the above theorem descents to \cite[Theorem 4.7]{RangSutl}.
}\end{remark}

\subsubsection{Bicrossed product Hopf algebras}
Let $M$ be an AYD module over a double crossed sum Lie algebra $\Fg_1 \bowtie \Fg_2$ with a locally finite $\Fg_1 \bi \Fg_2$-action and a locally conilpotent $\Fg_1 \bi \Fg_2$-coaction. Then by Proposition \ref{aux-46} $M$ is a YD module over the bicrossed product Hopf algebra $R(\Fg_2) \acl U(\Fg_1)$.

\medskip
\nin Let also $(\delta,\sigma)$ be an MPI for the Hopf algebra $R(\Fg_2) \acl U(\Fg_1)$, see \cite[Theorem 3.2]{RangSutl}. Then $\, ^{\sigma}M_{\delta} := M \ot \, ^{\sigma}\mathbb{C}_{\delta}$ is an AYD module over the bicrossed product Hopf algebra $R(\Fg_2) \acl U(\Fg_1)$.

\medskip
\nin Finally, let us assume that $M$ is stable over $R(\Fg_2)$ and $U(\Fg_1)$. Then $\, ^{\sigma}M_{\delta}$ is stable if and only if
\begin{align}
\begin{split}
& m \ot 1_{\mathbb{C}} = (m^{\sns{0}}\snsb{0} \ot 1_{\mathbb{C}}) \cdot (m^{\sns{-1}} \acl m^{\sns{0}}\snsb{-1})(\sigma \acl 1) \\
& = (m^{\sns{0}}\snsb{0} \ot 1_{\mathbb{C}}) \cdot (m^{\sns{-1}} \acl 1)(1 \acl m^{\sns{0}}\snsb{-1})(\sigma \acl 1) \\
& = (m^{\sns{0}}\snsb{0} \cdot m^{\sns{-1}} \ot 1_{\mathbb{C}}) \cdot (1 \acl m^{\sns{0}}\snsb{-1})(\sigma \acl 1) \\
& = ((m^{\sns{0}}\snsb{0} \cdot m^{\sns{-1}}) \cdot m^{\sns{0}}\snsb{-1}\delta(m^{\sns{0}}\snsb{-2}) \ot 1_{\mathbb{C}}) \cdot (\sigma \acl 1) \\
& = m\snsb{0}\delta(m\snsb{-1})\sigma \ot 1_{\mathbb{C}}.
\end{split}
\end{align}
Here, on the fourth equality we have used Proposition \ref{aux-59}. In other words, $M$ satisfying the hypothesis of the Proposition \ref{aux-59}, $\, ^{\sigma}M_{\delta}$ is stable if and only if
\begin{align}\label{aux-47}
m\snsb{0}\delta(m\snsb{-1})\sigma = m
\end{align}

\begin{theorem}\label{aux-63}
Let $(\Fg_1, \Fg_2)$ be a matched pair of Lie algebras and $\Fg_2 = \Fh \ltimes \Fl$ be a Levi decomposition such that $\Fh$ is $\Fg_1$ invariant and it acts on $\Fg_1$ by derivations. Let $M$ be a unimodular SAYD module over $\Fg_1 \bowtie \Fg_2$ with a locally finite $\Fg_1 \bi \Fg_2$-action and locally conilpotent $\Fg_1 \bowtie \Fg_2$-coaction. Finally assume that $\, ^{\sigma}M_{\delta}$ is stable. Then we have
\begin{equation}
HP(R(\Fg_2) \acl U(\Fg_1),\, ^{\sigma}M_{\delta}) \cong \widetilde{HP}(\Fg_1 \bowtie \Fg_2, \Fh, M)
\end{equation}
\end{theorem}

\begin{proof}
Let $C(U(\Fg_1) \cl R(\Fg_2),\, ^{\sigma}M_{\delta})$ be the complex computing the Hopf-cyclic cohomology of the bicrossed product Hopf algebra $\mathcal{H} =  R(\Fg_2) \acl U(\Fg_1)$ with coefficients in the SAYD module $^{\sigma}M_{\delta}$.

\medskip
\nin By \cite{MoscRang09}, Theorem 3.16, this complex is quasi-isomorphic with the total complex of the mixed complex $\FZ(\mathcal{H},R(\Fg_2);\, ^{\sigma}M_{\delta})$ whose cylindrical structure is given by the following operators. The horizontal operators are
\begin{align}\label{horizontal-operators}
\begin{split}
& \hd_{0}(m \ot \td{f} \ot \td{u}) = m \ot 1 \ot f^1 \odots f^p \ot \td{u} \\
& \hd_{i}(m \ot \td{f} \ot \td{u}) = m \ot f^1 \odots \D(f^i)\odots f^p \ot \td{u} \\
& \hd_{p+1}(m \ot \td{f} \ot \td{u}) = m\pr{0} \ot f^1 \odots f^p \ot \td{u}^{\pr{-1}}m\pr{-1} \rhd 1_{R(\Fg_2)} \ot \td{u}^{\pr{0}} \\
& \overset{\ra}{\sigma}_{j}(m \ot \td{f} \ot \td{u}) = m \ot f^1 \odots \ve(f^{j+1}) \odots f^p \ot \td{u} \\
&\hta(m \ot \td{f} \ot \td{u}) = m\pr{0} f^1\ps{1} \ot S(f^1\ps{2}) \cdot (f^2 \odots f^p \ot \td{u}^{\pr{-1}}m\pr{-1} \rhd 1_{R(\Fg_2)} \ot \td{u}^{\pr{0}}),
\end{split}
\end{align}
and the vertical operators are
\begin{align}\label{vertical-operators}
\begin{split}
&\vd_{0}(m \ot \td{f} \ot \td{u}) = m \ot \td{f} \ot \dot 1 \ot u^1\odots  u^q \\
&\vd_{i}(m \ot \td{f} \ot \td{u}) = m \ot \td{f} \ot u^0 \odots \D(u^i) \odots u^q \\
&\vd_{q+1}(m \ot \td{f} \ot \td{u}) = m\pr{0} \ot \td{f} \ot u^1 \odots u^q \ot \wbar{m\pr{-1}} \\
&\vs_j(m \ot \td{f} \ot \td{u}) = m \ot \td{f} \ot u^1 \odots \ve(u^{j+1}) \odots u^q \\
&\vta(m \ot \td{f} \ot \td{u}) = m\pr{0} u^1\ps{4}S^{-1}(u^1\ps{3} \rhd 1_{R(\Fg_2)}) \ot \\
& S(S^{-1}(u^1\ps{2}) \rhd 1_{R(\Fg_2)}) \cdot \left( S^{-1}(u^1\ps{1}) \rhd \td{f} \ot S(u^1\ps{5}) \cdot (u^2 \odots u^q \ot \wbar{m\pr{-1}})
\right).
\end{split}
\end{align}
for any $m \in \, ^{\sigma}M_{\delta}$.

\medskip
\nin Here,
\begin{equation}
U(\Fg_1)^{\ot \, q} \to \mathcal{H} \ot U(\Fg_1)^{\ot \, q}, \quad \td{u} \mapsto \td{u}^{\pr{-1}} \ot \td{u}^{\pr{0}}
\end{equation}
arises from the left $\mathcal{H}$-coaction on $U(\Fg_1)$ that coincides with the original $R(\Fg_2)$-coaction, \cite[Proposition 3.20]{MoscRang09}. On the other hand, $U(\Fg_1) \cong \mathcal{H} \ot_{R(\Fg_2)} \mathbb{C} \cong \mathcal{H}/\mathcal{H}R(\Fg_2)^+$ as coalgebras via the map $(f \acl u) \ot_{R(\Fg_2)} 1_{\mathbb{C}} \mapsto \ve(f)u$ and $\wbar{f \acl u} = \ve(f)u$ denotes the corresponding class.

\medskip
\nin Since $M$ is a unimodular SAYD module over $\Fg_1 \bowtie \Fg_2$, it admits the filtration $(F_pM)_{p \in \mathbb{Z}}$ defined similarly as before. We recall by Proposition \ref{proposition-comodule-doublecrossed-sum} that $\Fg_1 \bowtie \Fg_2$-coaction is the summation of $\Fg_1$-coaction and $\Fg_2$-coaction. Therefore, since $\Fg_1 \bi \Fg_2$-coaction respects the filtration, we conclude that $\Fg_1$-coaction and $\Fg_2$-coaction respect the filtration. Similarly, since $\Fg_1 \bi \Fg_2$-action respects the filtration, we conclude that $\Fg_1$-action and $\Fg_2$-action respects the filtration. Finally, by a similar argument as in the proof of Theorem \ref{g-R(g)spectral sequence} we can say that $R(\Fg_2)$-action and $R(\Fg_2)$-coaction respect the filtration.

\medskip
\nin As a result, the (co)boundary maps $d_{\rm CE}$ and $d_{\rm K}$ of the complex $W(\Fg_1 \bi \Fg_2, \Fh, M)$, and $b$ and $B$ from $C(U(\Fg_1) \cl R(\Fg_2), \, ^{\sigma}M_{\delta})$ respect the filtration.

\medskip
\nin Next, we proceed to the $E_1$ terms of the associated spectral sequences. We have
\begin{align}\label{aux-48}
\begin{split}
& E_1^{j,i}(R(\Fg_2) \acl U(\Fg_1), \, ^{\sigma}M_{\delta}) = \\
& H^{i+j}(F_jC(U(\Fg_1) \cl R(\Fg_2),\, ^{\sigma}M_{\delta})/F_{j-1}C(U(\Fg_1) \cl R(\Fg_2),\, ^{\sigma}M_{\delta})),
\end{split}
\end{align}
where
\begin{align}
\begin{split}
& F_jC(U(\Fg_1) \cl R(\Fg_2),\, ^{\sigma}M_{\delta})/F_{j-1}C(U(\Fg_1) \cl R(\Fg_2),\, ^{\sigma}M_{\delta}) = \\
& C(U(\Fg_1) \cl R(\Fg_2), F_j \, ^{\sigma}M_{\delta}/F_{j-1} \, ^{\sigma}M_{\delta}).
\end{split}
\end{align}
Since $$F_j \, ^{\sigma}M_{\delta}/F_{j-1} \, ^{\sigma}M_{\delta} = \, ^{\sigma}(F_jM/F_{j-1}M)_{\delta} =: \, ^{\sigma}\wbar{M}_{\delta}$$ has a trivial $\Fg_1 \bowtie \Fg_2$-comodule structure, its $U(\Fg_1)$-comodule structure and $R(\Fg_2)$-module structure are also trivial. Therefore, it is an $R(\Fg_2)$-SAYD in the sense of \cite{MoscRang09}. In this case, by \cite[Proposition 3.16]{MoscRang09}, $\FZ(\mathcal{H},R(\Fg_2);\, ^{\sigma}M_{\delta})$ is a bicocyclic module and the cohomology in \eqref{aux-48} is computed from the total of the following bicocyclic complex
\begin{align}\label{bicocyclic-bicomplex}
\begin{xy}  \xymatrix{  \vdots\ar@<.6 ex>[d]^{\uparrow B} & \vdots\ar@<.6 ex>[d]^{\uparrow B}
 &\vdots \ar@<.6 ex>[d]^{\uparrow B} & &\\
 ^\s{\overline{M}}_\d \ot U(\Fg_1)^{\ot 2} \ar@<.6 ex>[r]^{\hb}\ar@<.6
ex>[u]^{  \uparrow b  } \ar@<.6 ex>[d]^{\uparrow B}&
  ^\s{\overline{M}}_\d\ot U(\Fg_1)^{\ot 2}\ot  R(\Fg_2)   \ar@<.6 ex>[r]^{\hb}\ar@<.6 ex>[l]^{\hB}\ar@<.6 ex>[u]^{  \uparrow b  }
   \ar@<.6 ex>[d]^{\uparrow B} & ^\s{\overline{M}}_\d\ot U(\Fg_1)^{\ot 2}\ot R(\Fg_2)^{\ot 2}
   \ar@<.6 ex>[r] \ar@<.6 ex>[l]^{~~\hB} \ar@<.6 ex>[u]^{  \uparrow b  }
   \ar@<.6 ex>[d]^{\uparrow B}& \ar@<.6 ex>[l] \hdots&\\
^\s{\overline{M}}_\d \ot U(\Fg_1) \ar@<.6 ex>[r]^{\hb}\ar@<.6 ex>[u]^{  \uparrow b  }
 \ar@<.6 ex>[d]^{\uparrow B}&  ^\s{\overline{M}}_\d \ot U(\Fg_1) \ot R(\Fg_2) \ar@<.6 ex>[r]^{\hb}
 \ar@<.6 ex>[l]^{\hB}\ar@<.6 ex>[u]^{  \uparrow b  } \ar@<.6 ex>[d]^{\uparrow B}
 &^\s{\overline{M}}_\d\ot U(\Fg_1) \ot  R(\Fg_2)^{\ot 2}  \ar@<.6 ex>[r] \ar@<.6 ex>[l]^{\hB}\ar@<.6 ex>[u]^{  \uparrow b  }
  \ar@<.6 ex>[d]^{\uparrow B}&\ar@<.6 ex>[l] \hdots&\\
^\s{\overline{M}}_\d  \ar@<.6 ex>[r]^{\hb}\ar@<.6 ex>[u]^{  \uparrow b  }&
^\s{\overline{M}}_\d\ot R(\Fg_2) \ar@<.6 ex>[r]^{\hb}\ar[l]^{\hB}\ar@<.6
ex>[u]^{  \uparrow b  }&^\s{\overline{M}}_\d\ot R(\Fg_2)^{\ot 2}  \ar@<.6
ex>[r]^{~~\hb}\ar@<.6 ex>[l]^{\hB}\ar@<1 ex >[u]^{  \uparrow b  }
&\ar@<.6 ex>[l]^{~~\hB} \hdots& .}
\end{xy}
\end{align}

\nin Moreover, by \cite[Proposition 5.1]{RangSutl}, the total of the bicomplex \eqref{bicocyclic-bicomplex} is quasi-isomorphic to the total of the bicomplex
\begin{align}
\begin{xy} \xymatrix{  \vdots & \vdots
 &\vdots &&\\
 \;^\s{\overline{M}}_\d\ot \wg^2{\Fg_1}^\ast  \ar[u]^{d_{\rm CE}}\ar[r]^{b^\ast_{R(\Fg_2)}\;\;\;\;\;\;\;\;\;\;}&  \;^\s{\overline{M}}_\d\ot \wg^2{\Fg_1}^\ast\ot R(\Fg_2) \ar[u]^{d_{\rm CE}} \ar[r]^{b^\ast_R(\Fg_2)}& \;^\s{\overline{M}}_\d\ot \wg^2{\Fg_1}^\ast\ot R(\Fg_2)^{\ot 2} \ar[u]^{d_{\rm CE}} \ar[r]^{\;\;\;\;\;\;\;\;\;\;\;\;\;\;\;\;\;\;\;\;\;\;\;b^\ast_{R(\Fg_2)}} & \hdots&  \\
 \;^\s{\overline{M}}_\d\ot {\Fg_1}^\ast  \ar[u]^{d_{\rm CE}}\ar[r]^{b^\ast_{R(\Fg_2)}\;\;\;\;\;\;\;\;\;\;\;\;\;}& \;^\s{\overline{M}}_\d\ot {\Fg_1}^\ast\ot R(\Fg_2) \ar[u]^{d_{\rm CE}} \ar[r]^{b^\ast_{R(\Fg_2)}}& \;^\s{\overline{M}}_\d\ot  {\Fg_1}^\ast\ot R(\Fg_2)^{\ot 2} \ar[u]^{d_{\rm CE}} \ar[r]^{\;\;\;\;\;\;\;\;\;\;\;\;\;\;\;\;\;\;\;b^\ast_R(\Fg_2) }& \hdots&  \\
   \;^\s{\overline{M}}_\d\ar[u]^{d_{\rm CE}}\ar[r]^{b^\ast_{R(\Fg_2)}~~~~~~~}& \;^\s{\overline{M}}_\d\ot R(\Fg_2) \ar[u]^{d_{\rm CE}}\ar[r]^{b^\ast_R(\Fg_2)}& \;^\s{\overline{M}}_\d\ot R(\Fg_2)^{\ot 2} \ar[u]^{d_{\rm CE}} \ar[r]^{\;\;\;\;\;\;\;\;\;\;\;\;\;\;\;b^\ast_{R(\Fg_2)}} & \hdots&, }
\end{xy}
\end{align}
where $b^\ast_{R(\Fg_2)}$ is the coalgebra Hochschild coboundary with coefficients in the $R(\Fg_2)$-comodule $\;^\s{\overline{M}}_\d \ot \wg^q{\Fg_1}^\ast$.

\medskip
\nin Similarly,
\begin{align}\label{aux-49}
\begin{split}
& E_1^{j,i}(\Fg_1 \bowtie \Fg_2,\Fh,M) = H^{i+j}(F_jW(\Fg_1 \bowtie \Fg_2,\Fh,M)/F_{j-1}W(\Fg_1 \bowtie \Fg_2,\Fh,M)) \\
& = H^{i+j}(W(\Fg_1 \bowtie \Fg_2,\Fh,F_jM/F_{j-1}M)) = \bigoplus_{i+j = n \, mod \, 2}H^n(\Fg_1 \bi \Fg_2, \Fh, F_jM/F_{j-1}M)
\end{split}
\end{align}
where the last equality follows from the fact that $F_jM/F_{j-1}M$ has a trivial $\Fg_1 \bowtie \Fg_2$-comodule structure.

\medskip
\nin Finally, the quasi isomorphism between the $E_1$-terms \eqref{aux-48} and \eqref{aux-49} is given by the Corollary 5.10 of \cite{RangSutl}.
\end{proof}

\begin{remark}
{\rm
In case of a trivial $\Fg_1 \bowtie \Fg_2$-coaction, this theorem becomes \cite[Corollary 5.10]{RangSutl}. In this case, $U(\Fg_1)$-coaction and $R(\Fg_2)$-action are trivial, therefore the condition \eqref{aux-47} is obvious.}
\end{remark}

\section{Illustration}

In this section, first  we exercise our method in Sections \ref{Sec-Lie-Hopf} to provide a highly nontrivial 4-dimensional  SAYD module over $\Hc_{1\rm S}\cop\cong R(\Cb)\acl U({gl_1}^{\rm aff})$, the Schwarzian Hopf
algebra, which is introduced in \cite{ConnMosc98}.    The merit of this example is the nontrivially  of the $R(\Cb)$-action and the
$U({gl_1}^{\rm aff})$-coaction which were assumed to be trivial for induced modules in \cite{RangSutl}.   We then  illustrate Theorem \ref{aux-63} by computing two hand sides of the theorem. At the end we explicitly compute
the representative cocycles for these  cohomology classes.
From now on we denote $R(\Cb)$ by $\Fc$, $U({gl_1}^{\rm aff})$ by $\Uc$ and $\Hc_{1\rm S}\cop$ by $\Hc$.

\subsection{A 4-dimensional SAYD module on the Schwarzian Hopf algebra}
Let us first recall the Lie algebra $sl_2$ as a double crossed sum Lie algebra. We have $sl_2 = \Fg_1 \bowtie \Fg_2$, $\Fg_1 = \Cb\Big\langle X,Y \Big\rangle$, $\Fg_2 = \Cb\Big\langle Z \Big\rangle$, and the Lie bracket is
\begin{equation}
[Y,X] = X, \quad [Z,X] = Y, \quad [Z,Y] = Z.
\end{equation}

\nin Let us take $M = S({sl_2}^\ast)\nsb{2}$. By Example \ref{ex-2}, $M$ is an SAYD over $sl_2$ via the coadjoint action and the Koszul coaction.

\medskip
\nin Writing ${\Fg_2}^\ast = \Cb\Big\langle \d_1 \Big\rangle$, we have $\Fc = R(\Fg_2) = \mathbb{C}[\delta_1]$. Also, $\Uc = U(\Fg_1)$ and it is immediate to realize that $\Fc \acl \Uc \cong {\Hc_{1 \rm S}}\cop$ \cite{MoscRang07}.

\medskip
\nin Next, we construct the $\Fc \acl \Uc$-(co)action explicitly and we verify that $\, ^{\sigma}M_{\delta}$ is an SAYD over $\Fc \acl \Uc$. Here, $(\sigma,\delta)$ is the canonical modular pair in involution associated to the bicrossed product $\Fc \acl \Uc$ \cite{RangSutl}.
By definition  $\delta = Tr \circ ad_{\Fg_1}$. Let us compute $\sigma \in \Fc$ from the right $\Fc$-coaction on $\Uc$.

\medskip
\nin Considering the formula $[v,X] = v \rhd X \oplus v \lhd X$, the action $\Fg_2 \lhd \Fg_1$ is given by
\begin{equation}
Z \lhd X = 0, \quad Z \lhd Y = Z.
\end{equation}
Similarly, the action $\Fg_2 \rhd \Fg_1$ is
\begin{equation}
Z \rhd X = Y, \quad Z \rhd Y = 0.
\end{equation}

\nin Dualizing the left action $\Fg_2 \rhd \Fg_1$, we have the $\Fc$-coaction on $\Uc$ as follows
\begin{align}
\begin{split}
& \Uc \to \Uc \ot \Fc, \quad u \mapsto u^{\pr{0}} \ot u^{\pr{1}} \\
& X \mapsto X \ot 1 + Y \ot \delta_1 \\
& Y \mapsto Y \ot 1.
\end{split}
\end{align}
Hence, by \cite{RangSutl} Section 3.1
\begin{equation}
\sigma = det \left(
               \begin{array}{cc}
                 1 & \delta_1 \\
                 0 & 1 \\
               \end{array}
             \right) = 1.
\end{equation}

\nin On the other hand, by the Lie algebra structure of $\Fg_1 \cong {gl_1}^{\rm aff}$, we have
\begin{equation}
\delta(X) = 0, \qquad \delta(Y) = 1.
\end{equation}

\nin Next,  we express the $\Fc \acl \Uc$-coaction on $M = S({sl_2}^\ast)\nsb{2}$ explicitly. A vector space basis of $M$ is given by $\Big\{1_M, R^X, R^Y, R^Z\Big\}$ and the $\Fg_1$-coaction (Kozsul) is
\begin{equation}
M \to \Fg_1 \ot M, \quad 1_M \mapsto X \ot R^X + Y \ot R^Y, \quad R^i \mapsto 0, \quad i = X,Y,Z.
\end{equation}

\nin Note that the application of this coaction twice is zero, thus it is locally conilpotent. Then the corresponding $\Uc$ coaction is
\begin{align}
\begin{split}
& M \to \Uc \ot M, \quad m \mapsto m\snsb{-1} \ot m\snsb{0} \\
& 1_M \mapsto 1 \ot 1_M + X \ot R^X + Y \ot R^Y \\
& R^i \mapsto 1 \ot R^i, \quad i = X,Y,Z.
\end{split}
\end{align}

\nin To determine the left $\Fc$-coaction, we need to dualize the right $\Fg_2$-action. We have
\begin{equation}
1_M \lhd Z = 0, \quad R^X \lhd Z = 0, \quad R^Y \lhd Z = R^X, \quad R^Z \lhd Z = R^Y,
\end{equation}
implying
\begin{align}
\begin{split}
& M \to \Fc \ot M, \quad m \mapsto m^{\sns{-1}} \ot m^{\sns{0}} \\
& 1_M \mapsto 1 \ot 1_M \\
& R^X \mapsto 1 \ot R^X \\
& R^Y \mapsto 1 \ot R^Y + \delta_1 \ot R^X \\
& R^Z \mapsto 1 \ot R^Z + \delta_1 \ot R^Y + \frac{1}{2}\delta_1^2 \ot R^X.
\end{split}
\end{align}
As a result, $\Fc \acl \Uc$-coaction appears as follows
\begin{align}
\begin{split}
& M \to \Fc \acl \Uc \ot M, \quad m \mapsto m^{\sns{-1}} \acl m^{\sns{0}}\snsb{-1} \ot m^{\sns{0}}\snsb{0} \\
& 1_M \mapsto 1 \ot 1_M + X \ot R^X + Y \ot R^Y \\
& R^X \mapsto 1 \ot R^X \\
& R^Y \mapsto 1 \ot R^Y + \delta_1 \ot R^X \\
& R^Z \mapsto 1 \ot R^Z + \delta_1 \ot R^Y + \frac{1}{2}\delta_1^2 \ot R^X.
\end{split}
\end{align}

\nin Let us next determine the right $\Fc \acl \Uc$-action. It is enough to determine the $\Uc$-action and $\Fc$-action separately. The action of $\Uc$ is directly given by
\begin{align}
\begin{split}
& 1_M \lhd X = 0, \quad 1_M \lhd Y = 0 \\
& R^X \lhd X = -R^Y, \quad R^X \lhd Y = R^X \\
& R^Y \lhd X = -R^Z, \quad R^Y \lhd Y = 0 \\
& R^Z \lhd X = 0, \quad R^Z \lhd Y = -R^Z.
\end{split}
\end{align}

\nin To be able to see the $\Fc$-action, we determine the $\Fg_2$-coaction. This follows from the Kozsul coaction on $M$, \ie
\begin{align}
\begin{split}
& M \to U(\Fg_2) \ot M, \quad m \mapsto m\sns{-1} \ot m\sns{0} \\
& 1_M \mapsto 1 \ot 1_M + Z \ot R^Z \\
& R^i \mapsto 1 \ot R^i, \quad i = X,Y,Z.
\end{split}
\end{align}
Hence, $\Fc$-action is given by
\begin{equation}
1_M \lhd \delta_1 = R^Z, \quad R^i \lhd \delta_1 = 0, \quad i = X,Y,Z.
\end{equation}

\nin We will now check carefully  that $M$ is a YD module over the bicrossed product Hopf algebra $\Fc \acl \Uc$. We leave to the reader to check that $M$ satisfies the conditions introduced in Lemma \ref{module on bicrossed product} and Lemma \ref{comodule on bicrossed product}; that is  $M$ is a module and comodule on $\Fc \acl \Uc$ respectively. We proceed to the verification of the YD condition on the bicrossed product Hopf algebra $\Fc \acl \Uc$.

\medskip
\nin By the multiplicative property of the YD condition, it is enough to check that the condition holds  for the elements $X, Y, \delta_1 \in \Fc \acl \Uc$.

\medskip
\nin For simplicity of the notation, we write the $\Fc \acl \Uc$-coaction as $m \mapsto m\pr{-1} \ot m\pr{0}$.

\medskip
\nin We begin with $1_M \in M$ and $X \in \Fc \acl \Uc$. On one hand we have
\begin{align}
\begin{split}
& X\ps{2} \cdot (1_M \lhd X\ps{1})\pr{-1} \ot (1_M \lhd X\ps{1})\pr{0} = \\
& (1_M \lhd X)\pr{-1} \ot (1_M \lhd X)\pr{0} + X 1_M\pr{-1} \ot 1_M\pr{0} + \delta_1(1_M \lhd Y)\pr{-1} \ot (1_M \lhd Y)\pr{0} = \\
& X \ot 1_M + X^2 \ot R^X + XY \ot R^Y,
\end{split}
\end{align}
and on the other hand,
\begin{align}
\begin{split}
& 1_M\pr{-1} X\ps{1} \ot 1_M\pr{0} \lhd X\ps{2} = \\
& 1_M\pr{-1} X \ot 1_M\pr{0} + 1_M\pr{-1} \ot 1_M\pr{0} \lhd X + 1_M\pr{-1} Y \ot 1_M\pr{0} \lhd \delta_1 = \\
& X \ot 1_M + X^2 \ot R^X + YX \ot R^Y + X \ot R^X \lhd X + Y \ot R^Y \lhd X + Y \ot 1_M \lhd \delta_1.
\end{split}
\end{align}
In view of $[Y,X]=X$, $R^X \lhd X = -R^X$, $R^Y \lhd X = -R^Z$ and $1_M \lhd \delta_1 = R^Z$, we have the YD compatibility is satisfied for $1_M \in M$ and $X \in \Fc \acl \Uc$.

\medskip
\nin We proceed to check the condition for  $1_M \in M$ and $Y \in \Fc \acl \Uc$. We have
\begin{align}
\begin{split}
& Y\ps{2} \cdot (1_M \lhd Y\ps{1})\pr{-1} \ot (1_M \lhd Y\ps{1})\pr{0} = \\
& (1_M \lhd Y)\pr{-1} \ot (1_M \lhd Y)\pr{0} + Y 1_M\pr{-1} \ot 1_M\pr{0} = \\
& Y \ot 1_M + YX \ot R^X + Y^2 \ot R^Y,
\end{split}
\end{align}
and
\begin{align}
\begin{split}
& 1_M\pr{-1} Y\ps{1} \ot 1_M\pr{0} \lhd Y\ps{2} = \\
& 1_M\pr{-1} Y \ot 1_M\pr{0} + 1_M\pr{-1} \ot 1_M\pr{0} \lhd Y = \\
& Y \ot 1_M + XY \ot R^X + Y^2 \ot R^Y + X \ot R^X \lhd Y.
\end{split}
\end{align}
We use $[Y,X]=X$ and $R^X \lhd Y = R^X$, and hence the YD condition is satisfied for $1_M \in M$ and $Y \in \Fc \acl \Uc$.

\medskip
\nin For $1_M \in M$ and $\delta_1 \in \Fc \acl \Uc$ we have
\begin{align}
\begin{split}
& \delta_1\ps{2} (1_M \lhd \delta_1\ps{1})\pr{-1} \ot (1_M \lhd \delta_1\ps{1})\pr{0} = \\
& (1_M \lhd \delta_1)\pr{-1} \ot (1_M \lhd \delta_1)\pr{0} + \delta_1 1_M\pr{-1} \ot 1_M\pr{0} = \\
& 1 \ot R^Z + \delta_1 \ot R^Y + \frac{1}{2}\delta_1^2  \ot R^X + \delta_1 \ot 1_M + \delta_1X \ot R^X + \delta_1Y \ot R^Y.
\end{split}
\end{align}
On the other hand,
\begin{align}
\begin{split}
& 1_M\pr{-1} \delta_1\ps{1} \ot 1_M\pr{0} \lhd \delta_1\ps{2} = \\
& 1_M\pr{-1} \delta_1 \ot 1_M\pr{0} + 1_M\pr{-1} \ot 1_M\pr{0} \lhd \delta_1 = \\
& \delta_1 \ot 1_M + X\delta_1 \ot R^X + Y\delta_1 \ot R^Y + 1 \ot 1_M \lhd \delta_1.
\end{split}
\end{align}
Thus, the YD condition for $1_M \in M$ and $\delta_1 \in \Fc \acl \Uc$ follows from $[X,\delta_1] = \frac{1}{2}\delta_1^2$, $[Y,\delta_1] = \delta_1$ and $1_M \lhd \delta_1 = R^Z$.

\medskip
\nin Next, we consider $R^X \in M$ and $X \in \Fc \acl \Uc$. In this case we have,
\begin{align}
\begin{split}
& X\ps{2} (R^X \lhd X\ps{1})\pr{-1} \ot (R^X \lhd X\ps{1})\pr{0} = \\
& (R^X \lhd X)\pr{-1} \ot (R^X \lhd X)\pr{0} + X R^X\pr{-1} \ot R^X\pr{0} + \delta_1 (R^X \lhd Y)\pr{-1} \ot (R^X \lhd Y)\pr{0} = \\
& -1 \ot R^Y - \delta_1 \ot R^X + X \ot R^X + \delta_1 \ot R^X = \\
& -1 \ot R^Y + X \ot R^X.
\end{split}
\end{align}
On the other hand,
\begin{align}
\begin{split}
& R^X\pr{-1} X\ps{1} \ot R^X\pr{0} \lhd X\ps{2} = \\
& R^X\pr{-1} X \ot R^X\pr{0} + R^X\pr{-1} \ot R^X\pr{0} \lhd X + R^X\pr{-1} Y \ot R^X\pr{0} \lhd \delta_1 = \\
& X \ot R^X + 1 \ot R^X \lhd X,
\end{split}
\end{align}
and we have the equality in view of the fact that $R^X \lhd X = -R^Y$.

\medskip
\nin For $R^X \in M$ and $Y \in \Fc \acl \Uc$, on one hand side we have
\begin{align}
\begin{split}
& Y\ps{2} (R^X \lhd Y\ps{1})\pr{-1} \ot (R^X \lhd Y\ps{1})\pr{0} = \\
& (R^X \lhd Y)\pr{-1} \ot (R^X \lhd Y)\pr{0} + Y R^X\pr{-1} \ot R^X\pr{0} = \\
& 1 \ot R^X + Y \ot R^X,
\end{split}
\end{align}
and on the other hand,
\begin{align}
\begin{split}
& R^X\pr{-1}  Y\ps{1} \ot R^X\pr{0} \lhd Y\ps{2} = \\
& R^X\pr{-1}  Y \ot R^X\pr{0} + R^X \ot R^X\pr{0} \lhd Y = \\
& Y \ot R^X + 1 \ot R^X \lhd Y.
\end{split}
\end{align}
The equality is the consequence of $R^X \lhd Y = R^X$.

\medskip
\nin For $R^X \in M$ and $\delta_1 \in \Fc \acl \Uc$ we have
\begin{align}
\begin{split}
& \delta_1\ps{2} (R^X \lhd \delta_1\ps{1})\pr{-1} \ot (R^X \lhd \delta_1\ps{1})\pr{0} = \\
& (R^X \lhd \delta_1)\pr{-1} \ot (R^X \lhd \delta_1)\pr{0} + \delta_1 R^X\pr{-1} \ot R^X\pr{0} = \\
& \delta_1 \ot R^X,
\end{split}
\end{align}
and
\begin{align}
\begin{split}
& R^X\pr{-1} \delta_1\ps{1} \ot R^X\pr{0} \lhd \delta_1\ps{2} = \\
& R^X\pr{-1} \delta_1 \ot R^X\pr{0} + R^X\pr{-1} \ot R^X\pr{0} \lhd \delta_1 = \\
& \delta_1 \ot R^X + 1 \ot R^X \lhd \delta_1.
\end{split}
\end{align}
The result follows from $R^X \lhd \delta_1 = 0$.

\medskip
\nin We proceed to verify the condition for $R^Y \in M$. For $R^Y \in M$ and $X \in \Fc \acl \Uc$, we have
\begin{align}
\begin{split}
& X\ps{2} (R^Y \lhd X\ps{1})\pr{-1} \ot (R^Y \lhd X\ps{1})\pr{0} = \\
& (R^Y \lhd X)\pr{-1} \ot (R^Y \lhd X)\pr{0} + X R^Y\pr{-1} \ot R^Y\pr{0} + \delta_1 (R^Y \lhd Y)\pr{-1} \ot (R^Y \lhd Y)\pr{0} = \\
& -1 \ot R^Z - \delta_1 \ot R^Y - \frac{1}{2}\delta_1^2 \ot R^X + X \ot R^Y + X\delta_1 \ot R^X,
\end{split}
\end{align}
as well as
\begin{align}
\begin{split}
& R^Y\pr{-1} X\ps{1} \ot R^Y\pr{0} \lhd X\ps{2} = \\
& R^Y\pr{-1} X \ot R^Y\pr{0} + R^Y\pr{-1} \ot R^Y\pr{0} \lhd X + R^Y\pr{-1} Y \ot R^Y\pr{0} \lhd \delta_1 = \\
& X \ot R^Y + \delta_1X \ot R^X + 1 \ot R^Y \lhd X + \delta_1 \ot R^X \lhd X.
\end{split}
\end{align}
To see the equality, we use $[X,\delta_1] = \frac{1}{2}\delta_1^2$, $R^Y \lhd X = -R^Z$ and $R^X \lhd X = -R^Y$.

\medskip
\nin Similarly for $R^Y \in M$ and $Y \in \Fc \acl \Uc$, we have on one hand
\begin{align}
\begin{split}
& Y\ps{2} (R^Y \lhd Y\ps{1})\pr{-1} \ot (R^Y \lhd Y\ps{1})\pr{0} = \\
& (R^Y \lhd Y)\pr{-1} \ot (R^Y \lhd Y)\pr{0} + Y R^Y\pr{-1} \ot R^Y\pr{0} = \\
& Y \ot R^Y + Y\delta_1 \ot R^X,
\end{split}
\end{align}
and on the other hand,
\begin{align}
\begin{split}
& R^Y\pr{-1} Y\ps{1} \ot R^Y\pr{0} \lhd Y\ps{2} = \\
& R^Y\pr{-1} Y \ot R^Y\pr{0} + R^Y\pr{-1} \ot R^Y\pr{0} \lhd Y = \\
& Y \ot R^Y + \delta_1Y \ot R^X + \delta_1 \ot R^X \lhd Y.
\end{split}
\end{align}
Hence the equality by $[Y,\delta_1] = \delta_1$ and $R^X \lhd Y = R^X$.

\medskip
\nin Finally, for $R^Y \in M$ and $\delta_1 \in \Fc \acl \Uc$ we have
\begin{align}
\begin{split}
& \delta_1\ps{2} (R^Y \lhd \delta_1\ps{1})\pr{-1} \ot (R^Y \lhd \delta_1\ps{1})\pr{0} = \\
& (R^Y \lhd \delta_1)\pr{-1} \ot (R^Y \lhd \delta_1)\pr{0} + \delta_1 R^Y\pr{-1} \ot R^Y\pr{0} = \\
& \delta_1 \ot R^Y + \delta_1^2 \ot R^X,
\end{split}
\end{align}
and
\begin{align}
\begin{split}
& R^Y\pr{-1} \delta_1\ps{1} \ot R^Y\pr{0} \lhd \delta_1\ps{2} = \\
& R^Y\pr{-1} \delta_1 \ot R^Y\pr{0} + R^Y\pr{-1} \ot R^Y\pr{0} \lhd \delta_1 = \\
& \delta_1 \ot R^Y + \delta_1^2 \ot R^X.
\end{split}
\end{align}

\nin Now we check the condition for $R^Z \in M$. For $R^Z \in M$ and $X \in \Fc \acl \Uc$,
\begin{align}
\begin{split}
& X\ps{2} (R^Z \lhd X\ps{1})\pr{-1} \ot (R^Z \lhd X\ps{1})\pr{0} = \\
& (R^Z \lhd X)\pr{-1} \ot (R^Z \lhd X)\pr{0} + X R^Z\pr{-1} \ot R^Z\pr{0} + \delta_1 (R^Z \lhd Y)\pr{-1} \ot (R^Z \lhd Y)\pr{0} = \\
& X \ot R^Z + X\delta_1 \ot R^Y + \frac{1}{2}X\delta_1^2 \ot R^X - \delta_1 \ot R^Z - \delta_1^2 \ot R^Y - \frac{1}{2}\delta_1^3 \ot R^X.
\end{split}
\end{align}
On the other hand,
\begin{align}
\begin{split}
& R^Z\pr{-1} X\ps{1} \ot R^Z\pr{0} \lhd X\ps{2} = \\
& R^Z\pr{-1} X \ot R^Z\pr{0} + R^Z\pr{-1} \ot R^Z\pr{0} \lhd X + R^Z\pr{-1} Y \ot R^Z\pr{0} \lhd \delta_1 = \\
& X \ot R^Z + \delta_1X \ot R^Y + \frac{1}{2}\delta_1^2X \ot R^X + \delta_1 \ot R^Y \lhd X + \frac{1}{2}\delta_1^2 \ot R^X \lhd X.
\end{split}
\end{align}
Equality follows from $[X,\delta_1] = \frac{1}{2}\delta_1^2$, $R^Y \lhd X = -R^Z$ and $R^X \lhd X = -R^Y$.

\medskip
\nin Next, we consider $R^Z \in M$ and $Y \in \Fc \acl \Uc$. In this case we have
\begin{align}
\begin{split}
& Y\ps{2} (R^Z \lhd Y\ps{1})\pr{-1} \ot (R^Z \lhd Y\ps{1})\pr{0} = \\
& (R^Z \lhd Y)\pr{-1} \ot (R^Z \lhd Y)\pr{0} + Y R^Z\pr{-1} \ot R^Z\pr{0} = \\
& -1 \ot R^Z - \delta_1 \ot R^Y - \frac{1}{2}\delta_1^2 \ot R^X + Y \ot R^Z + Y\delta_1 \ot R^Y + \frac{1}{2}Y\delta_1^2 \ot R^X,
\end{split}
\end{align}
and on the other hand,
\begin{align}
\begin{split}
& R^Z\pr{-1} Y\ps{1} \ot R^Z\pr{0} \lhd Y\ps{2} = \\
& R^Z\pr{-1} Y \ot R^Z\pr{0} + R^Z\pr{-1} \ot R^Z\pr{0} \lhd Y = \\
& Y \ot R^Z + \delta_1Y \ot R^Y + \frac{1}{2}\delta_1^2Y \ot R^X + 1 \ot R^Z \lhd Y + \frac{1}{2}\delta_1^2 \ot R^X \lhd Y.
\end{split}
\end{align}
Equality follows from $[Y,\delta_1] = \delta_1$, $R^Z \lhd Y = -R^Z$ and $R^X \lhd Y = R^X$.

\medskip
\nin Finally we check the YD compatibility for $R^Z \in M$ and $\delta_1 \in \Fc \acl \Uc$. We have
\begin{align}
\begin{split}
& \delta_1\ps{2} (R^Z \lhd \delta_1\ps{1})\pr{-1} \ot (R^Z \lhd \delta_1\ps{1})\pr{0} = \\
& (R^Z \lhd \delta_1)\pr{-1} \ot (R^Z \lhd \delta_1)\pr{0} + \delta_1 R^Z\pr{-1} \ot R^Z\pr{0} = \\
& \delta_1 \ot R^Z + \delta_1^2 \ot R^Y + \frac{1}{2}\delta_1^3 \ot R^X,
\end{split}
\end{align}
and
\begin{align}
\begin{split}
& R^Z\pr{-1} \delta_1\ps{1} \ot R^Z\pr{0} \lhd \delta_1\ps{2} = \\
& R^Z\pr{-1} \delta_1 \ot R^Z\pr{0} + R^Z\pr{-1} \ot R^Z\pr{0} \lhd \delta_1 = \\
& \delta_1 \ot R^Z + \delta_1^2 \ot R^Y + \frac{1}{2}\delta_1^3 \ot R^X.
\end{split}
\end{align}

\nin We have proved  that $M$ is a YD module over the bicrossed product Hopf algebra $\Fc \acl \Uc = {\Hc_{1 \rm S}}^{\rm cop}$.

\medskip
\nin Let us now check the stability condition. Since in this case $\sigma = 1$, $\,^{\sigma}M_{\delta}$ has the same coaction as $M$. Thus, $(m \ot 1_{\mathbb{C}})\pr{-1} \ot (m \ot 1_{\mathbb{C}})\pr{0} \in \Fc \acl \Uc \ot \,^{\sigma}M_{\delta}$ denoting the coaction, we have
\begin{align*}\notag
& (1_M \ot 1_{\mathbb{C}})\pr{0} \cdot (1_M \ot 1_{\mathbb{C}})\pr{-1} = (1_M \ot 1_{\mathbb{C}}) \cdot 1 + (R^X \ot 1_{\mathbb{C}}) \cdot X + (R^Y \ot 1_{\mathbb{C}}) \cdot Y \\\notag
& = 1_M \ot 1_{\mathbb{C}} + R^X \cdot X\ps{2}\delta(X\ps{1}) \ot 1_{\mathbb{C}} + R^Y \cdot Y\ps{2}\delta(Y\ps{1}) \ot 1_{\mathbb{C}} \\\notag
& = 1_M \ot 1_{\mathbb{C}} + R^X \cdot X \ot 1_{\mathbb{C}} + R^X \delta(X) \ot 1_{\mathbb{C}} + R^Y \cdot Y \ot 1_{\mathbb{C}} + R^Y \delta(Y) \ot 1_{\mathbb{C}} \\
& = 1_M \ot 1_{\mathbb{C}}.\\[.5cm]
& (R^X \ot 1_{\mathbb{C}})\pr{0} \cdot (R^X \ot 1_{\mathbb{C}})\pr{-1} = (R^X \ot 1_{\mathbb{C}}) \cdot 1
= R^X \ot 1_{\mathbb{C}}.\\[.5cm]
& (R^Y \ot 1_{\mathbb{C}})\pr{0} \cdot (R^Y \ot 1_{\mathbb{C}})\pr{-1} = (R^Y \ot 1_{\mathbb{C}}) \cdot 1 + (R^X \ot 1_{\mathbb{C}}) \cdot \delta_1
 = R^Y \ot 1_{\mathbb{C}}.\\[.5cm]
& (R^Z \ot 1_{\mathbb{C}})\pr{0} \cdot (R^Z \ot 1_{\mathbb{C}})\pr{-1} =\\
&\hspace{5cm} (R^Z \ot 1_{\mathbb{C}}) \cdot 1 + (R^Y \ot 1_{\mathbb{C}}) \cdot \delta_1 + (R^X \ot 1_{\mathbb{C}}) \cdot \frac{1}{2}\delta_1^2  = R^Z \ot 1_{\mathbb{C}}.
\end{align*}
Hence the stability is satisfied.

\medskip
\nin We record our discussion in the following proposition.

\begin{proposition}
 The four dimensional module-comodule  $$M_\d:=M \ot \mathbb{C}_{\delta}= \Cb\Big\langle 1_M, R^X,R^Y,R^Z\Big\rangle \ot \Cb_\d$$ is an SAYD module over the the Schwarzian Hopf algebra $\Hc_{\rm 1S}\cop$, via the action and coaction
$$
\left.
  \begin{array}{c|ccc}

    \lhd & X & Y & \d_1 \\[.2cm]
    \hline
    &&&\\[-.2cm]
    \one & 0 & 0 & \bfR^Z \\[.1cm]

    \bfR^X & -\bfR^Y & 2\bfR^X & 0 \\[.1cm]
    \bfR^Y & -\bfR^Z & \bfR^Y & 0 \\[.1cm]
    \bfR^Z & 0 & 0 & 0 \\
  \end{array}
\right.\qquad
  \begin{array}{rl}
  &\Db: M_\d \longrightarrow \mathcal{H}_{\rm 1S}\cop \ot M_\d \\[.2cm]
\hline
  &\\[-.2cm]
& \one \mapsto 1 \ot \one + X \ot \bfR^X + Y \ot \bfR^Y \\
& \bfR^X \mapsto 1 \ot \bfR^X \\
& \bfR^Y \mapsto 1 \ot \bfR^Y + \delta_1 \ot \bfR^X \\
& \bfR^Z \mapsto 1 \ot \bfR^Z + \delta_1 \ot \bfR^Y + \frac{1}{2}\delta_1^2 \ot \bfR^X.
  \end{array}
$$

\nin Here, $\one := 1_M \ot \Cb_\d, \; \bfR^X := R^X \ot \Cb_\d, \; \bfR^Y := R^Y \ot \Cb_\d, \; \bfR^Z := R^Z \ot \Cb_\d$.
\end{proposition}

\subsection{Computation of  $\widetilde{HP}(sl_2,S(sl_2^\ast)\nsb{2})$  }
This subsection is devoted to the computation of $\widetilde{HP}(sl_2,M)$ by demonstrating  explicit representatives of the cohomology classes.
We know that the perturbed Koszul complex $(W(sl_2,M), d_{\rm CE} + d_{\rm K})$ computes this  cohomology.

\medskip
\noindent Being an SAYD over $U(sl_2)$, $M$ admits the filtration $(F_pM)_{p \in \mathbb{Z}}$ from \cite{JaraStef}. Explicitly,
\begin{equation}
F_0M = \{R^X,R^Y,R^Z\}, \quad F_pM = M, \quad p \geq 1.
\end{equation}
The induced filtration on the complex is
\begin{equation}
F_j(W(sl_2,M)) := W(sl_2,F_jM),
\end{equation}
and the $E_1$ term of the associated spectral sequence is
\begin{equation}
E_1^{j,i}(sl_2,M) = H^{i+j}(W(sl_2,F_jM)/W(sl_2,F_{j-1}M)) \cong H^{i+j}(W(sl_2,F_jM/F_{j-1}M)).
\end{equation}

\nin Since $F_jM/F_{j-1}M$ has trivial $sl_2$-coaction, the boundary $d_{\rm K}$ vanish on the quotient complex $W(sl_2,F_jM/F_{j-1}M)$ and hence
\begin{equation}
E_1^{j,i}(sl_2,M) = \bigoplus_{i+j \cong \bullet \, mod \, 2}H^{\bullet}(sl_2,F_jM/F_{j-1}M).
\end{equation}

\nin In particular,
\begin{equation}
E_1^{0,i}(sl_2,M) = H^i(W(sl_2,F_0M)) \cong \bigoplus_{i \cong \bullet \, mod \, 2}H^{\bullet}(sl_2,F_0M) = 0.
\end{equation}
The last equality follows from the Whitehead's theorem (noticing that $F_0M$ is an irreducible $sl_2$-module of dimension greater than 1). For $j = 1$ we have $M/F_0M \cong \Cb$ and hence
\begin{equation}
E_1^{1,i}(sl_2,M) = \bigoplus_{i+1 \cong \bullet \, mod \, 2}H^{\bullet}(sl_2),
\end{equation}
which gives two cohomology classes as a result of Whitehead's 1st and 2nd lemmas. Finally, by $F_pM = M$ for $p \ge 1$, we have $E_1^{j,i}(sl_2,M) = 0$ for $j \geq 2$.

\medskip
\nin Let us now write the complex as
\begin{equation}
W(sl_2,M) = W^{\rm even}(sl_2,M) \oplus W^{\rm odd}(sl_2,M),
\end{equation}
where
\begin{equation}
W^{\rm even}(sl_2,M) = M \oplus (\wedge^2 {sl_2}^* \ot M), \quad W^{\rm odd}(sl_2,M) = ({sl_2}^* \ot M) \oplus (\wedge^3 {sl_2}^* \ot M).
\end{equation}

\nin Next, we demonstrate the explicit cohomology cocycles of $\widetilde{HP}(sl_2,M)$. First, let us take $1_M \in W^{\rm even}(sl_2,M)$. It is immediate to observe $d_{\rm CE}(1_M) = 0$ as well as $d_{\rm K}(1_M) = 0$. On the other hand, in the level of spectral sequence it descends to the nontrivial class $0$ of the cohomology of $sl_2$. Hence, it is a representative of the  even cohomology class.

\medskip
\nin Secondly, we consider $$(2\theta^X \ot R^Z - \theta^Y \ot R^Y \; , \; \theta^X \wg \theta^Y \wg \theta^Z \ot 1_M) \in W^{\rm odd}(sl_2,M)$$ Here $\Big\{\theta^X,\theta^Y,\theta^Z\Big\}$ is the dual basis corresponding to the basis $\Big\{X,Y,Z\Big\}$ of $sl_2$. Let us show that it is a $d_{\rm CE} + d_{\rm K}$-cocycle. It is immediate that
\begin{equation}
d_{\rm CE}(\theta^X \wg \theta^Y \wg \theta^Z \ot 1_M) = 0
\end{equation}
As for the Koszul differential,
\begin{align}
\begin{split}
& d_{\rm K}(\theta^X \wg \theta^Y \wg \theta^Z \ot 1_M) = \\
& \iota_X(\theta^X \wg \theta^Y \wg \theta^Z) \ot R^X + \iota_Y(\theta^X \wg \theta^Y \wg \theta^Z) \ot R^Y + \iota_Z(\theta^X \wg \theta^Y \wg \theta^Z) \ot R^Z = \\
& \theta^Y \wg \theta^Z \ot R^X - \theta^X \wg \theta^Z \ot R^Y + \theta^X \wg \theta^Y \ot R^Z.
\end{split}
\end{align}

\nin On the other hand, we have
\begin{align}
\begin{split}
& d_{\rm CE}(\theta^X \ot R^Z) = \theta^X \wg \theta^Y \ot R^Z - \theta^Y \wg \theta^X \ot R^Z \cdot Y - \theta^Z \wg \theta^X \ot R^Z \cdot Z \\
& = \theta^X \wg \theta^Z \ot R^Y,
\end{split}
\end{align}
and
\begin{align}
\begin{split}
& d_{\rm CE}(\theta^Y \ot R^Y) = \theta^X \wg \theta^Z \ot R^Y - \theta^X \wg \theta^Y \ot R^Y \cdot X - \theta^Z \wg \theta^Y \ot R^Z \cdot Z \\
& = \theta^X \wg \theta^Y \ot R^Z + \theta^X \wg \theta^Z \ot R^Y + \theta^Y \wg \theta^Z \ot R^X.
\end{split}
\end{align}
Therefore,
\begin{equation}
d_{\rm CE}(2\theta^X \ot R^Z - \theta^Y \ot R^Y) = - \theta^X \wg \theta^Y \ot R^Z + \theta^X \wg \theta^Z \ot R^Y - \theta^Y \wg \theta^Z \ot R^X.
\end{equation}
We also have
\begin{equation}
d_{\rm K}(2\theta^X \ot R^Z - \theta^Y \ot R^Y) = 2R^ZR^X - R^YR^Y = 0 - 0 = 0.
\end{equation}
Therefore, we can write
\begin{equation}
(d_{\rm CE} + d_{\rm K})((2\theta^X \ot R^Z - \theta^Y \ot R^Y\;,\;\theta^X \wg \theta^Y \wg \theta^Z \ot 1_M)) = 0.
\end{equation}
Finally we note that $(2\theta^X \ot R^Z - \theta^Y \ot R^Y \; , \; \theta^X \wg \theta^Y \wg \theta^Z \ot 1_M)$ descends, in the $E_1$-level of the spectral sequence,   to the cohomology class represented by the  3-cocycle $$\theta^X \wg \theta^Y \wg \theta^Z.$$ Hence, it represents the odd cohomology class.
Let us summarize our discussion so far

\begin{proposition}
The periodic cyclic cohomology of the Lie algebra $sl_2$ with coefficients in SAYD module  $M:=S({sl_2}^\ast)_{[2]}$  is represented by
\begin{align}
&\widetilde{HP}^0(sl_2,M)=\Cb\Big\langle  1_M\Big\rangle, \\
&\widetilde{HP}^1(sl_2,M)=\Cb\Big\langle (2\theta^X \ot R^Z - \theta^Y \ot R^Y\;,\;\theta^X\wedge \theta^Y\wedge \theta^Z \ot 1_M) \Big\rangle.
\end{align}
\end{proposition}

\subsection{Computation of $HP(\Hc_{1\rm S}, M_\d)$}
We now consider the complex $C(\Uc \cl \Fc, M_{\delta})$, which  computes  the periodic Hopf cyclic cohomology
\begin{equation}
HP(\Hc_{1\rm S}\cop,M_{\delta}) = HP(\Fc \acl \Uc,M_{\delta}).
\end{equation}

\nin We can immediately conclude that $M_{\delta}$ is also an SAYD module over $U(sl_2)$ with the same action and coaction due to the unimodularity of $sl_2$. The corresponding filtration is then given by
\begin{equation}
F_0M_{\delta} = F_0M \ot \mathbb{C}_{\delta} = \Cb\Big\langle R^X,R^Y,R^Z\Big\rangle \ot \mathbb{C}_{\delta}, \quad F_pM_{\delta} = M_{\delta}, p \geq 1.
\end{equation}

\nin We will first derive a Cartan type homotopy formula for Hopf cyclic cohomology, as in \cite{MoscRang07}. One notes that in \cite{MoscRang07} the SAYD module was one dimensional. We have to adapt the homotopy formula to fit our situation. To this end, let
\begin{equation}
D_Y:\mathcal{H} \to \mathcal{H}, \quad D_Y(h) := hY.
\end{equation}
Obviously, $D_Y$ is an $\mathcal{H}$-linear coderivation. Hence the operators
\begin{align}
\begin{split}
& \mathcal{L}_{D_Y}:C^n(\Uc \cl \Fc,M_{\delta}) \to C^n(\Uc \cl \Fc,M_{\delta}) \\
& \Lc_{D_Y}(m\ot_\Hc c^0\odots c^n)=\sum_{i=0}^{n}m\ot c^0_\Hc \odots D_Y(c_i) \odots c^n,
\end{split}
\end{align}
\begin{align}
\begin{split}
& e_{D_Y}:C^n(\Uc \cl \Fc,M_{\delta}) \to C^{n+1}(\Uc \cl \Fc,M_{\delta}) \\
& e_{D_Y}(m\ot_\Hc c^0\odots c^n)=(-1)^{n}m\pr{0}\ot_\Hc c^0\ps{2}\ot c^1\odots c^n\ot m\pr{-1}D_Y(c^0\ps{1}),
\end{split}
\end{align}
and
\begin{align}
\begin{split}
& E_{D_Y}:C^n(\Uc \cl \Fc,M_{\delta}) \to C^{n-1}(\Uc \cl \Fc,M_{\delta}) \\
&E_{D_Y}^{j,i}(m\ot_\Hc c^0\odots c^{n})=\\
&=(-1)^{n(i+1)}\epsilon(c^0)m\pr{0}\ot_\Hc c^{n-i+2}\odots c^{n+1}\ot
m\pr{-1}c_1\odots m\pr{-(j-i)}c^{j-i}\ot
\\
&\ot m\pr{-(j-i+1)}D_Y(c^{j-i+1})\ot m\pr{-(j-i+2)}c^{j-i+2}\odots
m\pr{-(n-i+1)}c^{n-i+1},
\end{split}
\end{align}
satisfy, by \cite[Proposition 3.7]{MoscRang07},
\begin{equation}
[E_{D_Y} + e_{D_Y}, b + B] = \mathcal{L}_{D_Y}.
\end{equation}

\nin We next obtain an analogous of \cite[Lemma 3.8]{MoscRang07}.

\begin{lemma}\label{aux-68}
We have
\begin{equation}
\mathcal{L}_{D_Y} = I - \widetilde{ad}Y,
\end{equation}
where
\begin{equation}
\widetilde{ad}Y(m_{\delta} \ot \widetilde{f} \ot \widetilde{u}) = m_{\delta} \ot \widetilde{ad}Y(\widetilde{f} \ot \widetilde{u}) - (m \cdot Y)_{\delta} \ot \widetilde{f} \ot \widetilde{u},
\end{equation}
and $m_{\delta} := m \ot 1_{\mathbb{C}}$.
\end{lemma}

\begin{proof}
Let us first recall the isomorphism
\begin{equation}
\Theta := \Phi_2 \circ \Phi_1 \circ \Psi:C^{\bullet}_{\mathcal{H}}(\Uc \cl \Fc,M_{\delta}) \to \FZ^{\bullet,\bullet}
\end{equation}
of cocyclic modules. By \cite{MoscRang09}, we know that
\begin{align}\label{aux-76}
\begin{split}
& \Psi(m_{\delta} \ot_{\mathcal{H}} u^0 \cl f^0 \ot \ldots \ot u^n \cl f^n) = \\
& m_{\delta} \ot_{\mathcal{H}} {u^0}^{\pr{-n-1}}f^0 \ot \ldots \ot {u^0}^{\pr{-1}} \ldots \ot {u^n}^{\pr{-1}}f^n \ot {u^0}^{\pr{0}} \ot \ldots \ot {u^n}^{\pr{0}} \\
& \Psi^{-1}(m_{\delta} \ot_{\mathcal{H}} f^0 \ot \ldots \ot f^n \ot u^0 \ot \ldots \ot u^n) = \\
& m_{\delta} \ot_{\mathcal{H}} {u^0}^{\pr{0}} \cl S^{-1}({u^0}^{\pr{-1}})f^0 \ot \ldots \ot {u^n}^{\pr{0}} \cl S^{-1}({u^0}^{\pr{-n-1}}{u^1}^{\pr{-n}} \ldots {u^n}^{\pr{-1}})f^n,
\end{split}
\end{align}

\begin{align}\label{aux-77}
\begin{split}
& \Phi_1(m_{\delta} \ot_{\mathcal{H}} f^0 \ot \ldots \ot f^n \ot u^0 \ot \ldots \ot u^n) = \\
& m_{\delta} \cdot u^0\ps{2} \ot_{\Fc} S^{-1}(u^0\ps{1}) \rhd (f^0 \ot \ldots \ot f^n) \ot S(u^0\ps{3}) \cdot (u^1 \ot \ldots \ot u^n) \\
& \Phi_1^{-1}(m_{\delta} \ot_{\Fc} f^0 \ot \ldots \ot f^n \ot u^1 \ot \ldots \ot u^n) = \\
& m_{\delta} \ot_{\mathcal{H}} f^0 \ot \ldots \ot f^n \ot 1_{U(\Fg_1)} \ot u^1 \ot \ldots \ot u^n,
\end{split}
\end{align}
and
\begin{align}\label{aux-78}
\begin{split}
& \Phi_2(m_{\delta} \ot_{\Fc} f^0 \ot \ldots \ot f^n \ot u^1 \ot \ldots \ot u^n) = \\
& m_{\delta} \cdot f^0\ps{1} \ot S(f^0\ps{2}) \cdot (f^1 \ot \ldots \ot f^n \ot u^1 \ot \ldots \ot u^n) \\
& \Phi_2^{-1}(m_{\delta} \ot f^1 \ot \ldots \ot f^n \ot u^1 \ot \ldots \ot u^n) = \\
& m_{\delta} \ot_{\Fc} 1_{\Fc} \ot f^1 \ot \ldots \ot f^n \ot u^1 \ot \ldots \ot u^n.
\end{split}
\end{align}

\nin Here, the left $\mathcal{H}$-coaction on $\Uc$ is the one corresponding to the right $\Fc$-coaction. Namely,
\begin{equation}\label{aux-65}
u^{\pr{-1}} \ot u^{\pr{0}} = S(u^{\pr{1}}) \ot u^{\pr{0}}.
\end{equation}

\nin We also recall that
\begin{align}
\begin{split}
& \Phi:\mathcal{H} = \Fc \acl \Uc \to \Uc \cl \Fc \\
& \Phi(f \acl u) = u^{\pr{0}} \cl fu^{\pr{1}}\\
& \Phi^{-1}(u \cl f) = fS^{-1}(u^{\pr{1}}) \acl u^{\pr{0}}.
\end{split}
\end{align}

\nin Therefore, we have
\begin{align}
\begin{split}
& \Theta \circ \mathcal{L}_{D_Y} \circ \Theta^{-1} (m_{\delta} \ot f^1 \ot \ldots \ot f^n \ot u^1 \ot \ldots \ot u^n) = \\
& \Theta \circ \mathcal{L}_{D_Y} \circ \Psi^{-1} \circ \Phi_1^{-1} \circ \Phi_2^{-1} (m_{\delta} \ot f^1 \ot \ldots \ot f^n \ot u^1 \ot \ldots \ot u^n) = \\
& \Theta \circ \mathcal{L}_{D_Y} \circ \Psi^{-1} \circ \Phi_1^{-1} (m_{\delta} \ot_{\Fc} 1_{\Fc} \ot f^1 \ot \ldots \ot f^n \ot u^1 \ot \ldots \ot u^n) = \\
& \Theta \circ \mathcal{L}_{D_Y} \circ \Psi^{-1} (m_{\delta} \ot_{\mathcal{H}} 1_{\Fc} \ot f^1 \ot \ldots \ot f^n \ot 1_{\Uc} \ot u^1 \ot \ldots \ot u^n) = \\
& \Theta \circ \mathcal{L}_{D_Y} (m_{\delta} \ot_{\mathcal{H}} 1_{\Uc} \cl 1_{\Fc} \ot {u^1}^{\pr{0}} \cl S^{-1}({u^1}^{\pr{-1}}) \rhd f^1 \ot \ldots \\
& \ldots \ot {u^n}^{\pr{0}} \cl S^{-1}({u^1}^{\pr{-n}} \ldots {u^n}^{\pr{-1}}) \rhd f^n) = \\
& \Theta \circ \mathcal{L}_{D_Y} (m_{\delta} \ot_{\mathcal{H}} 1_{\Uc} \cl 1_{\Fc} \ot {u^1}^{\pr{0}} \cl {u^1}^{\pr{1}}f^1 \ot \ldots \ot {u^n}^{\pr{0}} \cl {u^n}^{\pr{1}} \ldots {u^1}^{\pr{n}}f^n),
\end{split}
\end{align}
where on the last equality we have used \eqref{aux-65}. In order to apply $\mathcal{L}_{D_Y}$, we make the observation
\begin{align}
\begin{split}
& \Phi D_Y \Phi^{-1} (u \cl f) = \Phi (fS^{-1}(u^{\pr{1}}) \acl u^{\pr{0}}Y) = \\
& (u^{\pr{0}}Y)^{\pr{0}} \cl fS^{-1}(u^{\pr{1}})(u^{\pr{0}}Y)^{\pr{1}} = \\
& {u^{\pr{0}}\ps{1}}^{\pr{0}}Y^{\pr{0}} \cl fS^{-1}(u^{\pr{1}}){u^{\pr{0}}\ps{1}}^{\pr{1}}(u^{\pr{0}}\ps{2} \rhd Y^{\pr{1}}) = \\
& {u\ps{1}}^{\pr{0}}Y^{\pr{0}} \cl fS^{-1}({u\ps{1}}^{\pr{2}}{u\ps{2}}^{\pr{1}}){u\ps{1}}^{\pr{1}}({u\ps{2}}^{\pr{0}} \rhd Y^{\pr{1}}) = \\
& uY \cl f,
\end{split}
\end{align}
using the action-coaction compatibilities of a bicrossed product. Hence,
\begin{align}
\begin{split}
& \Theta \circ \mathcal{L}_{D_Y} (m_{\delta} \ot_{\mathcal{H}} 1_{\Uc} \cl 1_{\Fc} \ot {u^1}^{\pr{0}} \cl {u^1}^{\pr{1}}f^1 \ot \ldots \ot {u^n}^{\pr{0}} \cl {u^n}^{\pr{1}} \ldots {u^1}^{\pr{n}}f^n) = \\
& \Theta (m_{\delta} \ot_{\mathcal{H}} Y \cl 1_{\Fc} \ot {u^1}^{\pr{0}} \cl {u^1}^{\pr{1}}f^1 \ot \ldots \ot {u^n}^{\pr{0}} \cl {u^n}^{\pr{1}} \ldots {u^1}^{\pr{n}}f^n) + \\
& \sum_{i = 1}^n \Theta (m_{\delta} \ot_{\mathcal{H}} 1_{\Uc} \cl 1_{\Fc} \ot {u^1}^{\pr{0}} \cl {u^1}^{\pr{1}}f^1 \ot \ldots \\
& \ldots \ot {u^i}^{\pr{0}}Y \cl {u^i}^{\pr{1}} \ldots {u^1}^{\pr{i}}f^i \ot \ldots \ot {u^n}^{\pr{0}} \cl {u^n}^{\pr{1}} \ldots {u^1}^{\pr{n}}f^n).
\end{split}
\end{align}

\nin We notice,
\begin{align}
\begin{split}
& \Psi (m_{\delta} \ot_{\mathcal{H}} Y \cl 1_{\Fc} \ot {u^1}^{\pr{0}} \cl {u^1}^{\pr{1}}f^1 \ot \ldots \ot {u^n}^{\pr{0}} \cl {u^n}^{\pr{1}} \ldots {u^1}^{\pr{n}}f^n) = \\
& m_{\delta} \ot_{\mathcal{H}} Y^{\pr{-n-1}} \cdot 1_{\Fc} \ot Y^{\pr{-n}}({u^1}^{\pr{0}})^{\pr{-n}}{u^1}^{\pr{1}}f^1 \ot \ldots \\
& \ldots \ot Y^{\pr{-1}}({u^1}^{\pr{0}})^{\pr{-1}} \ldots ({u^n}^{\pr{0}})^{\pr{-1}}{u^n}^{\pr{1}} \ldots {u^1}^{\pr{n}}f^n \ot Y^{\pr{0}} \ot ({u^1}^{\pr{0}})^{\pr{0}} \ot \ldots \ot ({u^n}^{\pr{0}})^{\pr{0}} \\
& = m_{\delta} \ot_{\mathcal{H}} 1_{\Fc} \ot f^1 \ot \ldots \ot f^n \ot Y \ot u^1 \ot \ldots \ot u^n,
\end{split}
\end{align}
where on the last equality we have used \eqref{aux-65}. Similarly,

\begin{align}
\begin{split}
& \Psi (m_{\delta} \ot_{\mathcal{H}} 1_{\Uc} \cl 1_{\Fc} \ot {u^1}^{\pr{0}} \cl {u^1}^{\pr{1}}f^1 \ot \ldots \\
& \ldots \ot {u^i}^{\pr{0}}Y \cl {u^i}^{\pr{1}} \ldots {u^1}^{\pr{i}}f^i \ot \ldots \ot {u^n}^{\pr{0}} \cl {u^n}^{\pr{1}} \ldots {u^1}^{\pr{n}}f^n) = \\
& m_{\delta} \ot_{\mathcal{H}} 1_{\Fc} \ot f^1 \ot \ldots \ot f^n \ot 1_{\Uc} \ot u^1 \ot \ldots u^iY \ot \ldots \ot u^n.
\end{split}
\end{align}

\nin Therefore,
\begin{align}
\begin{split}
& \Theta (m_{\delta} \ot_{\mathcal{H}} Y \cl 1_{\Fc} \ot {u^1}^{\pr{0}} \cl {u^1}^{\pr{1}}f^1 \ot \ldots \ot {u^n}^{\pr{0}} \cl {u^n}^{\pr{1}} \ldots {u^1}^{\pr{n}}f^n) + \\
& \sum_{i = 1}^n \Theta (m_{\delta} \ot_{\mathcal{H}} 1_{\Uc} \cl 1_{\Fc} \ot {u^1}^{\pr{0}} \cl {u^1}^{\pr{1}}f^1 \ot \ldots \\
& \ldots \ot {u^i}^{\pr{0}}Y \cl {u^i}^{\pr{1}} \ldots {u^1}^{\pr{i}}f^i \ot \ldots \ot {u^n}^{\pr{0}} \cl {u^n}^{\pr{1}} \ldots {u^1}^{\pr{n}}f^n) = \\
& \Phi_2 \circ \Phi_1 (m_{\delta} \ot_{\mathcal{H}} 1_{\Fc} \ot \widetilde{f} \ot Y \ot \widetilde{u} + m_{\delta} \ot_{\mathcal{H}} 1_{\Fc} \ot \widetilde{f} \ot 1_{\Uc} \ot \widetilde{u} \cdot Y) = \\
& \Phi_2 (m_{\delta} \cdot Y\ps{2} \ot_{\Fc} S^{-1}(Y\ps{1}) \rhd (1_{\Fc} \ot \widetilde{f}) \ot S(Y\ps{3}) \cdot \widetilde{u} + m_{\delta} \ot_{\Fc} 1_{\Fc} \ot \widetilde{f} \ot \widetilde{u} \cdot Y).
\end{split}
\end{align}

\nin Considering he fact that $Y \in \mathcal{H}$ is primitive, and hence $adY(f) = [Y,f] = Y \rhd f$, we conclude
\begin{align}
\begin{split}
& \Phi_2 (m_{\delta} \cdot Y\ps{2} \ot_{\Fc} S^{-1}(Y\ps{1}) \rhd (1_{\Fc} \ot \widetilde{f}) \ot S(Y\ps{3}) \cdot \widetilde{u} + m_{\delta} \ot_{\Fc} 1_{\Fc} \ot \widetilde{f} \ot \widetilde{u} \cdot Y) = \\
& \Phi_2 (- m_{\delta} \ot_{\Fc} 1_{\Fc} \ot adY(\widetilde{f}) \ot \widetilde{u} - m_{\delta} \ot_{\Fc} 1_{\Fc} \ot \widetilde{f} \ot Y \cdot \widetilde{u} + \\
& m_{\delta} \cdot Y \ot_{\Fc} 1_{\Fc} \ot \widetilde{f} \ot \widetilde{u} + m_{\delta} \ot_{\Fc} 1_{\Fc} \ot \widetilde{f} \ot \widetilde{u} \cdot Y) = \\
& m_{\delta} \cdot Y \ot_{\Fc} \widetilde{f} \ot \widetilde{u} - m_{\delta} \ot_{\Fc} adY(\widetilde{f}) \ot \widetilde{u} - m_{\delta} \ot_{\Fc} \widetilde{f} \ot adY(\widetilde{u}).
\end{split}
\end{align}

\nin Finally, we recall $m_{\delta} \cdot Y = (m \cdot Y\ps{1})_{\delta}\delta(Y\ps{2}) = (m \cdot Y)_{\delta} + m_{\delta}$ to finish the proof.
\end{proof}

\begin{lemma}\label{aux-69}
The operator $\widetilde{ad}Y$ commutes with the horizontal operators \eqref{horizontal-operators} and the vertical operators \eqref{vertical-operators}.
\end{lemma}

\begin{proof}
We start with the horizontal operators. For the first horizontal coface, we have
\begin{align}
\begin{split}
& \hd_{0}(\widetilde{ad}Y(m_{\delta} \ot \widetilde{f} \ot \widetilde{u})) = \hd_0(m_{\delta} \ot adY(\widetilde{f} \ot \widetilde{u}) - (m \cdot Y)_{\delta} \ot \widetilde{f} \ot \widetilde{u}) = \\
& m_{\delta} \ot 1 \ot adY(\widetilde{f} \ot \widetilde{u}) - (m \cdot Y)_{\delta} \ot 1 \ot \widetilde{f} \ot \widetilde{u} = \\
& \widetilde{ad}Y (\hd_0(m_{\delta} \ot \widetilde{f} \ot \widetilde{u}))).
\end{split}
\end{align}
For  $\hd_{i}$ with $1 \leq i \leq n$, the commutativity  is a consequence of $adY \circ \D = \D \circ adY$ on $\Fc$. To see this, we notice that
\begin{align}
\begin{split}
& \D(adY(f)) = \D(Y \rhd f) = {Y\ps{1}}^{\pr{0}} \rhd f\ps{1} \ot {Y\ps{1}}^{\pr{1}}(Y\ps{2} \rhd f\ps{2}) \\
& = adY(f\ps{1}) \ot f\ps{2} + f\ps{1} \ot adY(f\ps{2}) = adY(\D(f)).
\end{split}
\end{align}

\nin For the commutation with the last horizontal coface operator, we proceed as follows. First we observe
\begin{align}
\begin{split}
& \widetilde{ad}Y(\hd_{n+1}(m_{\delta} \ot \widetilde{f} \ot \widetilde{u})) = \widetilde{ad}Y({m\pr{0}}_{\delta} \ot \widetilde{f} \ot \widetilde{u}^{\pr{-1}}m\pr{-1} \rhd 1_{\Fc} \ot \widetilde{u}^{\pr{0}}) = \\
& {m\pr{0}}_{\delta} \ot adY(\widetilde{f}) \ot [\widetilde{u}^{\pr{-1}}m\pr{-1}] \ot \widetilde{u}^{\pr{0}} + {m\pr{0}}_{\delta} \ot \widetilde{f} \ot adY(\widetilde{u}^{\pr{-1}}m\pr{-1} \rhd 1_{\Fc})\ot \widetilde{u}^{\pr{0}} \\
& + {m\pr{0}}_{\delta} \ot \widetilde{f} \ot [\widetilde{u}^{\pr{-1}}m\pr{0}] \ot adY(\widetilde{u}^{\pr{0}}) - (m\pr{0} \cdot Y)_{\delta} \ot \widetilde{f} \ot [\widetilde{u}^{\pr{-1}}m\pr{-1}] \ot \widetilde{u}^{\pr{0}}.
\end{split}
\end{align}

\nin Next, for any $h = g \acl u \in \mathcal{H}$ and $f \in \Fc$, on one hand we have
\begin{equation}
adY(h \rhd f) = adY(g(u \rhd f)) = adY(g)(u \rhd f) + g(Yu \rhd f),
\end{equation}
and on the other hand,
\begin{align*}
&adY(h) \rhd f = (adY(g) \acl u + g \acl Yu - g \acl uY) \rhd f \\
&\hspace{1.9cm}= adY(g)(u \rhd f) + g(Yu \rhd f) - g(uY \rhd f).
\end{align*}

\nin In other words,
\begin{equation}
adY(h \rhd f) = adY(h) \rhd f + h \rhd adY(f).
\end{equation}

\nin Therefore we have
\begin{align}
\begin{split}
& {m\pr{0}}_{\delta} \ot \widetilde{f} \ot adY(\widetilde{u}^{\pr{-1}}m\pr{-1} \rhd 1_{\Fc})\ot \widetilde{u}^{\pr{0}} = \\
& {m\pr{0}}_{\delta} \ot \widetilde{f} \ot adY(\widetilde{u}^{\pr{-1}})m\pr{-1} \rhd 1_{\Fc}\ot \widetilde{u}^{\pr{0}} + \\
& {m\pr{0}}_{\delta} \ot \widetilde{f} \ot \widetilde{u}^{\pr{-1}}adY(m\pr{-1}) \rhd 1_{\Fc}\ot \widetilde{u}^{\pr{0}}.
\end{split}
\end{align}

\nin Recalling \eqref{aux-65} and the coaction - multiplication compatibility on a bicrossed product, we observe that
\begin{align}
\begin{split}
& {adY(u)}^{\pr{-1}} \ot {adY(u)}^{\pr{0}} = S({adY(u)}^{\pr{1}}) \ot {adY(u)}^{\pr{0}} = \\
& S({Y\ps{1}}^{\pr{1}}(Y\ps{2} \rhd u^{\pr{1}})) \ot {Y\ps{1}}^{\pr{0}}u^{\pr{0}} - S({u\ps{1}}^{\pr{1}}(u\ps{2} \rhd Y^{\pr{1}})) \ot {u\ps{1}}^{\pr{0}}Y^{\pr{0}} = \\
& S(u^{\pr{1}}) \ot Yu^{\pr{0}} + S(Y \rhd u^{\pr{1}}) \ot u^{\pr{0}} - S(u^{\pr{1}}) \ot u^{\pr{0}}Y,
\end{split}
\end{align}
where we have used $Y^{\pr{0}} \ot Y^{\pr{1}} = Y \ot 1$. This follows from $[Z,Y] = Z$ implying $Z \rhd Y = 0$.

\medskip
\nin By \cite[Lemma 1.1]{MoscRang11} we also have $S(Y \rhd f) = Y \rhd S(f)$ for any $f \in \Fc$. Hence we can conclude
\begin{equation}\label{aux-72}
{adY(u)}^{\pr{-1}} \ot {adY(u)}^{\pr{0}} = adY(u^{\pr{-1}}) \ot u^{\pr{0}} + u^{\pr{-1}} \ot adY(u^{\pr{0}}),
\end{equation}
which implies immediately that,
\begin{equation}
{adY(\widetilde{u})}^{\pr{-1}} \ot {adY(\widetilde{u})}^{\pr{0}} = adY(\widetilde{u}^{\pr{-1}}) \ot \widetilde{u}^{\pr{0}} + \widetilde{u}^{\pr{-1}} \ot adY(\widetilde{u}^{\pr{0}}).
\end{equation}

\nin Finally, by the right-left AYD compatibility of $M$ over $\Hc$ we have
\begin{equation}
(m \cdot Y)\pr{-1} \ot (m \cdot Y)\pr{0} = m\pr{-1} \ot m\pr{0} \cdot Y - adY{m\pr{-1}} \ot m\pr{0}.
\end{equation}

\nin So, $\widetilde{ad}Y$ commutes with the last horizontal coface  $\hd_{n+1}$ as
\begin{align}
\begin{split}
& \widetilde{ad}Y(\hd_{n+1}(m_{\delta} \ot \widetilde{f} \ot \widetilde{u})) = {m\pr{0}}_{\delta} \ot adY(\widetilde{f}) \ot \widetilde{u}^{\pr{-1}}m\pr{-1} \rhd 1_{\Fc} \ot \widetilde{u}^{\pr{0}} + \\
& {m\pr{0}}_{\delta} \ot \widetilde{f} \ot {adY(\widetilde{u})}^{\pr{-1}}m\pr{-1} \rhd 1_{\Fc}\ot {adY(\widetilde{u})}^{\pr{0}} - \\
& {(m \cdot Y)\pr{0}}_{\delta} \ot \widetilde{f} \ot \widetilde{u}^{\pr{-1}}(m \cdot Y)\pr{0} \rhd 1_{\Fc} \ot \widetilde{u}^{\pr{0}} = \\
& \hd_{n+1}(\widetilde{ad}Y(m_{\delta} \ot \widetilde{f} \ot \widetilde{u})).
\end{split}
\end{align}

\nin It is immediate to observe the commutation $\sigma_j \circ \widetilde{ad}Y = \widetilde{ad}Y \circ \sigma_j$ with the horizontal degeneracy operators.

\medskip
\nin We now consider the horizontal cyclic operator. Let us first note that
\begin{equation}
m_{\delta} \cdot f = (m \cdot f)_{\delta}, \qquad f\in \Fc.
\end{equation}
We then have
\begin{align}
\begin{split}
& \widetilde{ad}Y(\hta(m_{\delta} \ot \widetilde{f} \ot \widetilde{u})) = \\
& \widetilde{ad}Y(({m\pr{0}} \cdot f^1\ps{1})_{\delta} \ot S(f^1\ps{2}) \cdot (f^2 \odots f^p \ot \widetilde{u}^{\pr{-1}}m\pr{-1} \rhd 1_{\Fc} \ot \widetilde{u}^{\pr{0}})) = \\
& ({m\pr{0}} \cdot f^1\ps{1})_{\delta} \ot adY(S(f^1\ps{2})) \cdot (f^2 \odots f^p \ot \widetilde{u}^{\pr{-1}}m\pr{-1} \rhd 1_{\Fc} \ot \widetilde{u}^{\pr{0}}) + \\
& ({m\pr{0}} \cdot f^1\ps{1})_{\delta} \ot S(f^1\ps{2}) \cdot (adY(f^2 \odots f^p) \ot \widetilde{u}^{\pr{-1}}m\pr{-1} \rhd 1_{\Fc} \ot \widetilde{u}^{\pr{0}}) + \\
& ({m\pr{0}} \cdot f^1\ps{1})_{\delta} \ot S(f^1\ps{2}) \cdot (f^2 \odots f^p \ot adY(\widetilde{u}^{\pr{-1}})m\pr{-1} \rhd 1_{\Fc} \ot \widetilde{u}^{\pr{0}}) + \\
& ({m\pr{0}} \cdot f^1\ps{1})_{\delta} \ot S(f^1\ps{2}) \cdot (f^2 \odots f^p \ot \widetilde{u}^{\pr{-1}}adY(m\pr{-1}) \rhd 1_{\Fc} \ot \widetilde{u}^{\pr{0}}) + \\
& ({m\pr{0}} \cdot f^1\ps{1})_{\delta} \ot S(f^1\ps{2}) \cdot (f^2 \odots f^p \ot \widetilde{u}^{\pr{-1}}m\pr{-1} \rhd 1_{\Fc} \ot adY(\widetilde{u}^{\pr{0}})) - \\
& (({m\pr{0}} \cdot f^1\ps{1}) \cdot Y)_{\delta} \ot S(f^1\ps{2}) \cdot (f^2 \odots f^p \ot \widetilde{u}^{\pr{-1}}m\pr{-1} \rhd 1_{\Fc} \ot \widetilde{u}^{\pr{0}}).
\end{split}
\end{align}

\nin Next, by the commutativity of $adY$ with the left $\mathcal{H}$-coaction on $\Uc$ as well as with the antipode, we can immediately conclude
\begin{align}
\begin{split}
& \widetilde{ad}Y(\hta(m_{\delta} \ot \widetilde{f} \ot \widetilde{u})) = \\
& ({m\pr{0}} \cdot f^1\ps{1})_{\delta} \ot S(adY(f^1\ps{2})) \cdot (f^2 \odots f^p \ot \widetilde{u}^{\pr{-1}}m\pr{-1} \rhd 1_{\Fc} \ot \widetilde{u}^{\pr{0}}) + \\
& ({m\pr{0}} \cdot f^1\ps{1})_{\delta} \ot S(f^1\ps{2}) \cdot (adY(f^2 \odots f^p) \ot \widetilde{u}^{\pr{-1}}m\pr{-1} \rhd 1_{\Fc} \ot \widetilde{u}^{\pr{0}}) + \\
& ({m\pr{0}} \cdot f^1\ps{1})_{\delta} \ot S(f^1\ps{2}) \cdot (f^2 \odots f^p \ot {adY(\widetilde{u})}^{\pr{-1}}m\pr{-1} \rhd 1_{\Fc} \ot {adY(\widetilde{u})}^{\pr{0}}) + \\
& ({m\pr{0}} \cdot f^1\ps{1})_{\delta} \ot S(f^1\ps{2}) \cdot (f^2 \odots f^p \ot \widetilde{u}^{\pr{-1}}adY(m\pr{-1}) \rhd 1_{\Fc} \ot \widetilde{u}^{\pr{0}}) - \\
& (({m\pr{0}} \cdot f^1\ps{1}) \cdot Y)_{\delta} \ot S(f^1\ps{2}) \cdot (f^2 \odots f^p \ot \widetilde{u}^{\pr{-1}}m\pr{-1} \rhd 1_{\Fc} \ot \widetilde{u}^{\pr{0}}).
\end{split}
\end{align}

\nin Then by the module compatibility over the bicrossed product $\mathcal{H} = \Fc \acl \Uc$, we have
\begin{equation}
(m \cdot Y) \cdot f = (m \cdot f) \cdot Y + m \cdot adY(f).
\end{equation}
Therefore,
\begin{align}
\begin{split}
& \widetilde{ad}Y(\hta(m_{\delta} \ot \widetilde{f} \ot \widetilde{u})) = \\
& ({m\pr{0}})_{\delta} \cdot adY(f^1\ps{1}) \ot S(f^1\ps{2}) \cdot (f^2 \odots f^p \ot \widetilde{u}^{\pr{-1}}m\pr{-1} \rhd 1_{\Fc} \ot \widetilde{u}^{\pr{0}}) + \\
& {m\pr{0}}_{\delta} \cdot f^1\ps{1} \ot S(adY(f^1\ps{2})) \cdot (f^2 \odots f^p \ot \widetilde{u}^{\pr{-1}}m\pr{-1} \rhd 1_{\Fc} \ot \widetilde{u}^{\pr{0}}) + \\
& {m\pr{0}}_{\delta} \cdot f^1\ps{1} \ot S(f^1\ps{2}) \cdot (adY(f^2 \odots f^p) \ot \widetilde{u}^{\pr{-1}}m\pr{-1} \rhd 1_{\Fc} \ot \widetilde{u}^{\pr{0}}) + \\
& {m\pr{0}}_{\delta} \cdot f^1\ps{1} \ot S(f^1\ps{2}) \cdot (f^2 \odots f^p \ot {adY(\widetilde{u})}^{\pr{-1}}m\pr{-1} \rhd 1_{\Fc} \ot {adY(\widetilde{u})}^{\pr{0}}) - \\
& {(m \cdot Y)\pr{0}}_{\delta} \cdot f^1\ps{1} \ot S(f^1\ps{2}) \cdot (f^2 \odots f^p \ot \widetilde{u}^{\pr{-1}}(m \cdot Y)\pr{-1} \rhd 1_{\Fc} \ot \widetilde{u}^{\pr{0}}).
\end{split}
\end{align}

\nin Finally, by the commutativity $adY \circ \D = \D \circ adY$ on $\Fc$ we finish as
\begin{align}
\begin{split}
& \widetilde{ad}Y(\hta(m_{\delta} \ot \widetilde{f} \ot \widetilde{u})) = \\
& {m\pr{0}}_{\delta} \cdot {ad(f^1)}\ps{1} \ot S({ad(f^1)}\ps{2}) \cdot (f^2 \odots f^p \ot \widetilde{u}^{\pr{-1}}m\pr{-1} \rhd 1_{\Fc} \ot \widetilde{u}^{\pr{0}}) + \\
& {m\pr{0}}_{\delta} \cdot f^1\ps{1} \ot S(f^1\ps{2}) \cdot (adY(f^2 \odots f^p) \ot \widetilde{u}^{\pr{-1}}m\pr{-1} \rhd 1_{\Fc} \ot \widetilde{u}^{\pr{0}}) + \\
& {m\pr{0}}_{\delta} \cdot f^1\ps{1} \ot S(f^1\ps{2}) \cdot (f^2 \odots f^p \ot {adY(\widetilde{u})}^{\pr{-1}}m\pr{-1} \rhd 1_{\Fc} \ot {adY(\widetilde{u})}^{\pr{0}}) - \\
& {(m \cdot Y)\pr{0}}_{\delta} \cdot f^1\ps{1} \ot S(f^1\ps{2}) \cdot (f^2 \odots f^p \ot \widetilde{u}^{\pr{-1}}(m \cdot Y)\pr{-1} \rhd 1_{\Fc} \ot \widetilde{u}^{\pr{0}}) \\
& = \hta(\widetilde{ad}Y(m_{\delta} \ot \widetilde{f} \ot \widetilde{u})).
\end{split}
\end{align}

\nin We continue with the vertical operators. We see that
\begin{equation}
\vd_i \circ \widetilde{ad}Y = \widetilde{ad}Y \circ \vd_i, \quad 0 \leq i \leq n
\end{equation}
are similar to their  horizontal counterparts. One notes that this time the commutativity $adY \circ \D = \D \circ adY$ on $\Uc$ is needed.

\medskip
\nin Commutativity with the last vertical coface operator follows, similarly as the horizontal case, from the AYD compatibility on $M$ over $\Hc$. Indeed,
\begin{align}
\begin{split}
& \widetilde{ad}Y(\vd_{n+1}(m_{\delta} \ot \widetilde{f} \ot \widetilde{u})) = \widetilde{ad}Y({m\pr{0}}_{\delta} \ot \widetilde{f} \ot \widetilde{u} \ot \wbar{m\pr{-1}}) = \\
& {m\pr{0}}_{\delta} \ot adY(\widetilde{f} \ot \widetilde{u}) \ot \wbar{m\pr{-1}} + {m\pr{0}}_{\delta} \ot \widetilde{f} \ot \widetilde{u} \ot \wbar{adY(m\pr{-1})} \\
& - {(m\pr{0} \cdot Y)}_{\delta} \ot \widetilde{f} \ot \widetilde{u} \ot \wbar{m\pr{-1}} = \\
& {m\pr{0}}_{\delta} \ot adY(\widetilde{f} \ot \widetilde{u}) \ot \wbar{m\pr{-1}} - {(m \cdot Y)\pr{0}}_{\delta} \ot \widetilde{f} \ot \widetilde{u} \ot \wbar{(m \cdot Y)\pr{-1}} = \\
& \vd_{n+1}(\widetilde{ad}Y(m_{\delta} \ot \widetilde{f} \ot \widetilde{u})).
\end{split}
\end{align}

\nin Finally, we show the commutativity of $\widetilde{ad}Y$  with the vertical cyclic operator. First, we notice that we can rewrite it as
\begin{align}
\begin{split}
&\vta(m_{\delta} \ot \widetilde{f} \ot \widetilde{u}) = ({m\pr{0}}_{\delta} \cdot u^1\ps{4}) \cdot S^{-1}(u^1\ps{3} \rhd 1_{\Fc}) \, \ot \\
& S(S^{-1}(u^1\ps{2}) \rhd 1_{\Fc}) \cdot \left( S^{-1}(u^1\ps{1}) \rhd \widetilde{f} \ot S(u^1\ps{5}) \cdot (u^2 \odots u^q \ot \wbar{m\pr{-1}})
\right) \\
& = {m\pr{0}}_{\delta} \cdot u^1\ps{2} \ot S^{-1}(u^1\ps{1}) \rhd \widetilde{f} \ot S(u^1\ps{3}) \cdot (u^2 \odots u^q \ot \wbar{m\pr{-1}}) \\
& = {(m\pr{0} \cdot u^1\ps{3})}_{\delta}\delta(u^1\ps{2}) \ot S^{-1}(u^1\ps{1}) \rhd \widetilde{f} \ot S(u^1\ps{4}) \cdot (u^2 \odots u^q \ot \wbar{m\pr{-1}}).
\end{split}
\end{align}

\nin Therefore we have
\begin{align}
\begin{split}
& \widetilde{ad}Y(\vta(m_{\delta} \ot \widetilde{f} \ot \widetilde{u})) = \\
& \widetilde{ad}Y({m\pr{0}}_{\delta} \cdot u^1\ps{2} \ot S^{-1}(u^1\ps{1}) \rhd \widetilde{f} \ot S(u^1\ps{3}) \cdot (u^2 \odots u^q \ot \wbar{m\pr{-1}})) = \\
& {m\pr{0}}_{\delta} \cdot u^1\ps{2} \ot adY(S^{-1}(u^1\ps{1}) \rhd \widetilde{f}) \ot S(u^1\ps{3}) \cdot (u^2 \odots u^q \ot \wbar{m\pr{-1}}) + \\
& {m\pr{0}}_{\delta} \cdot u^1\ps{2} \ot S^{-1}(u^1\ps{1}) \rhd \widetilde{f} \ot adY(S(u^1\ps{3})) \cdot (u^2 \odots u^q \ot \wbar{m\pr{-1}}) + \\
& {m\pr{0}}_{\delta} \cdot u^1\ps{2} \ot S^{-1}(u^1\ps{1}) \rhd \widetilde{f} \ot S(u^1\ps{3}) \cdot (adY(u^2 \odots u^q) \ot \wbar{m\pr{-1}}) + \\
& {m\pr{0}}_{\delta} \cdot u^1\ps{2} \ot S^{-1}(u^1\ps{1}) \rhd \widetilde{f} \ot S(u^1\ps{3}) \cdot (u^2 \odots u^q \ot \wbar{adY(m\pr{-1})}) - \\
& {(m\pr{0} \cdot u^1\ps{3}Y)}_{\delta}\delta(u^1\ps{2}) \ot S^{-1}(u^1\ps{1}) \rhd \widetilde{f} \ot S(u^1\ps{4}) \cdot (u^2 \odots u^q \ot \wbar{m\pr{-1}}).
\end{split}
\end{align}

\nin Recalling that
\begin{equation}\label{aux-71}
adY(h \rhd f) = adY(h) \rhd f + h \rhd adY(f),
\end{equation}
we then straightforwardly extend it to
\begin{equation}
adY(h \rhd \widetilde{f}) = adY(h) \rhd \widetilde{f} + h \rhd adY(\widetilde{f}).
\end{equation}

\nin As a result, we have
\begin{align}
\begin{split}
{m\pr{0}}_{\delta} \cdot u^1\ps{2} \ot adY(S^{-1}(u^1\ps{1}) \rhd \widetilde{f}) \ot S(u^1\ps{3}) \cdot (u^2 \odots u^q \ot \wbar{m\pr{-1}}) = \\
{m\pr{0}}_{\delta} \cdot u^1\ps{2} \ot S^{-1}(adY(u^1\ps{1})) \rhd \widetilde{f} \ot S(u^1\ps{3}) \cdot (u^2 \odots u^q \ot \wbar{m\pr{-1}}) + \\
{m\pr{0}}_{\delta} \cdot u^1\ps{2} \ot S^{-1}(u^1\ps{1}) \rhd adY(\widetilde{f}) \ot S(u^1\ps{3}) \cdot (u^2 \odots u^q \ot \wbar{m\pr{-1}}).
\end{split}
\end{align}

\nin Next, we observe that
\begin{align}
\begin{split}
& - {(m\pr{0} \cdot u^1\ps{3}Y)}_{\delta}\delta(u^1\ps{2}) \ot S^{-1}(u^1\ps{1}) \rhd \widetilde{f} \ot S(u^1\ps{4}) \cdot (u^2 \odots u^q \ot \wbar{m\pr{-1}}) = \\
& {(m\pr{0} \cdot adY(u^1\ps{3}))}_{\delta}\delta(u^1\ps{2}) \ot S^{-1}(u^1\ps{1}) \rhd \widetilde{f} \ot S(u^1\ps{4}) \cdot (u^2 \odots u^q \ot \wbar{m\pr{-1}}) -\\
& {(m\pr{0} \cdot Y)}_{\delta} \cdot u^1\ps{2} \ot S^{-1}(u^1\ps{1}) \rhd \widetilde{f} \ot S(u^1\ps{3}) \cdot (u^2 \odots u^q \ot \wbar{m\pr{-1}}),
\end{split}
\end{align}
where,
\begin{equation}\label{aux-73}
(m \cdot adY(u\ps{2}))_{\delta} \delta(u\ps{1}) = m_{\delta} \cdot adY(u).
\end{equation}

\nin Therefore we have
\begin{align}
\begin{split}
& \widetilde{ad}Y(\vta(m_{\delta} \ot \widetilde{f} \ot \widetilde{u})) = \\
& {m\pr{0}}_{\delta} \cdot u^1\ps{2} \ot S^{-1}(u^1\ps{1}) \rhd adY(\widetilde{f}) \ot S(u^1\ps{3}) \cdot (u^2 \odots u^q \ot \wbar{m\pr{-1}}) + \\
& {m\pr{0}}_{\delta} \cdot u^1\ps{2} \ot S^{-1}(adY(u^1\ps{1})) \rhd \widetilde{f} \ot S(u^1\ps{3}) \cdot (u^2 \odots u^q \ot \wbar{m\pr{-1}}) + \\
& {m\pr{0}}_{\delta} \cdot adY(u^1\ps{2}) \ot S^{-1}(u^1\ps{1}) \rhd \widetilde{f} \ot S(u^1\ps{4}) \cdot (u^2 \odots u^q \ot \wbar{m\pr{-1}}) + \\
& {m\pr{0}}_{\delta} \cdot u^1\ps{2} \ot S^{-1}(u^1\ps{1}) \rhd \widetilde{f} \ot S(adY(u^1\ps{3})) \cdot (u^2 \odots u^q \ot \wbar{m\pr{-1}}) + \\
& {m\pr{0}}_{\delta} \cdot u^1\ps{2} \ot S^{-1}(u^1\ps{1}) \rhd \widetilde{f} \ot S(u^1\ps{3}) \cdot (adY(u^2 \odots u^q) \ot \wbar{m\pr{-1}}) - \\
& {(m \cdot Y)\pr{0}}_{\delta} \cdot u^1\ps{2} \ot S^{-1}(u^1\ps{1}) \rhd \widetilde{f} \ot S(u^1\ps{3}) \cdot (u^2 \odots u^q \ot \wbar{(m \cdot Y)\pr{-1}}).
\end{split}
\end{align}

\nin Then the commutativity $\D \circ adY = adY \circ \D$ on $\Uc$ finishes the proof as
\begin{align}
\begin{split}
& \widetilde{ad}Y(\vta(m_{\delta} \ot \widetilde{f} \ot \widetilde{u})) = \\
& {m\pr{0}}_{\delta} \cdot u^1\ps{2} \ot S^{-1}(u^1\ps{1}) \rhd adY(\widetilde{f}) \ot S(u^1\ps{3}) \cdot (u^2 \odots u^q \ot \wbar{m\pr{-1}}) + \\
& {m\pr{0}}_{\delta} \cdot adY(u^1)\ps{2} \ot S^{-1}(adY(u^1)\ps{1}) \rhd \widetilde{f} \ot S(adY(u^1)\ps{3}) \cdot (u^2 \odots u^q \ot \wbar{m\pr{-1}}) + \\
& {m\pr{0}}_{\delta} \cdot u^1\ps{2} \ot S^{-1}(u^1\ps{1}) \rhd \widetilde{f} \ot S(u^1\ps{3}) \cdot (adY(u^2 \odots u^q) \ot \wbar{m\pr{-1}}) - \\
& {(m \cdot Y)\pr{0}}_{\delta} \cdot u^1\ps{2} \ot S^{-1}(u^1\ps{1}) \rhd \widetilde{f} \ot S(u^1\ps{3}) \cdot (u^2 \odots u^q \ot \wbar{(m \cdot Y)\pr{-1}}) \\
& = \vta(\widetilde{ad}Y(m_{\delta} \ot \widetilde{f} \ot \widetilde{u})).
\end{split}
\end{align}
\end{proof}

\nin For the generators $X,Y,\delta_1 \in \mathcal{H}$, it is already known that
\begin{equation}
adY(Y) = 0, \quad adY(X) = X, \quad adY(\delta_1) = \delta_1.
\end{equation}

\nin We recall here the action of $Y \in sl_2$ as
\begin{equation}
1_M \lhd Y = 0, \quad R^X \lhd Y = R^X, \quad R^Y \lhd Y = 0, \quad R^Z \lhd Y = -R^Z.
\end{equation}

\nin Hence we define the following weight on the cyclic complex by
\begin{align}
\begin{split}
& \nm{Y} = 0, \quad \nm{X} = 1, \quad \nm{\delta_1} = 1, \\
& \nm{1_M} = 0, \quad \nm{R^X} = -1, \quad \nm{R^Y} = 0, \quad \nm{R^Z} = 1,
\end{split}
\end{align}
we can express the following property of the operator $\widetilde{ad}Y$;
\begin{equation}
\widetilde{ad}Y(m_{\delta} \ot \widetilde{f} \ot \widetilde{u}) = \nm{m_{\delta} \ot \widetilde{f} \ot \widetilde{u}}m_{\delta} \ot \widetilde{f} \ot \widetilde{u},
\end{equation}
where $\nm{m_{\delta} \ot \widetilde{f} \ot \widetilde{u}} := \nm{m} + \nm{\widetilde{f}} + \nm{\widetilde{u}}$.

\medskip
\nin Hence, the operator $\widetilde{ad}Y$ acts as a grading (weight) operator. Extending the above grading to the cocyclic complex $\FZ^{\bullet, \bullet}$, we have
\begin{equation}
\FZ^{\bullet, \bullet} = \bigoplus_{k \in \mathbb{Z}}\FZ[k]^{\bullet, \bullet},
\end{equation}
where
\begin{equation}
\FZ[k] = \Big\{m_{\delta} \ot \widetilde{f} \ot \widetilde{u} \, \Big| \, \nm{m_{\delta} \ot \widetilde{f} \ot \widetilde{u}}= k\Big\}.
\end{equation}

\nin As a result of Lemma \ref{aux-69}, we can say that $\FZ[k]$ is a subcomplex for any $k \in \mathbb{Z}$, and hence the cohomology inherits the grading. Namely,
\begin{equation}
HP(\mathcal{H},M_{\delta}) = \bigoplus_{k \in \mathbb{Z}}H(\FZ[k]).
\end{equation}

\nin Moreover, using Lemma \ref{aux-68} we conclude the following analogous of Corollary 3.10 in \cite{MoscRang07}.

\begin{corollary}\label{corollary-weight-1}
The cohomology is captured by the weight 1 subcomplex, \ie
\begin{equation}
H(\FZ[1]) = HP(\mathcal{H},M_{\delta}), \quad H(\FZ[k]) = 0, \quad k \neq 1.
\end{equation}
\end{corollary}

\begin{proposition}
The odd and even periodic Hopf cyclic cohomology of $\Hc_{1\rm S }\cop$ with coefficients in $M_\d$ are both one dimensional. Their classes approximately are given by the following cocycles in the  $E_1$ term of the natural spectral sequence associated to $M_\d$.
\begin{align}
&c^{\rm odd}=\one \ot \d_1 \in E_1^{1,{\rm odd}}\\
&c^{\rm even} = \one \ot X \ot Y - \one \ot Y \ot X + \one \ot Y \ot \delta_1Y \in E_1^{1,{\rm even}}.
\end{align}
Here, $\one := 1_M \ot \Cb_\d$.
\end{proposition}
\begin{proof}
We have seen that all cohomology classes are concentrated in the weight 1 subcomplex.
On the other hand, $E_1$ term of the spectral sequence associated to the above mentioned filtration on $M_{\delta}$ is
\begin{equation}
E_1^{j,i}(\mathcal{H},M_{\delta}) = H^{i+j}(C(\Uc \cl \Fc, F_jM_{\delta} / F_{j-1}M_{\delta})),
\end{equation}
where $F_{0}M_{\delta} / F_{-1}M_{\delta} \cong F_{0}M_{\delta}$, $F_{1}M_{\delta} / F_0M_{\delta} \cong \mathbb{C}_{\delta}$ and $F_{j+1}M_{\delta} / F_jM_{\delta} = 0$ for $j \geq 1$.

\medskip
\nin Therefore,
\begin{equation}
E_1^{0,i}(\mathcal{H},M_{\delta}) = 0, \quad E_1^{1,i}(\mathcal{H},M_{\delta}) = H^i(C(\Uc \cl \Fc, \mathbb{C}_{\delta})), \quad E_1^{j,i}(\mathcal{H},M_{\delta}) = 0, \quad j \geq 1.
\end{equation}

\nin So the spectral sequence collapses at the $E_2$ term and we get
\begin{equation}
E_2^{0,i}(\mathcal{H},M_{\d}) \cong E_{\infty}^{0,i}(\mathcal{H},M_{\d}) = 0,
\end{equation}
\begin{equation}\label{aux-70}
E_2^{1,i}(\mathcal{H}, M_\d) \cong E_{\infty}^{1,i}(\mathcal{H}) = F_1H^i(C(\Uc \cl \Fc, M_{\delta})) / F_0H^i(C(\Uc \cl \Fc, M_{\delta})),
\end{equation}
and
\begin{equation}
E_2^{j,i}(\mathcal{H}, M_\d) \cong E_{\infty}^{j,i}(\mathcal{H},M_\d) = 0, \quad j \geq 2.
\end{equation}

\nin By definition of the induced filtration on the cohomology groups, we have
\begin{align}
\begin{split}
& F_1H^i(C(\Uc \cl \Fc, M_{\delta})) = H^i(C(\Uc \cl \Fc, F_1M_{\delta})) = \\
& H^i(C(\Uc \cl \Fc, M_{\delta})),
\end{split}
\end{align}
and
\begin{align}
\begin{split}
& F_0H^i(C(\Uc \cl \Fc, M_{\delta})) = H^i(C(\Uc \cl \Fc, F_0M_{\delta})) \cong \\
& H^i(W(sl_2, F_0M)) = 0,
\end{split}
\end{align}
where the last equality follows from the Whitehead's theorem.
\end{proof}

\subsubsection{Construction of a representative cocycle for the odd class}
In this subsection we first compute the odd cocycle in the total complex    ${\rm{Tot}}^{\bullet}(\Fc,\Uc,M_\d)$  of the bicomplex  \eqref{bicocyclic-bicomplex}.
Let us  recall the total mixed complex
\begin{equation}
{\rm{Tot}}^{\bullet}(\Fc,\Uc,M_\d):=\bigoplus_{p+q = \bullet}M_\d \ot \Fc^{\ot\;p} \ot \Uc^{\ot\;q},
\end{equation}
with the operators
\begin{align}
& \hb_p=\sum_{i=0}^{p+1} (-1)^{i}\hd_i, \qquad \vb_q=\sum_{i=0}^{q+1} (-1)^{i}\vd_i, \qquad b_T=\sum_{p+q=n}\hb_p +(-1)^p \vb_q \\
& \hB_p=(\sum_{i=0}^{p-1}(-1)^{(p-1)i}\hta^i)\overset{\ra}{\sigma}_{p-1}\hta, \quad \vB_q=(\sum_{i=0}^{q-1}(-1)^{(q-1)i}\vta^i)\uparrow\hspace{-4pt}\sigma_{q-1}\vta, \\\notag
& \hspace{3cm} \quad B_T=\sum_{p+q=n}\hB_p+(-1)^p\vB_q.
\end{align}

\begin{proposition}
Let
\begin{equation}
c':=\one \ot \d_1 \in M_\d \ot \Fc
\end{equation}
and
\begin{equation}
c''':= \bfR^Y \ot X + 2\bfR^Z \ot Y \in M_\d \ot \Uc.
\end{equation}
Then $c'+c''' \in {\rm{Tot}}^1(\Fc,\Uc,M_\d)$ is a Hochschild cocycle.
\end{proposition}

\begin{proof}
We start with the element $
c':=\one \ot \d_1 \in M_\d \ot \Fc .$

\medskip
\nin The equality $\vb(c')=0$ is immediate to notice. Next, we observe that
\begin{align}
\begin{split}
& \hb(c') = -\bfR^X \ot \d_1 \ot X - \bfR^Y \ot \d_1 \ot Y \\
& = -\bfR^X \ot \d_1 \ot X + \bfR^Y \ot \d_1 \ot Y + \bfR^X \ot {\d_1}^2 \ot Y - \bfR^X \ot {\d_1}^2 \ot Y - 2\bfR^Y \ot \d_1 \ot Y \\
& = \vb(\bfR^Y \ot X+ 2\bfR^Z \ot Y).
\end{split}
\end{align}
So, for the element $c''':=\bfR^Y \ot X + 2\bfR^Z \ot Y \in M_\d \ot \Uc$, we have $\hb(c') - \vb(c''') = 0$. Finally we notice $\hb(c''')=0$.
\end{proof}

\begin{proposition}
The element $c'+c''' \in {\rm{Tot}}^1(\Fc,\Uc,M_\d)$ is a Connes cycle.
\end{proposition}

\begin{proof}
Using the action of $\Fc$ and $\Uc$ on $M_\d$, we directly conclude that on one hand side we have $\vB(c') = \bfR^Z$, and on the other hand  $\hB(c''') = -\bfR^Z$.
\end{proof}

\bigskip

\nin Our next task is to send this cocycle to the cyclic complex $C^1(\Hc,M_\d)$. This is a two step process. We first use the Alexander-Whitney map
\begin{equation}\label{AW}
\, AW:= \bigoplus_{p+q=n} AW_{p,q}: \Tot^n(\Fc,\Uc,M_\d)\ra \FZ^{n,n},
\end{equation}
\begin{equation*}
AW_{p,q}: \Fc^{\ot p}\ot  \Uc^{\ot q}\longrightarrow \Fc^{\ot
p+q}\ot  \Uc^{\ot p+q}
\end{equation*}
\begin{equation*}
AW_{p,q}=(-1)^{p+q}\underset{
p\;\text{times}}{\underbrace{\vd_0\vd_0\dots
\vd_0}}\hd_n\hd_{n-1}\dots \hd_{p+1} \, .
\end{equation*}
to pass to the diagonal complex $\FZ^{\bullet,\bullet}(\Hc,\Fc,M_\d)$. It is checked that
\begin{equation}
AW_{1,0}(c') = - \one \ot \d_1 \ot 1 - \bfR^X \ot \d_1 \ot X - \bfR^Y \ot \d_1 \ot Y,
\end{equation}
as well as
\begin{equation}
AW_{0,1}(c''') = - \bfR^Y \ot 1 \ot X - 2\bfR^Z \ot 1 \ot Y.
\end{equation}

\nin Summing them up we get
\begin{equation}
c^{\rm odd}_{\rm diag} := - \one \ot \d_1 \ot 1 - \bfR^X \ot \d_1 \ot X - \bfR^Y \ot \d_1 \ot Y - \bfR^Y \ot 1 \ot X - 2\bfR^Z \ot 1 \ot Y.
\end{equation}

\nin Finally, via the quasi-isomorphism
 \begin{align}\label{PSI-1}
 \begin{split}
& \Psi: \FZ^{\bullet,\bullet}(\Hc,\Fc,M_\d) \longrightarrow C^\bullet(\Hc, M_\d) \\
&\Psi(m\ot f^1\odots f^n \ot u^1\odots u^n)\\
&=\sum m\ot f^1\acl u^1\ns{0}\ot f^2u^1\ns{1}\acl u^2\ns{0}\odots \\
& \odots f^n u^1\ns{n-1} \dots u^{n-1}\ns{1}\acl u^n ,
 \end{split}
\end{align}
which is recalled from \cite{RangSutl},
we carry the element $c^{\rm odd}_{\rm diag} \in \FZ^{2,2}(\Hc,\Fc,M_\d)$ to
\begin{equation}\label{c-odd}
c^{\rm odd} = - \Big(\one \ot \delta_1 + \bfR^Y \ot X + \bfR^X \ot \delta_1X + \bfR^Y \ot \delta_1Y + 2 \bfR^Z \ot Y \Big) \in C^1(\Hc, M_\d).
\end{equation}

\begin{proposition}
The element $c^{\rm odd}$ defined in \eqref{c-odd} is a Hochschild cocycle.
\end{proposition}

\begin{proof}
We first calculate its images under the Hochschild coboundary $b:C^1(\mathcal{H},M_{\delta}) \to C^2(\mathcal{H},M_{\delta})$.
\begin{align}
\begin{split}
& b(\one \ot \delta_1) = \one \ot 1_{\mathcal{H}} \ot \delta_1 - \one \ot \D(\delta_1) + \one \ot \delta_1 \ot 1 + \bfR^Y \ot \delta_1 \ot Y + \bfR^X \ot \delta_1 \ot X \\
& = \bfR^Y \ot \delta_1 \ot Y + \bfR^X \ot \delta_1 \ot X, \\[.2cm]
& b(\bfR^Y \ot X) = \bfR^Y \ot 1_{\mathcal{H}} \ot X - \bfR^Y \ot \D(X) + \bfR^Y \ot X \ot 1_{\mathcal{H}} + \bfR^X \ot X \ot \delta_1 \\
& = \bfR^X \ot X \ot \delta_1 - \bfR^Y \ot Y \ot \delta_1, \\[.2cm]
& b(\bfR^X \ot \delta_1X) = \bfR^X \ot 1_{\mathcal{H}} \ot \delta_1X - \bfR^X \ot \D(\delta_1X) + \bfR^X \ot \delta_1X \ot 1_{\mathcal{H}} \\
& = - \bfR^X \ot \delta_1 \ot X - \bfR^X \ot \delta_1Y \ot \delta_1 - \bfR^X \ot X \ot \delta_1 - \bfR^X \ot Y \ot {\delta_1}^2, \\[.2cm]
& b(\bfR^Y \ot \delta_1Y) = \bfR^Y \ot 1_{\mathcal{H}} \ot \delta_1Y - \bfR^Y \ot \D(\delta_1Y) + \bfR^Y \ot \delta_1Y \ot 1_{\mathcal{H}} + \bfR^X \ot \delta_1Y \ot \delta_1 \\
& = \bfR^X \ot \delta_1Y \ot \delta_1 - \bfR^Y \ot \delta_1 \ot Y - \bfR^Y \ot Y \ot \delta_1, \\[.2cm]
& b(\bfR^Z \ot Y) = \bfR^Z \ot 1_{\mathcal{H}} \ot Y - \bfR^Z \ot \D(Y) + \bfR^Z \ot Y \ot 1_{\mathcal{H}} \\
& + \bfR^Y \ot Y \ot \delta_1 + \frac{1}{2} \bfR^X \ot Y \ot {\delta_1}^2 \\
& = \bfR^Y \ot Y \ot \delta_1 + \frac{1}{2} \bfR^X \ot Y \ot {\delta_1}^2.
\end{split}
\end{align}

\nin Now, summing up we get
\begin{equation}
b(\one \ot \delta_1 + \bfR^Y \ot X + \bfR^X \ot \delta_1X + \bfR^Y \ot \delta_1Y + 2 \cdot \bfR^Z \ot Y) = 0.
\end{equation}
\end{proof}

\begin{proposition}
The Hochschild cocycle $c^{\rm odd}$ defined in \eqref{c-odd} vanishes under the Connes boundary map.
\end{proposition}

\begin{proof}
The Connes boundary is defined on the normalized bi-complex by the formula
\begin{equation}
B = \sum_{i = 0}^n (-1)^{ni}\tau^i \circ \sigma_{-1},
\end{equation}
where
\begin{equation}
\sigma_{-1}(m_{\delta} \ot h^1 \ot \ldots \ot h^{n+1}) = m_{\delta} \cdot h^1\ps{1} \ot S(h^1\ps{2}) \cdot (h^2 \ot \ldots \ot h^{n+1})
\end{equation}
is the extra degeneracy. Accordingly,
\begin{align}
\begin{split}
& B(\one \ot \delta_1 + \bfR^Y \ot X + \bfR^X \ot \delta_1X + \bfR^Y \ot \delta_1Y + 2 \cdot \bfR^Z \ot Y) = \\
& \one \cdot \delta_1 + \bfR^Y \cdot X + \bfR^X \cdot \delta_1X + \bfR^Y \cdot \delta_1Y + 2 \cdot \bfR^Z \cdot Y = \\
& \bfR^Z - \bfR^Z = 0.
\end{split}
\end{align}
\end{proof}

\subsubsection{Construction of  a representative cocycle for the even class}

\begin{proposition}
Let
\begin{align}
\begin{split}
& c := \one \ot X \ot Y - \one \ot Y \ot X - \bfR^X \ot XY \ot X - \bfR^X \ot Y \ot X^2 + \bfR^Y \ot XY \ot Y \\
& \hspace{1.2cm} + \bfR^Y \ot X \ot Y^2 - \bfR^Y \ot Y \ot X \in M_\d \ot \Uc^{\ot\;2}
\end{split}
\end{align}
and
\begin{align}
\begin{split}
& c'' := - \bfR^X \ot \d_1 \ot XY^2 + \frac{2}{3}\bfR^X \ot {\d_1}^2 \ot Y^3 +\frac{1}{3}\bfR^Y \ot \d_1 \ot Y^3 \\
& \hspace{1.2cm} - \frac{1}{4}\bfR^X \ot {\d_1}^2 \ot Y^2 - \frac{1}{2}\bfR^Y \ot \d_1 \ot Y^2 \in M_\d \ot \Fc \ot \Uc.
\end{split}
\end{align}
Then  $c+c'' \in {\rm{Tot}}^2(\Fc,\Uc,M_\d)$ is a Hochschild cocycle.
\end{proposition}

\begin{proof}
We start with the element
\begin{align}
\begin{split}
& c := \one \ot X \ot Y - \one \ot Y \ot X - \bfR^X \ot XY \ot X - \bfR^X \ot Y \ot X^2 + \bfR^Y \ot XY \ot Y \\
& + \bfR^Y \ot X \ot Y^2 - \bfR^Y \ot Y \ot X.
\end{split}
\end{align}

\nin It is immediate that $\hb(c) = 0$ . To be able to compute $\vb(c)$, we need to determine the following $\Fc$-coaction.
\begin{align}
\begin{split}
& \bfR^X \ot (XY)^{\pr{-1}}X^{\pr{-1}} \ot (XY)^{\pr{0}} \ot X^{\pr{0}} \\
& + \bfR^X \ot (X^2)^{\pr{-1}} \ot Y \ot (X^2)^{\pr{0}} - {\bfR^Y}\pr{0} \ot {\bfR^Y}\pr{-1}(XY)^{\pr{-1}} \ot (XY)^{\pr{0}} \ot Y \\
& - {\bfR^Y}\pr{0} \ot {\bfR^Y}\pr{-1}X^{\pr{-1}} \ot X^{\pr{0}} \ot Y^2 + {\bfR^Y}\pr{0} \ot {\bfR^Y}\pr{-1}X^{\pr{-1}} \ot Y \ot X^{\pr{0}}.
\end{split}
\end{align}

\nin Hence, observing
\begin{align}
\begin{split}
& \Db(X^2) = (X^2)^{\pr{0}} \ot (X^2)^{\pr{1}} = {X\ps{1}}^{\pr{0}}X^{\pr{0}} \ot {X\ps{1}}^{\pr{1}}(X\ps{2} \rhd X^{\pr{1}}) \\
& = X^2 \ot 1 + 2XY \ot \d_1 + X \ot \d_1 + Y^2 \ot {\d_1}^2 + \frac{1}{2}Y \ot {\d_1}^2,
\end{split}
\end{align}
and
\begin{equation}
\Db(XY) = (XY)^{\pr{0}} \ot (XY)^{\pr{1}} = X^{\pr{0}}Y \ot X^{\pr{1}}, \quad XY \mapsto XY \ot 1 + Y^2 \ot \d_1,
\end{equation}
we have
\begin{align}
\begin{split}
& \vb_0(c) = -\bfR^X \ot \d_1 \ot Y^2 \ot X - \bfR^X \ot \d_1 \ot XY \ot Y + \bfR^X \ot {\d_1}^2 \ot Y^2 \ot Y \\
& - 2\bfR^X \ot \d_1 \ot Y \ot XY - \bfR^X \ot \d_1 \ot Y \ot X + \bfR^X \ot {\d_1}^2 \ot Y \ot Y^2 \\
& + \frac{1}{2}\bfR^X \ot {\d_1}^2 \ot Y \ot Y - \bfR^X \ot \d_1 \ot XY \ot Y + \bfR^Y \ot \d_1 \ot Y^2 \ot Y \\
& + \bfR^X \ot {\d_1}^2 \ot Y^2 \ot Y - \bfR^X \ot \d_1 \ot X \ot Y^2 + \bfR^Y \ot \d_1 \ot Y \ot Y^2 \\
& + \bfR^X \ot {\d_1}^2 \ot Y \ot Y^2 + \bfR^X \ot \d_1 \ot Y \ot X - \bfR^Y \ot \d_1 \ot Y \ot Y - \bfR^X \ot {\d_1}^2 \ot Y \ot Y.
\end{split}
\end{align}

\nin It is now clear that
\begin{align}
\begin{split}
& \vb(c) = \hb\big(\bfR^X \ot \d_1 \ot XY^2 - \frac{2}{3}\bfR^X \ot {\d_1}^2 \ot Y^3 -\frac{1}{3}\bfR^Y \ot \d_1 \ot Y^3 \\
& \hspace{1.6cm} + \frac{1}{4}\bfR^X \ot {\d_1}^2 \ot Y^2 + \frac{1}{2}\bfR^Y \ot \d_1 \ot Y^2 \big).
\end{split}
\end{align}

\nin Therefore, for the element
\begin{align}
\begin{split}
& c'' := - \bfR^X \ot \d_1 \ot XY^2 + \frac{2}{3}\bfR^X \ot {\d_1}^2 \ot Y^3 +\frac{1}{3}\bfR^Y \ot \d_1 \ot Y^3 \\
& \hspace{1.2cm} - \frac{1}{4}\bfR^X \ot {\d_1}^2 \ot Y^2 - \frac{1}{2}\bfR^Y \ot \d_1 \ot Y^2,
\end{split}
\end{align}
we have $\hb(c'') + \vb(c) = 0$.

\nin Finally we observe that,
\begin{align}
\begin{split}
& \vb(c'') = \bfR^X \ot \d_1 \ot \d_1 \ot Y^3 - \frac{4}{3}\bfR^X \ot \d_1 \ot \d_1 \ot Y^3 + \frac{1}{3}\bfR^X \ot \d_1 \ot \d_1 \ot Y^3 \\
& \hspace{1.5cm} + \frac{1}{2}\bfR^X \ot \d_1 \ot \d_1 \ot Y^2 - \frac{1}{2}\bfR^X \ot \d_1 \ot \d_1 \ot Y^2 = 0.
\end{split}
\end{align}
\end{proof}

\begin{proposition}
The element $c+c'' \in {\rm{Tot}}^2(\Fc,\Uc,M_\d)$ vanishes under the  Connes boundary map.
\end{proposition}

\begin{proof}
As above, we start with
\begin{align}
\begin{split}
& c := \one \ot X \ot Y - \one \ot Y \ot X - \bfR^X \ot XY \ot X - \bfR^X \ot Y \ot X^2 + \bfR^Y \ot XY \ot Y \\
& + \bfR^Y \ot X \ot Y^2 - \bfR^Y \ot Y \ot X.
\end{split}
\end{align}

\nin To compute $\hB$, it suffices to consider the horizontal extra degeneracy operator $\overset{\ra}{\sigma}_{-1}:=\overset{\ra}{\sigma}_{1}\hta$. We have,
\begin{equation}
\overset{\ra}{\sigma}_{-1}(\one \ot X \ot Y) = \one \cdot X\ps{1} \ot S(X\ps{2})Y = -\one \ot XY,
\end{equation}
and
\begin{equation}
\overset{\ra}{\sigma}_{-1}(\one \ot Y \ot X) = \one \cdot Y\ps{1} \ot S(Y\ps{2})X = \one \ot X - \one \ot YX = -\one \ot XY.
\end{equation}
Therefore we proved that $\overset{\ra}{\sigma}_{-1}(\one \ot X \ot Y - \one \ot Y \ot X) = 0$. For the other  terms in $c \in M_\d \ot \Uc^{\ot\;2}$ we proceed by
\begin{align}
\begin{split}
& \overset{\ra}{\sigma}_{-1}(\bfR^X \ot XY \ot X) = \bfR^X \cdot X\ps{1}Y\ps{1} \ot S(Y\ps{2})S(X\ps{2})X \\
& = \bfR^X \cdot XY \ot X + \bfR^X \ot YX^2 - \bfR^X \cdot X \ot YX - \bfR^X \cdot Y \ot X^2 \\
& = \bfR^Y \ot XY + \bfR^X \ot YX^2 - 2\bfR^X \ot X^2,
\end{split}
\end{align}
\begin{align}
\begin{split}
& \overset{\ra}{\sigma}_{-1}(\bfR^Y \ot XY \ot Y) = \bfR^Y \cdot XY \ot Y + \bfR^Y \ot YXY - \bfR^Y \cdot X \ot Y^2 - \bfR^Y \cdot Y \ot XY \\
& = \bfR^Y \ot XY^2 + \bfR^Z \ot Y^2,
\end{split}
\end{align}
\begin{equation}
\overset{\ra}{\sigma}_{-1}(\bfR^X \ot Y \ot X^2) = 2\bfR^X \ot X^2 - \bfR^X \ot YX^2,
\end{equation}
\begin{equation}
\overset{\ra}{\sigma}_{-1}(\bfR^Y \ot X \ot Y^2) = -\bfR^Z \ot Y^2 - \bfR^Y \ot XY^2,
\end{equation}
and finally
\begin{equation}
\overset{\ra}{\sigma}_{-1}(\bfR^Y \ot Y \ot X) = \bfR^Y \ot X - \bfR^Y \ot YX = - \bfR^Y \ot XY.
\end{equation}

\nin This way we prove
\begin{equation}
\overset{\ra}{\sigma}_{-1}(- \bfR^X \ot XY \ot X - \bfR^X \ot Y \ot X^2 + \bfR^Y \ot XY \ot Y + \bfR^Y \ot X \ot Y^2 - \bfR^Y \ot Y \ot X) = 0.
\end{equation}

\nin Hence we conclude that $\hB(c)=0$.

\medskip

\nin Since the action of $\d_1$ on $F_0M_\d = \Cb\Big\langle \bfR^X, \bfR^Y, \bfR^Z \Big\rangle$ is trivial, we have $\vB(c'')=0$. The horizontal counterpart $\hB(c'') = 0$ follows from the following observations. First we notice
\begin{equation}
(\bfR^X \ot {\d_1}^2) \cdot Y = \bfR^X \cdot Y\ps{1} \ot S(Y\ps{2}) \rhd {\d_1}^2 = 0,
\end{equation}
and secondly,
\begin{equation}
(\bfR^Y \ot \d_1) \cdot Y = \bfR^Y \cdot Y\ps{1} \ot S(Y\ps{2}) \rhd \d_1 = 0.
\end{equation}
\end{proof}

\nin Next, we send the element $c+c'' \in {\rm Tot}^2(\Fc, \Uc,M_\d)$ to the cyclic complex $C(\Hc,M_\d)$. As before, on the first step we use the Alexander-Whitney map to land in the diagonal complex $\FZ(\Hc,\Fc,M_\d)$. To this end, we have
\begin{align}
\begin{split}
& AW_{0,2}(c) = \; \vd_0\vd_0(c) \\
& = \one \ot 1 \ot 1 \ot X \ot Y - \one \ot 1 \ot 1 \ot Y \ot X - \bfR^X \ot 1 \ot 1 \ot XY \ot X \\
&  - \bfR^X\ot 1 \ot 1 \ot Y \ot X^2 + \bfR^Y \ot 1 \ot 1 \ot XY \ot Y + \bfR^Y \ot 1 \ot 1 \ot X \ot Y^2 \\
& - \bfR^Y \ot 1 \ot 1 \ot Y \ot X,
\end{split}
\end{align}
and
\begin{align}
\begin{split}
& AW_{1,1}(c) =\; \vd_0\hd_1(c) \\
& = - \bfR^X \ot 1 \ot \d_1 \ot XY^2 \ot 1 + \frac{2}{3}\bfR^X \ot 1 \ot {\d_1}^2 \ot Y^3 \ot 1 +\frac{1}{3}\bfR^Y \ot 1 \ot \d_1 \ot Y^3 \ot 1 \\
& - \frac{1}{4}\bfR^X \ot 1 \ot {\d_1}^2 \ot Y^2 \ot 1 - \frac{1}{2}\bfR^Y \ot 1 \ot \d_1 \ot Y^2 \ot 1.
\end{split}
\end{align}

\nin As a result, we obtain the element
\begin{align}
\begin{split}
& c^{{\rm even}}_{{\rm diag}} = \one \ot 1 \ot 1 \ot X \ot Y - \one \ot 1 \ot 1 \ot Y \ot X - \bfR^X \ot 1 \ot 1 \ot XY \ot X \\
& - \bfR^X\ot 1 \ot 1 \ot Y \ot X^2 + \bfR^Y \ot 1 \ot 1 \ot XY \ot Y + \bfR^Y \ot 1 \ot 1 \ot X \ot Y^2 \\
& - \bfR^Y \ot 1 \ot 1 \ot Y \ot X - \bfR^X \ot 1 \ot \d_1 \ot XY^2 \ot 1 + \frac{2}{3}\bfR^X \ot 1 \ot {\d_1}^2 \ot Y^3 \ot 1 \\
& +\frac{1}{3}\bfR^Y \ot 1 \ot \d_1 \ot Y^3 \ot 1 - \frac{1}{4}\bfR^X \ot 1 \ot {\d_1}^2 \ot Y^2 \ot 1 - \frac{1}{2}\bfR^Y \ot 1 \ot \d_1 \ot Y^2 \ot 1.
\end{split}
\end{align}

\nin On the second step, we use the map \eqref{PSI-1} to obtain
\begin{align}\label{even-cocycle}
\begin{split}
& c^{{\rm even}}:= \Psi\big(c^{{\rm even}}_{{\rm diag}}\big) = \one \ot X \ot Y - \one \ot Y \ot X + \one \ot Y \ot \d_1Y - \bfR^X \ot XY \ot X \\
& - \bfR^X \ot Y^2 \ot \d_1X - \bfR^X \ot Y \ot X^2 + \bfR^Y \ot XY \ot Y + \bfR^Y \ot Y^2 \ot \d_1Y \\
& + \bfR^Y \ot X \ot Y^2 + \bfR^Y \ot Y \ot \d_1Y^2 - \bfR^Y \ot Y \ot X - \bfR^X \ot XY^2 \ot \d_1 \\
& - \frac{1}{3}\bfR^X \ot Y^3 \ot {\d_1}^2 + \frac{1}{3} \bfR^Y \ot Y^3 \ot \d_1 - \frac{1}{4} \bfR^X \ot Y^2 \ot {\d_1}^2 - \frac{1}{2} \bfR^Y \ot Y^2 \ot \d_1,\\
\end{split}
\end{align}
in  $ C^2(\Hc, M_\d)$.
\begin{proposition}
The element  $c^{\rm even}$ defined in \eqref{even-cocycle} is a Hochschild cocycle.
\end{proposition}

\begin{proof}
We first recall that
\begin{align}
\begin{split}
& b(\one \ot X \ot Y - \one \ot Y \ot X + \one \ot Y \ot \d_1Y) = \\
& -\bfR^X \ot X \ot Y \ot X - \bfR^Y \ot X \ot Y \ot Y + \bfR^X \ot Y \ot X \ot X + \bfR^Y \ot Y \ot X \ot Y \\
& - \bfR^X \ot Y \ot \d_1Y \ot X - \bfR^Y \ot Y \ot \d_1Y \ot Y.
\end{split}
\end{align}

\nin Next we compute
\begin{align}
\begin{split}
& b(\bfR^X \ot XY \ot X) = - \bfR^X \ot X \ot Y \ot X - \bfR^X \ot Y \ot X \ot X - \bfR^X \ot Y^2 \ot \d_1 \ot X \\
& - \bfR^X \ot Y \ot \d_1Y \ot X + \bfR^X \ot XY \ot Y \ot \d_1, \\[.2cm]
& b(\bfR^X \ot Y^2 \ot \d_1X) = - 2\bfR^X \ot Y \ot Y \ot \d_1X + \bfR^X \ot Y^2 \ot \d_1 \ot X + \bfR^X \ot Y^2 \ot X \ot \d_1 \\
& + \bfR^X \ot Y^2 \ot \d_1Y \ot \d_1 + \bfR^X \ot Y^2 \ot Y \ot {\d_1}^2, \\[.2cm]
& b(\bfR^X \ot Y \ot X^2) = 2 \bfR^X \ot Y \ot X \ot X + \bfR^X \ot Y \ot XY \ot \d_1 + \bfR^X \ot Y \ot Y \ot X\d_1 \\
& + \bfR^X \ot Y \ot YX \ot \d_1 + \bfR^X \ot Y \ot Y \ot \d_1X + \bfR^X \ot Y \ot Y^2 \ot {\d_1}^2, \\[.2cm]
& b(\bfR^Y \ot XY \ot Y) = - \bfR^Y \ot X \ot Y \ot Y - \bfR^Y \ot Y \ot X \ot Y - \bfR^Y \ot Y^2 \ot \d_1 \ot Y \\
& - \bfR^Y \ot Y \ot \d_1Y \ot Y - \bfR^X \ot XY \ot Y \ot \d_1, \\[.2cm]
& b(\bfR^Y \ot Y^2 \ot \d_1Y) = -2 \bfR^Y \ot Y \ot Y \ot \d_1Y + \bfR^Y \ot Y^2 \ot \d_1 \ot Y + \bfR^Y \ot Y^2 \ot Y \ot \d_1 \\
& - \bfR^X \ot Y^2 \ot \d_1Y \ot \d_1,
\end{split}
\end{align}
as well as
\begin{align}
\begin{split}
& b(\bfR^Y \ot X \ot Y^2) = - \bfR^Y \ot Y \ot \d_1 \ot Y^2 + 2\bfR^Y \ot X \ot Y \ot Y - \bfR^X \ot X \ot Y^2 \ot \d_1, \\[.2cm]
& b(\bfR^Y \ot Y \ot \d_1Y^2) = \bfR^Y \ot Y \ot \d_1 \ot Y^2 + \bfR^Y \ot Y \ot Y^2 \ot \d_1 + 2\bfR^Y \ot Y \ot \d_1Y \ot Y \\
& + 2\bfR^Y \ot Y \ot Y \ot \d_1Y - \bfR^X \ot Y \ot \d_1Y^2 \ot \d_1, \\[.2cm]
& b(\bfR^Y \ot Y \ot X) = \bfR^Y \ot Y \ot Y \ot \d_1 - \bfR^X \ot Y \ot X \ot \d_1, \\[.2cm]
& b(\bfR^X \ot XY^2 \ot \d_1) = - \bfR^X \ot X \ot Y^2 \ot \d_1 - \bfR^X \ot Y^2 \ot X \ot \d_1 - 2 \bfR^X \ot XY \ot Y \ot \d_1 \\
& -2 \bfR^X \ot Y \ot XY \ot \d_1 - 2\bfR^X \ot Y^2 \ot \d_1Y \ot \d_1 - \bfR^X \ot Y^3 \ot \d_1 \ot \d_1 \\
& - \bfR^X \ot Y \ot \d_1Y^2 \ot \d_1, \\[.2cm]
& b(\bfR^X \ot Y^3 \ot {\d_1}^2) = -3 \bfR^X \ot Y^2 \ot Y \ot {\d_1}^2 - 3\bfR^X \ot Y \ot Y^2 \ot {\d_1}^2 + 2 \bfR^X \ot Y^3 \ot \d_1 \ot \d_1, \\[.2cm]
& b(\bfR^Y \ot Y^3 \ot \d_1) = -3 \bfR^Y \ot Y^2 \ot Y \ot \d_1 - 3 \bfR^Y \ot Y \ot Y^2 \ot \d_1 - \bfR^X \ot Y^3 \ot \d_1 \ot \d_1, \\[.2cm]
& b(\bfR^X \ot Y^2 \ot {\d_1}^2) = -2 \bfR^X \ot Y \ot Y \ot {\d_1}^2 + 2 \bfR^X \ot Y^2 \ot \d_1 \ot \d_1, \\[.2cm]
& b(\bfR^Y \ot Y^2 \ot \d_1) = -2 \bfR^Y \ot Y \ot Y \ot \d_1 - \bfR^X \ot Y^2 \ot \d_1 \ot \d_1
\end{split}
\end{align}
Summing up, we get the result.
\end{proof}

\begin{proposition}
The Hochschild cocycle $c^{\rm even}$ defined in  \eqref{even-cocycle} vanishes under the  Connes boundary map.
\end{proposition}

\begin{proof}
We will first prove that the extra degeneracy operator $\s_{-1}$ vanishes on
\begin{align}
\begin{split}
& c := \one \ot X \ot Y - \one \ot Y \ot X + \one \ot Y \ot \d_1Y - \bfR^X \ot XY \ot X \\
& - \bfR^X \ot Y^2 \ot \d_1X - \bfR^X \ot Y \ot X^2 + \bfR^Y \ot XY \ot Y + \bfR^Y \ot Y^2 \ot \d_1Y \\
& + \bfR^Y \ot X \ot Y^2 + \bfR^Y \ot Y \ot \d_1Y^2 - \bfR^Y \ot Y \ot X \in C^2(\Hc, M_\d).
\end{split}
\end{align}
We observe that
\begin{align}
\begin{split}
&\s_{-1}(\one \ot X \ot Y - \one \ot Y \ot X + \one \ot Y \ot \d_1Y) = 0,\\[.2cm]
&\s_{-1}(\bfR^X \ot XY \ot X) = \bfR^Y \ot XY + \bfR^X \ot X^2Y - \bfR^X \ot \d_1XY^2,\\[.2cm]
&\s_{-1}(\bfR^X \ot Y^2 \ot \d_1X) = \bfR^X \ot \d_1XY^2,\\[.2cm]
&\s_{-1}(\bfR^X \ot Y \ot X^2) = -\bfR^X \ot X^2Y,\\[.2cm]
&\s_{-1}(\bfR^Y \ot XY \ot Y) = \bfR^Z \ot Y^2 + \bfR^Y \ot XY^2 - \bfR^Y \ot \d_1Y^3,\\[.2cm]
&\s_{-1}(\bfR^Y \ot Y^2 \ot \d_1Y) = \bfR^Y \ot \d_1Y^3,\\[.2cm]
&\s_{-1}(\bfR^Y \ot X \ot Y^2) = -\bfR^Z \ot Y^2 - \bfR^Y \ot XY^2 + \bfR^Y \ot \d_1Y^3,\\[.2cm]
&\s_{-1}(\bfR^Y \ot Y \ot \d_1Y^2) = - \bfR^Y \ot \d_1Y^3,\\[.2cm]
&\s_{-1}(\bfR^Y \ot Y \ot X) = -\bfR^Y \ot XY.
\end{split}
\end{align}

\nin As a result, we obtain $\s_{-1}(c) = 0$. On the second step, we prove that Connes boundary map $B$ vanishes on
\begin{align}
\begin{split}
& c'' := - \bfR^X \ot XY^2 \ot \d_1 - \frac{1}{3}\bfR^X \ot Y^3 \ot {\d_1}^2 + \frac{1}{3} \bfR^Y \ot Y^3 \ot \d_1 \\
& - \frac{1}{4} \bfR^X \ot Y^2 \ot {\d_1}^2 - \frac{1}{2} \bfR^Y \ot Y^2 \ot \d_1 \in C^2(\Hc, M_\d).
\end{split}
\end{align}

\nin Indeed, as in this case $B = (\Id - \tau) \circ \s_{-1}$, it suffices to observe that
\begin{align}
\begin{split}
&\s_{-1}(\bfR^X \ot XY^2 \ot \d_1)\\
& = - \bfR^Y \ot \d_1Y^2 - \bfR^X \ot \d_1XY^2 - \frac{1}{2}\bfR^X \ot
 {\d_1}^2Y^2 + \bfR^X \ot {\d_1}^2Y^3,\\[.2cm]
&\s_{-1}(\bfR^X \ot Y^3 \ot {\d_1}^2) = - \bfR^X \ot {\d_1}^2Y^3,\\[.2cm]
&\s_{-1}(\bfR^Y \ot Y^3 \ot \d_1) = - \bfR^Y \ot \d_1Y^3,\\[.2cm]
&\s_{-1}(\bfR^X \ot Y^2 \ot {\d_1}^2) = \bfR^X \ot {\d_1}^2Y^2,\\[.2cm]
&\s_{-1}(\bfR^Y \ot Y^2 \ot \d_1) = \bfR^Y \ot \d_1Y^2,\\[.2cm]
\end{split}
\end{align}
together with
\begin{align}
\begin{split}
&\tau(\bfR^Y \ot \d_1Y^2) = - \bfR^Y \ot \d_1Y^2 - \bfR^X \ot {\d_1}^2Y^2,\\[.2cm]
&\tau(\bfR^X \ot {\d_1}^2Y^2) = \bfR^X \ot {\d_1}^2Y^2,\\[.2cm]
&\tau(\bfR^Y \ot \d_1Y^3) = \bfR^Y \ot \d_1Y^3 + \bfR^X \ot {\d_1}^2Y^3,\\[.2cm]
&\tau(\bfR^X \ot {\d_1}^2Y) = - \bfR^Y \ot {\d_1}^2Y,\\[.2cm]
&\tau(\bfR^X \ot {\d_1}^2Y^3) = - \bfR^Y \ot {\d_1}^2Y^3,\\[.2cm]
&\tau(\bfR^X \ot \d_1XY^2) = \bfR^Y \ot \d_1Y^2 + \bfR^X \ot \d_1XY^2 + \frac{1}{2}\bfR^X \ot {\d_1}^2Y^2 - \bfR^X \ot {\d_1}^2Y^3.
\end{split}
\end{align}
\end{proof}

\nin We summarize our results in this section by the following theorem.

\begin{theorem}
The odd and even periodic Hopf cyclic cohomology of the Schwarzian Hopf algebra $\Hc_{\rm 1S}\cop$ with coefficients in the 4-dimensional SAYD module $M_\d = S({sl_2}^\ast)\nsb{2}$ are given by
\begin{equation}
HP^{\rm odd}(\Hc_{\rm 1S}\cop,M_\d) = \Cb \Big\langle \one \ot \delta_1 + \bfR^Y \ot X + \bfR^X \ot \delta_1X + \bfR^Y \ot \delta_1Y + 2 \bfR^Z \ot Y \Big\rangle,
\end{equation}
and
\begin{align}
\begin{split}
& HP^{\rm even}(\Hc_{\rm 1S}\cop,M_\d) = \Cb \Big\langle \one \ot X \ot Y - \one \ot Y \ot X + \one \ot Y \ot \d_1Y - \bfR^X \ot XY \ot X \\
& - \bfR^X \ot Y^2 \ot \d_1X - \bfR^X \ot Y \ot X^2 + \bfR^Y \ot XY \ot Y + \bfR^Y \ot Y^2 \ot \d_1Y \\
& + \bfR^Y \ot X \ot Y^2 + \bfR^Y \ot Y \ot \d_1Y^2 - \bfR^Y \ot Y \ot X - \bfR^X \ot XY^2 \ot \d_1 \\
& - \frac{1}{3}\bfR^X \ot Y^3 \ot {\d_1}^2 + \frac{1}{3} \bfR^Y \ot Y^3 \ot \d_1 - \frac{1}{4} \bfR^X \ot Y^2 \ot {\d_1}^2 - \frac{1}{2} \bfR^Y \ot Y^2 \ot \d_1 \Big\rangle.
\end{split}
\end{align}
\end{theorem}

\bibliographystyle{amsplain}
\bibliography{Rangipour-Sutlu-References}{}

\end{document}